\documentclass[11pt]{amsart}

\pdfoutput=1

\usepackage{amsmath,amssymb,latexsym}
\usepackage[T1]{fontenc}
\usepackage{enumerate}
\usepackage{array}
\usepackage{mathtools}
\usepackage{url}
\usepackage{xcolor}
\usepackage{tikz}

\usepackage[all]{xy}
\usepackage{pb-diagram,pb-xy}
\usepackage[a4paper,top=3cm, bottom=3cm, inner=3.5cm,outer=2.7cm,nofoot,headsep=1.5cm]{geometry}
\def\Ind#1#2{#1\setbox0=\hbox{$#1x$}\kern\wd0\hbox to 0pt{\hss$#1\mid$\hss}
\lower.9\ht0\hbox to 0pt{\hss$#1\smile$\hss}\kern\wd0}
\def\Notind#1#2{#1\setbox0=\hbox{$#1x$}\kern\wd0\hbox to 0pt{\mathchardef
\nn="3236\hss$#1\nn$\kern1.4\wd0\hss}\hbox to 0pt{\hss$#1\mid$\hss}\lower.9\ht0
\hbox to 0pt{\hss$#1\smile$\hss}\kern\wd0}
\def\indi{\mathop{\mathpalette\Ind{}}}

\theoremstyle{plain}
\newtheorem{theorem}{Theorem}[section]
\newtheorem*{theorem*}{Theorem}
\newtheorem{prop}[theorem]{Proposition}
\newtheorem{fact}[theorem]{Fact}
\newtheorem*{fact*}{Fact}
\newtheorem{lemma}[theorem]{Lemma}
\newtheorem{cor}[theorem]{Corollary}
\newtheorem{claim}[theorem]{Claim}

\theoremstyle{definition}
\newtheorem{defn}[theorem]{Definition}
\newtheorem{remark}[theorem]{Remark}

\newtheorem{problem}[theorem]{Problem}

\newtheorem{conj}[theorem]{Conjecture}

\newcommand{\M}{\mathcal{M}}
\newcommand{\Span}{\operatorname{Span}}
\newcommand{\ded}{\operatorname{ded}}

\newcommand{\qftp}{\operatorname{qftp}}

\newcommand{\RCF}{\operatorname{RCF}}
\newcommand{\ACF}{\operatorname{ACF}}
\newcommand{\Aut}{\operatorname{Aut}}
\newcommand{\ring}{\operatorname{ring}}

\newcommand{\lex}{\operatorname{lex}}
\newcommand{\IP}{\operatorname{IP}}
\newcommand{\opg}{\operatorname{opg}}
\newcommand{\op}{\operatorname{op}}
\newcommand{\Th}{\operatorname{Th}}
\newcommand{\VC}{\operatorname{VC}}
\newcommand{\tp}{\operatorname{tp}}

\newcommand{\ten}{\otimes}
\newcommand{\Ten}{\bigotimes}
\newcommand{\Bf}{$\langle - , - \rangle$\ }
\newcommand{\Bi}{$\langle - , -  \rangle_2$\ }
\newcommand{\Bli}{\langle - , - \rangle_2\ }
\newcommand{\Bn}{$\langle - ,\ldots , - \rangle_n$\ }
\newcommand{\Cn}{$[-,\ldots , -]_n$\ }
\newcommand{\Bln}{\langle -,\dots , - \rangle_n }
\newcommand{\Cln}{[ - ,\ldots , -]_n }
\newcommand{\Alt}{\operatorname{Alt}}
\newcommand{\Sym}{\operatorname{Sym}}

\newcommand{\id}{\operatorname{id}}

\newcommand{\C}{\mathbb{M}}
\newcommand{\VS}{\operatorname{VS}}
\newcommand{\ind}{\operatorname{\indi}}

\newcommand{\set}[1]{\left[#1\right]}

\newcommand{\dist}{\operatorname{dist}}

\DeclareMathOperator{\sign}{sign}
\numberwithin{equation}{section}

\title[On n-dependent groups and fields III]{On n-dependent groups and fields III. Multilinear forms and  invariant connected components}
\author{Artem Chernikov and Nadja Hempel}

\begin{document}

\begin{abstract}
We develop some model theory of multilinear forms, generalizing Granger's work in the bilinear case. In particular, after proving a quantifier elimination result, we show that for an NIP field $K$, the theory of infinite-dimensional non-degenerate alternating $n$-linear spaces over $K$ is strictly $n$-dependent, and is NSOP$_1$ if $K$ is. These results rely on a new Composition Lemma for functions of arbitrary arity and NIP relations (which in turn relies on certain higher-arity generalizations of the Sauer--Shelah lemma). We also study the invariant connected components $G^{\infty}$ in  $n$-dependent groups, demonstrating their relative absoluteness.
\end{abstract}
\maketitle

\tableofcontents

\section{Introduction}

In this article we study groups, fields and related structures satisfying
a model-theoretic tameness condition called \emph{$n$-dependence}, for $n \in \mathbb{N}$, continuing \cite{hempel2016n, chernikov2019mekler, chernikov2021n}; we further develop the pure theory of $n$-dependent structures and some related combinatorics, contributing to the emerging \emph{higher classification theory}.
 The class of $n$-dependent theories was introduced by Shelah in \cite{shelah2014strongly}, with the $1$-dependent (or just dependent) case corresponding to the class of \emph{NIP
theories} that has attracted a lot of attention recently. Roughly
speaking, $n$-dependence of a theory guarantees that the edge relation of an infinite generic $(n + 1)$-hypergraph is not definable in its models (see Section \ref{sec: n-dep and gen indisc}). For $n \geq 2$, we say
that a theory is strictly $n$-dependent if it is $n$-dependent, but not $(n-1)$-dependent.
Basic properties of $n$-dependent theories are investigated in \cite{chernikov2019n}, where in particular the numerical parameter \emph{$\VC_n$-dimension} for families of subsets of $n$-fold products (whose finiteness characterizes $n$-dependence in the same way as finiteness of $\VC$-dimension characterizes NIP \cite{laskowski1992vapnik}) is defined and investigated quantitatively, including a higher-arity version of the Sauer--Shelah lemma for $\VC_n$-dimension \cite[Proposition 3.9]{chernikov2019n} (where the case $n=1$ corresponds to the usual Sauer--Shelah lemma and VC-dimension; see Section \ref{sec: fin type count}).
Since then, the study of VC$_n$-dimension, or \emph{higher VC-theory}, has found deep connections to (arithmetic) hypergraph combinatorics \cite{chernikov2020hypergraph, terry2021irregular, terry2023improved, terry2024growth, chernikov2024perfect, terry2025structure, terry2025quadratic, gishboliner2025regularity, sheats2025linear}.

In Section \ref{sec: multilinear forms}, we develop some basic model theory of multilinear forms,  viewed as structures with two sorts for the vector space and for the field, generalizing (and correcting) some influential work of Granger \cite{granger1999stability} in the case $n=2$.
 In Section \ref{sec: non-degen} we consider possible generalizations of \emph{non-degeneracy} from bilinear forms to $n$-linear forms, for an arbitrary $n \in \mathbb{N}_{\geq 2}$. Several notions of non-degeneracy for multilinear forms appear in the literature, see e.g.~\cite{gelfand2009discriminants, hitchin2001stable, kazhdan2020properties}, but differ from the one that we consider here and that works well in the alternating case. We consider multilinear spaces as two-sorted structures $(V,K)$ in the language $\mathcal{L}^{K}_{\theta,f}$ containing the field language on $K$, functions for vector addition on $V$ and scalar multiplication on $K \times V$, a function $\langle -, \ldots, - \rangle_n : V^n \to K$ for the $n$-linear form, an $m$-ary relation symbol $\theta_m(v_1, \ldots, v_m)$ expressing that $v_1, \ldots, v_m \in V$ are $K$-linearly independent for each $m$, and, for each $p$ and $i \in [p]$, a function $ f_i^{p}(v; v_1, \dots, v_p) = \lambda_i$ if $v_1, \ldots, v_p$ are linearly independent and $v = \sum_{j=1}^{p} \lambda_j v_j$ for some $\lambda_1,\ldots,\lambda_p \in K$ (see Definition \ref{def: language of bilin forms}). Our main model-theoretic result is the following:
\begin{theorem*}
		For any $n \geq 2$ and field $K$, the theory $\prescript{}{\Alt}T^K_n$ of infinite-dimensional alternating non-degenerate $n$-linear spaces over $K$ has quantifier elimination in the language $\mathcal{L}^{K}_{\theta,f}$ and is complete (Theorem \ref{thm: QE for multilinear forms}). If $K$ is finite, then $\prescript{}{\Alt}T^K_n$ is $\omega$-categorical (Remark \ref{rem: VS complete and omega-cat}); and it is 
 pseudofinite if and only if $n = 2$ (Proposition \ref{prop: multilinear not pseudofinite}).
\end{theorem*} 
\noindent In the bilinear case, Granger \cite[Corollary 9.2.3]{granger1999stability} incorrectly claimed quantifier elimination in a smaller language without the coordinate functions $f_i^p$. A corrected language was suggested by Dugald Macpherson who pointed out the error, by analogy with Delon's quantifier elimination in pairs of algebraically closed fields \cite{delon2012elimination} (see also the discussion in \cite{met2023sets} in the paragraph after Definition 2.7).
While this paper was in preparation, a more general version of the $\omega$-categorical case was developed using Fra\"iss\'e constructions in \cite{harman2024ultrahomogeneous} (see also the related papers \cite{bik2022isogeny, harman2024tensor, neretin2023oligomorphic}), and the bilinear case of Theorem \ref{thm: QE for multilinear forms} was also considered in \cite{abd2023higher}.

Bilinear forms have played an important role in the study of $n$-dependent theories to date, and in fact until recently all known
``algebraic'' examples of strictly $n$-dependent theories with $n \geq 2$ looked like bilinear forms over \emph{finite} fields. Namely, it was observed in  \cite{hempel2016n}  that the theory of a bilinear form on an infinite-dimensional vector
space over a finite field is strictly $2$-dependent; smoothly approximable structures \cite{cherlin2003finite} are $2$-dependent (see \cite[Example 2.2(4)]{chernikov2019n})
and coordinatizable via bilinear forms over finite fields; and the strictly $n$-dependent
pure groups constructed in \cite{chernikov2019mekler} using Mekler's construction are essentially of this form as
well, using Baudisch's interpretation of Mekler's construction in alternating bilinear maps \cite{baudisch2002mekler}. In the previous paper, we demonstrated that finite fields can be replaced by arbitrary NIP fields for $n=2$:

\begin{fact}\label{fac: bilin 2-dep}
\cite[Theorem 6.3]{chernikov2021n}
Let $T$ be a theory of bilinear forms on infinite-dimensional vector spaces over $K$ eliminating quantifiers in the language $\mathcal{L}^{K}_{\theta,f}$. Then $T$ is $2$-dependent if and only if $\Th(K)$ is dependent.
\end{fact}
\noindent The two main ingredients were Granger's quantifier elimination for bilinear forms and what we called the \emph{Composition Lemma} \cite[Theorem 5.12]{chernikov2021n}, showing that the composition of a relation definable in an NIP structure with \emph{arbitrary} binary functions remains $2$-dependent.

In Section \ref{sec: Composition lemma} we obtain a generalization of the Composition Lemma for functions of arbitrary arity:
\begin{theorem*}[Composition Lemma, Theorem \ref{Composition Lemma}]
	Let $\mathcal{M}=(M,\ldots)$ be an $\mathcal{L}'$-structure such that its reduct to a sublanguage $\mathcal{L} \subseteq \mathcal{L}'$ is NIP. Let $d,k \in \mathbb{N}$, let $\varphi(x_1, \ldots, x_d)$ be an $\mathcal{L}$-formula, and let $y_0, \ldots, y_k$ be arbitrary tuples of variables. For each $1 \leq t \leq d$, let $0 \leq i^{t}_{1}, \ldots, i^{t}_k \leq k$ be arbitrary, and let $f_t: M^{y_{i^t_{1}}} \times \ldots \times M^{y_{i^{t}_{k}}} \to M^{x_t}$ be an arbitrary $\mathcal{L}'$-definable $k$-ary function. Then the $\mathcal{L}'$-formula 
	$$\psi\left( y_0; y_1, \ldots, y_k \right) := \varphi \left(f_1 \left(y_{i^1_1}, \ldots, y_{i^1_k} \right), \ldots, f_d \left(y_{i^d_1}, \ldots, y_{i^d_k} \right) \right)$$
	 is $k$-dependent in $\mathcal{M}$.
\end{theorem*}
\noindent Our proof of this theorem for $k = 2$ in \cite{chernikov2021n} utilized an infinite type-counting criterion for $2$-dependence involving the function $\ded(\kappa)$ and types realized \emph{cofinally often} over mutually indiscernible sequences (\cite[Section 5.1]{chernikov2021n}), combined with a set-theoretic absoluteness argument. There does not seem to be a natural way to generalize it to higher $k$.
Here we give a purely combinatorial proof for arbitrary $k\geq1$. First we provide some preliminaries on (finitary versions of) shrinking of indiscernible sequences and generalize it to indiscernible arrays in NIP theories in Section \ref{sec: array shrink NIP} (utilizing uniform definability of types over finite sets in NIP theories \cite{chernikov2015externally}). In Section \ref{sec: n-dep and gen indisc} we recall basic properties of $n$-dependent formulas and a characterization of $n$-dependence in terms of generalized indiscernibles from \cite{chernikov2019n}. In Section \ref{sec: fin type count} we provide a finitary type-counting criterion for $n$-dependence using the aforementioned generalization of the Sauer--Shelah lemma for $\VC_n$-dimension from \cite{chernikov2019n}, modeled on the infinitary type-counting criterion for $2$-dependence from \cite{chernikov2021n}:
\begin{theorem*}[See Proposition \ref{prop: type count} for the precise statement.]
A formula $\varphi(x; y_1, \ldots, y_k)$ is $k$-dependent if and only if given any tuple $b \in \mathbb{M}^x$ and finite mutually indiscernible sequences $I_1, \ldots, I_{k-1}, I_k$ with $I_1, \ldots, I_{k-1}$ of length $n$ and $I_k$ much longer, there is a large interval $J$ of $I_k$ with fewer than maximal possible number (more precisely, $\leq 2^{n^{k - 1 - \varepsilon}}$) of $\varphi(y_k)$-types over $b, I_1, \ldots, I_{k-1}$ realized in $J$.
\end{theorem*}

 In Section \ref{sec: array shattering} we prove an \emph{Array Shattering Lemma} for $k$-ary arrays in NIP structures, which we view as a natural generalization of the Sauer--Shelah lemma from binary to higher-arity relations definable in NIP structures. We only state it here in the case $k=2$, and refer to Lemma \ref{lem: comp lemma induction}(2) for the precise statement of the general case (and its ``dynamic'' version with an additional moving parameter used in the inductive proof, Lemma \ref{lem: comp lemma induction}(1); see also illustrations in the proof of Lemma \ref{lem: comp lemma induction}).

\begin{theorem*}[Array Shattering Lemma, Lemma \ref{lem: comp lemma induction}(2)  for $k=2$]
	Assume that an $\mathcal{L}$-structure $\mathcal{M}$ is NIP.	
%
For every formula $\varphi(y_1, y_2, y_3) \in \mathcal{L}$ there exist some $n' \in \mathbb{N}$ and $\varepsilon \in \mathbb{R}_{>0}$ satisfying the following. Let $n \geq n'$ and
		$\bar{\delta}=(\delta_{i_1,i_2} : (i_1,i_2)\in [n]^{2})$ be an array of tuples with $\delta_{i_1,i_2} \in \mathcal{M}^{y_3}$. For any sequences $\bar{\zeta}^1 = (\zeta^1_{i} : i \in [n])$ with $\zeta^1_{i} \in \mathcal{M}^{y_1}$ and $\bar{\zeta}^2 = (\zeta^2_{i} : i \in [n])$ with $\zeta^2_{i} \in \mathcal{M}^{y_2}$, define 
		 $$S^{\varphi,\bar{\delta}}_{\bar{\zeta}^1, \bar{\zeta}^2} := \left\{(i_1,i_2) \in [n]^{2}: \mathcal{M} \models \varphi \left(\zeta^1_{i_2},  \zeta^2_{i_1}, \delta_{i_1,i_2} \right) \right\}.$$
		  Then the family of sets 
		  $$\mathcal{F}^{\varphi,\bar{\delta}} = \left\{S^{\varphi,\bar{\delta}}_{\bar{\zeta}^1, \bar{\zeta}^2} : \bar{\zeta}^1,  \bar{\zeta}^2 \textrm{ arbitrary sequences} \right\}$$  has cardinality $<2^{n^{2 - \varepsilon}}$.
\end{theorem*}

\noindent For $k=3$ we instead consider subsets of a hypercube $\left(\delta_{i_1,i_2,i_3} : i_1,i_2,i_3 \in [n] \right)$ of the form 
$$\left \{ (i_1,i_2,i_3) \in [n]^3 : \models \varphi \left(  \zeta^1_{i_2,i_3}, \zeta^2_{i_1, i_3}, \zeta^3_{i_1,i_2}, \delta_{i_1,i_2,i_3} \right) \right \}$$
that can be obtained by varying two-dimensional arrays $\bar{\zeta}^1 = (\zeta^1_{i_2,i_3} : (i_2,i_3) \in [n]^2)$, $\bar{\zeta}^2 = (\zeta^2_{i_1,i_3} : (i_1,i_3) \in [n]^2)$ and $\bar{\zeta}^3 = (\zeta^3_{i_1,i_2} : (i_1,i_2) \in [n]^2)$; etc.

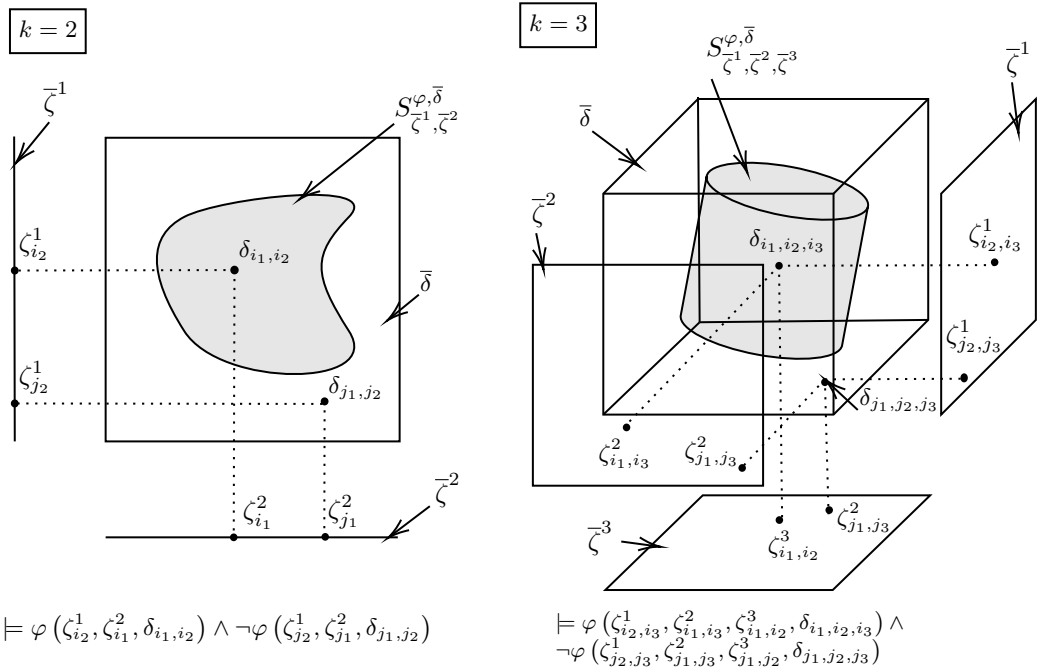
\tikzset{every picture/.style={line width=0.75pt}} 

\begin{tikzpicture}[x=0.75pt,y=0.75pt,yscale=-1,xscale=1]

\draw   (50,65.45) -- (197.49,65.45) -- (197.49,217.78) -- (50,217.78) -- cycle ;
\draw    (4.1,64.82) -- (4.41,217.78) ;
\draw    (50,265.9) -- (196.57,265.9) ;
\draw  [fill={rgb, 255:red, 0; green, 0; blue, 0 }  ,fill opacity=0.1 ] (89.48,105.13) .. controls (107.83,95.57) and (190.45,86.01) .. (172.09,105.13) .. controls (153.73,124.25) and (153.73,133.81) .. (172.09,162.49) .. controls (190.45,191.17) and (107.83,191.17) .. (89.48,162.49) .. controls (71.12,133.81) and (71.12,114.69) .. (89.48,105.13) -- cycle ;
\draw  [fill={rgb, 255:red, 0; green, 0; blue, 0 }  ,fill opacity=1 ] (113.04,131.82) .. controls (113.04,130.89) and (113.76,130.14) .. (114.64,130.14) .. controls (115.53,130.14) and (116.25,130.89) .. (116.25,131.82) .. controls (116.25,132.74) and (115.53,133.49) .. (114.64,133.49) .. controls (113.76,133.49) and (113.04,132.74) .. (113.04,131.82) -- cycle ;
\draw  [dash pattern={on 0.84pt off 2.51pt}]  (114.64,131.82) -- (114.26,265.9) ;
\draw  [dash pattern={on 0.84pt off 2.51pt}]  (114.64,131.82) -- (4.41,132.06) ;
\draw  [fill={rgb, 255:red, 0; green, 0; blue, 0 }  ,fill opacity=1 ] (112.88,265.9) .. controls (112.88,265.11) and (113.5,264.47) .. (114.26,264.47) .. controls (115.02,264.47) and (115.64,265.11) .. (115.64,265.9) .. controls (115.64,266.69) and (115.02,267.33) .. (114.26,267.33) .. controls (113.5,267.33) and (112.88,266.69) .. (112.88,265.9) -- cycle ;
\draw  [fill={rgb, 255:red, 0; green, 0; blue, 0 }  ,fill opacity=1 ] (2.95,132.06) .. controls (2.95,131.22) and (3.6,130.54) .. (4.41,130.54) .. controls (5.21,130.54) and (5.86,131.22) .. (5.86,132.06) .. controls (5.86,132.89) and (5.21,133.57) .. (4.41,133.57) .. controls (3.6,133.57) and (2.95,132.89) .. (2.95,132.06) -- cycle ;
\draw    (194.74,58.44) -- (148.57,95.28) ;
\draw [shift={(147,96.52)}, rotate = 321.42] [color={rgb, 255:red, 0; green, 0; blue, 0 }  ][line width=0.75]    (10.93,-3.29) .. controls (6.95,-1.4) and (3.31,-0.3) .. (0,0) .. controls (3.31,0.3) and (6.95,1.4) .. (10.93,3.29)   ;
\draw    (207.13,144.17) -- (192.61,156.04) ;
\draw [shift={(191.07,157.31)}, rotate = 320.71] [color={rgb, 255:red, 0; green, 0; blue, 0 }  ][line width=0.75]    (10.93,-3.29) .. controls (6.95,-1.4) and (3.31,-0.3) .. (0,0) .. controls (3.31,0.3) and (6.95,1.4) .. (10.93,3.29)   ;
\draw    (21.7,57.65) -- (5.98,76.89) ;
\draw [shift={(4.71,78.44)}, rotate = 309.24] [color={rgb, 255:red, 0; green, 0; blue, 0 }  ][line width=0.75]    (10.93,-3.29) .. controls (6.95,-1.4) and (3.31,-0.3) .. (0,0) .. controls (3.31,0.3) and (6.95,1.4) .. (10.93,3.29)   ;
\draw    (212.64,250.28) -- (193.12,265.09) ;
\draw [shift={(191.53,266.3)}, rotate = 322.82] [color={rgb, 255:red, 0; green, 0; blue, 0 }  ][line width=0.75]    (10.93,-3.29) .. controls (6.95,-1.4) and (3.31,-0.3) .. (0,0) .. controls (3.31,0.3) and (6.95,1.4) .. (10.93,3.29)   ;
\draw  [dash pattern={on 0.84pt off 2.51pt}]  (159.85,197.7) -- (159.7,264.7) ;
\draw  [dash pattern={on 0.84pt off 2.51pt}]  (158.71,198.98) -- (4.41,198.74) ;
\draw  [fill={rgb, 255:red, 0; green, 0; blue, 0 }  ,fill opacity=1 ] (158.78,265.66) .. controls (158.78,264.87) and (159.4,264.23) .. (160.16,264.23) .. controls (160.92,264.23) and (161.54,264.87) .. (161.54,265.66) .. controls (161.54,266.45) and (160.92,267.09) .. (160.16,267.09) .. controls (159.4,267.09) and (158.78,266.45) .. (158.78,265.66) -- cycle ;
\draw  [fill={rgb, 255:red, 0; green, 0; blue, 0 }  ,fill opacity=1 ] (2.95,198.74) .. controls (2.95,197.9) and (3.6,197.22) .. (4.41,197.22) .. controls (5.21,197.22) and (5.86,197.9) .. (5.86,198.74) .. controls (5.86,199.57) and (5.21,200.25) .. (4.41,200.25) .. controls (3.6,200.25) and (2.95,199.57) .. (2.95,198.74) -- cycle ;
\draw  [fill={rgb, 255:red, 0; green, 0; blue, 0 }  ,fill opacity=1 ] (158.4,197.7) .. controls (158.4,196.87) and (159.05,196.19) .. (159.85,196.19) .. controls (160.66,196.19) and (161.31,196.87) .. (161.31,197.7) .. controls (161.31,198.54) and (160.66,199.22) .. (159.85,199.22) .. controls (159.05,199.22) and (158.4,198.54) .. (158.4,197.7) -- cycle ;
\draw   (299.67,93.5) -- (347.17,46) -- (462.67,46) -- (462.67,156.83) -- (415.17,204.33) -- (299.67,204.33) -- cycle ; \draw   (462.67,46) -- (415.17,93.5) -- (299.67,93.5) ; \draw   (415.17,93.5) -- (415.17,204.33) ;
\draw    (347.17,45) -- (348,158) ;
\draw    (348,158) -- (462.67,156.83) ;
\draw    (348,158) -- (299.67,204.33) ;
\draw   (264.33,129.17) -- (379.83,129.17) -- (379.83,240) -- (264.33,240) -- cycle ;
\draw   (349.57,244.83) -- (463.67,244.83) -- (414.77,292.33) -- (300.67,292.33) -- cycle ;
\draw   (516.67,46) -- (516.67,156.83) -- (469.17,204.33) -- (469.17,93.5) -- cycle ;
\draw  [fill={rgb, 255:red, 0; green, 0; blue, 0 }  ,fill opacity=0.1 ] (432.47,99.82) -- (419.47,169.59) .. controls (418.23,176.29) and (399.09,178.35) .. (376.74,174.19) .. controls (354.39,170.03) and (337.28,161.22) .. (338.53,154.51) -- (351.53,84.75) .. controls (352.77,78.04) and (371.91,75.98) .. (394.26,80.14) .. controls (416.61,84.3) and (433.72,93.12) .. (432.47,99.82) .. controls (431.22,106.53) and (412.09,108.59) .. (389.74,104.42) .. controls (367.38,100.26) and (350.28,91.45) .. (351.53,84.75) ;
\draw  [fill={rgb, 255:red, 0; green, 0; blue, 0 }  ,fill opacity=1 ] (386.45,129.66) .. controls (386.45,128.87) and (387.07,128.23) .. (387.83,128.23) .. controls (388.59,128.23) and (389.2,128.87) .. (389.2,129.66) .. controls (389.2,130.45) and (388.59,131.09) .. (387.83,131.09) .. controls (387.07,131.09) and (386.45,130.45) .. (386.45,129.66) -- cycle ;
\draw  [dash pattern={on 0.84pt off 2.51pt}]  (387.83,129.66) -- (311.33,209.33) ;
\draw    (294.13,62.17) -- (316.4,78.8) ;
\draw [shift={(318,80)}, rotate = 216.77] [color={rgb, 255:red, 0; green, 0; blue, 0 }  ][line width=0.75]    (10.93,-3.29) .. controls (6.95,-1.4) and (3.31,-0.3) .. (0,0) .. controls (3.31,0.3) and (6.95,1.4) .. (10.93,3.29)   ;
\draw  [fill={rgb, 255:red, 0; green, 0; blue, 0 }  ,fill opacity=1 ] (309.96,210.77) .. controls (309.96,209.98) and (310.57,209.33) .. (311.33,209.33) .. controls (312.09,209.33) and (312.71,209.98) .. (312.71,210.77) .. controls (312.71,211.56) and (312.09,212.2) .. (311.33,212.2) .. controls (310.57,212.2) and (309.96,211.56) .. (309.96,210.77) -- cycle ;
\draw  [dash pattern={on 0.84pt off 2.51pt}]  (387.83,129.66) -- (496,129.33) ;
\draw  [dash pattern={on 0.84pt off 2.51pt}]  (387.83,129.66) -- (389.38,257.23) ;
\draw  [fill={rgb, 255:red, 0; green, 0; blue, 0 }  ,fill opacity=1 ] (494.62,127.9) .. controls (494.62,127.11) and (495.24,126.47) .. (496,126.47) .. controls (496.76,126.47) and (497.38,127.11) .. (497.38,127.9) .. controls (497.38,128.69) and (496.76,129.33) .. (496,129.33) .. controls (495.24,129.33) and (494.62,128.69) .. (494.62,127.9) -- cycle ;
\draw  [fill={rgb, 255:red, 0; green, 0; blue, 0 }  ,fill opacity=1 ] (386.62,257.23) .. controls (386.62,256.44) and (387.24,255.8) .. (388,255.8) .. controls (388.76,255.8) and (389.38,256.44) .. (389.38,257.23) .. controls (389.38,258.02) and (388.76,258.67) .. (388,258.67) .. controls (387.24,258.67) and (386.62,258.02) .. (386.62,257.23) -- cycle ;
\draw    (354,36) -- (371.29,81.32) ;
\draw [shift={(372,83.19)}, rotate = 249.12] [color={rgb, 255:red, 0; green, 0; blue, 0 }  ][line width=0.75]    (10.93,-3.29) .. controls (6.95,-1.4) and (3.31,-0.3) .. (0,0) .. controls (3.31,0.3) and (6.95,1.4) .. (10.93,3.29)   ;
\draw    (266,113.33) -- (269.32,129.81) ;
\draw [shift={(269.71,131.77)}, rotate = 258.61] [color={rgb, 255:red, 0; green, 0; blue, 0 }  ][line width=0.75]    (10.93,-3.29) .. controls (6.95,-1.4) and (3.31,-0.3) .. (0,0) .. controls (3.31,0.3) and (6.95,1.4) .. (10.93,3.29)   ;
\draw    (504.7,37.65) -- (507.74,61.35) ;
\draw [shift={(508,63.33)}, rotate = 262.67] [color={rgb, 255:red, 0; green, 0; blue, 0 }  ][line width=0.75]    (10.93,-3.29) .. controls (6.95,-1.4) and (3.31,-0.3) .. (0,0) .. controls (3.31,0.3) and (6.95,1.4) .. (10.93,3.29)   ;
\draw    (308.67,268) -- (326.4,270.38) ;
\draw [shift={(328.38,270.64)}, rotate = 187.64] [color={rgb, 255:red, 0; green, 0; blue, 0 }  ][line width=0.75]    (10.93,-3.29) .. controls (6.95,-1.4) and (3.31,-0.3) .. (0,0) .. controls (3.31,0.3) and (6.95,1.4) .. (10.93,3.29)   ;
\draw  [dash pattern={on 0.84pt off 2.51pt}]  (410.83,186.66) -- (369.33,230.67) ;
\draw  [dash pattern={on 0.84pt off 2.51pt}]  (410.83,186.66) -- (482,186.33) ;
\draw  [dash pattern={on 0.84pt off 2.51pt}]  (410.83,186.66) -- (412.67,254.33) ;
\draw  [fill={rgb, 255:red, 0; green, 0; blue, 0 }  ,fill opacity=1 ] (411.62,252.23) .. controls (411.62,251.44) and (412.24,250.8) .. (413,250.8) .. controls (413.76,250.8) and (414.38,251.44) .. (414.38,252.23) .. controls (414.38,253.02) and (413.76,253.67) .. (413,253.67) .. controls (412.24,253.67) and (411.62,253.02) .. (411.62,252.23) -- cycle ;
\draw  [fill={rgb, 255:red, 0; green, 0; blue, 0 }  ,fill opacity=1 ] (367.96,231.1) .. controls (367.96,230.31) and (368.57,229.67) .. (369.33,229.67) .. controls (370.09,229.67) and (370.71,230.31) .. (370.71,231.1) .. controls (370.71,231.89) and (370.09,232.53) .. (369.33,232.53) .. controls (368.57,232.53) and (367.96,231.89) .. (367.96,231.1) -- cycle ;
\draw  [fill={rgb, 255:red, 0; green, 0; blue, 0 }  ,fill opacity=1 ] (409.45,188.09) .. controls (409.45,187.3) and (410.07,186.66) .. (410.83,186.66) .. controls (411.59,186.66) and (412.2,187.3) .. (412.2,188.09) .. controls (412.2,188.89) and (411.59,189.53) .. (410.83,189.53) .. controls (410.07,189.53) and (409.45,188.89) .. (409.45,188.09) -- cycle ;
\draw    (427.33,202) -- (412.29,188.02) ;
\draw [shift={(410.83,186.66)}, rotate = 42.9] [color={rgb, 255:red, 0; green, 0; blue, 0 }  ][line width=0.75]    (10.93,-3.29) .. controls (6.95,-1.4) and (3.31,-0.3) .. (0,0) .. controls (3.31,0.3) and (6.95,1.4) .. (10.93,3.29)   ;
\draw  [fill={rgb, 255:red, 0; green, 0; blue, 0 }  ,fill opacity=1 ] (479.33,185.9) .. controls (479.33,184.92) and (479.94,184.13) .. (480.69,184.13) .. controls (481.44,184.13) and (482.04,184.92) .. (482.04,185.9) .. controls (482.04,186.88) and (481.44,187.67) .. (480.69,187.67) .. controls (479.94,187.67) and (479.33,186.88) .. (479.33,185.9) -- cycle ;

\draw (115.38,115.66) node [anchor=north west][inner sep=0.75pt]  [font=\footnotesize] [align=left] {$\displaystyle \delta _{i_{1} ,i_{2}}$};
\draw (206.26,130.21) node [anchor=north west][inner sep=0.75pt]  [font=\footnotesize] [align=left] {$\displaystyle \overline{\delta }$};
\draw (193.42,37.83) node [anchor=north west][inner sep=0.75pt]  [font=\small] [align=left] {$\displaystyle S_{\overline{\zeta }^{1} ,\overline{\zeta }^{2}}^{\varphi ,\overline{\delta }}$};
\draw (4.89,109.09) node [anchor=north west][inner sep=0.75pt]  [font=\small] [align=left] {$\displaystyle \zeta _{i_{2}}^{1}$};
\draw (17.33,36.13) node [anchor=north west][inner sep=0.75pt]  [font=\small] [align=left] {$\displaystyle \overline{\zeta }^{1}$};
\draw (213.78,231.96) node [anchor=north west][inner sep=0.75pt]  [font=\small] [align=left] {$\displaystyle \overline{\zeta }^{2}$};
\draw (117.35,243.89) node [anchor=north west][inner sep=0.75pt]  [font=\small] [align=left] {$\displaystyle \zeta _{i_{1}}^{2}$};
\draw (160.32,183.77) node [anchor=north west][inner sep=0.75pt]  [font=\footnotesize] [align=left] {$\displaystyle \delta _{j_{1} ,j_{2}}$};
\draw (4.89,175.77) node [anchor=north west][inner sep=0.75pt]  [font=\small] [align=left] {$\displaystyle \zeta _{j_{2}}^{1}$};
\draw (160.03,243.17) node [anchor=north west][inner sep=0.75pt]  [font=\small] [align=left] {$\displaystyle \zeta _{j_{1}}^{2}$};
\draw (-3.47,302.49) node [anchor=north west][inner sep=0.75pt]  [font=\footnotesize] [align=left] {$\displaystyle \models \varphi \left( \zeta _{i_{2}}^{1} ,\zeta _{i_{1}}^{2} ,\delta _{i_{1} ,i_{2}}\right) \land \neg \varphi \left( \zeta _{j_{2}}^{1} ,\zeta _{j_{1}}^{2} ,\delta _{j_{1} ,j_{2}}\right)$};
\draw (371.4,110.09) node [anchor=north west][inner sep=0.75pt]  [font=\footnotesize] [align=left] {$\displaystyle \delta _{i_{1} ,i_{2} ,i_{3}}$};
\draw (286.26,47.21) node [anchor=north west][inner sep=0.75pt]  [font=\footnotesize] [align=left] {$\displaystyle \overline{\delta }$};
\draw (349.75,7.83) node [anchor=north west][inner sep=0.75pt]  [font=\small] [align=left] {$\displaystyle S_{\overline{\zeta }^{1} ,\overline{\zeta }^{2} ,\overline{\zeta }^{3}}^{\varphi ,\overline{\delta }}$};
\draw (262.33,91.46) node [anchor=north west][inner sep=0.75pt]  [font=\small] [align=left] {$\displaystyle \overline{\zeta }^{2}$};
\draw (500.33,16.13) node [anchor=north west][inner sep=0.75pt]  [font=\small] [align=left] {$\displaystyle \overline{\zeta }^{1}$};
\draw (290.33,256.33) node [anchor=north west][inner sep=0.75pt]  [font=\small] [align=left] {$\displaystyle \overline{\zeta }^{3}$};
\draw (380.68,261.22) node [anchor=north west][inner sep=0.75pt]  [font=\footnotesize] [align=left] {$\displaystyle \zeta _{i_{1} ,i_{2}}^{3}$};
\draw (481.35,104.56) node [anchor=north west][inner sep=0.75pt]  [font=\small] [align=left] {$\displaystyle \zeta _{i_{2} ,i_{3}}^{1}$};
\draw (295.83,215.17) node [anchor=north west][inner sep=0.75pt]  [font=\footnotesize] [align=left] {$\displaystyle \zeta _{i_{1} ,i_{3}}^{2}$};
\draw (415.38,247.23) node [anchor=north west][inner sep=0.75pt]  [font=\footnotesize] [align=left] {$\displaystyle \zeta _{j_{1} ,j_{2}}^{3}$};
\draw (338.83,213.83) node [anchor=north west][inner sep=0.75pt]  [font=\footnotesize] [align=left] {$\displaystyle \zeta _{j_{1} ,j_{3}}^{2}$};
\draw (426.74,189.76) node [anchor=north west][inner sep=0.75pt]  [font=\footnotesize] [align=left] {$\displaystyle \delta _{j_{1} ,j_{2} ,j_{3}}$};
\draw (470.01,155.22) node [anchor=north west][inner sep=0.75pt]  [font=\small] [align=left] {$\displaystyle \zeta _{j_{2} ,j_{3}}^{1}$};
\draw (267.53,301) node [anchor=north west][inner sep=0.75pt]  [font=\scriptsize] [align=left] {$\displaystyle  \begin{array}{{>{\displaystyle}l}}
\models \varphi \left( \zeta _{i_{2} ,i_{3}}^{1} ,\zeta _{i_{1} ,i_{3}}^{2} ,\zeta_{i_{1} ,i_{2}}^{3} ,\delta _{i_{1} ,i_{2} ,i_{3}}\right) \land \\
\neg \varphi \left( \zeta _{j_{2} ,j_{3}}^{1} ,\zeta _{j_{1} ,j_{3}}^{2} ,\zeta _{j_{1} ,j_{2}}^{3} ,\delta _{j_{1} ,j_{2} ,j_{3}}\right)
\end{array}$};
\draw    (258,-2) -- (296,-2) -- (296,20) -- (258,20) -- cycle  ;
\draw (261,2) node [anchor=north west][inner sep=0.75pt]  [font=\footnotesize] [align=left] {$\displaystyle k=3$};
\draw    (2,1) -- (40,1) -- (40,23) -- (2,23) -- cycle  ;
\draw (5,5) node [anchor=north west][inner sep=0.75pt]  [font=\footnotesize] [align=left] {$\displaystyle k=2$};

\end{tikzpicture}

In Section \ref{sec: comp lemma proof} we combine the type counting criterion and the array shattering lemma in order to prove the Composition Lemma for all $k$ (Theorem \ref{Composition Lemma}).
Finally, in Section \ref{sec: comp lem discussion} we include some discussion of the composition lemma and its possible generalizations and refinements. In particular, we provide a different proof of the dynamic array shattering lemma (Lemma \ref{lem: comp lemma induction}(1)) in the $2$-dimensional case using UDTFS, giving a polynomial bound on the number of sets of pairs that can be cut out (Lemma \ref{lem: array shatter poly bound}).

Combining quantifier elimination with the Composition Lemma discussed above, in Section \ref{sec: multilin n-dep and NSOP1} we obtain a generalization of Fact \ref{fac: bilin 2-dep} to $n$-linear forms for all $n$:
\begin{theorem*}[Theorem \ref{thm: Granger}]
Let $n \geq 2$ and let $T$ be a theory of \emph{infinite-dimensional} $n$-linear $K$-spaces eliminating quantifiers in the language $\mathcal{L}^{K}_{\theta,f}$. If $K$ is NIP, then $T$ is $n$-dependent (and strictly $n$-dependent if the form is generic). In particular, if $K$ is NIP, the theory of non-degenerate alternating forms is $n$-dependent.
\end{theorem*}

We also take the opportunity to discuss an orthogonal model-theoretic tameness property, NSOP$_1$. Preservation of NSOP$_1$ in bilinear forms over NSOP$_1$ fields was considered in various contexts, starting with \cite{chernikov2016model}, where it was demonstrated (relying on the results of Granger \cite{granger1999stability}) that if $K$ is an algebraically closed field, then the theory of non-degenerate symmetric or alternating bilinear forms over $K$ is NSOP$_1$. Following this, \cite{kaplan2020kim} proposed a description of Kim-independence over models for $K \models \ACF$; \cite{met2023sets} then proposed corrections to the description and proof, and a generalization from models to arbitrary sets in the same setting. A variant incorporating these corrections was given in \cite{kruckman2024new}, allowing real closed fields.
It turns out that all of these proofs contain gaps (see the discussion in Remark \ref{rem: errors in proofs}). Here we provide a (hopefully) correct proof of the description of Kim-independence, generalizing to multilinear forms and $K$ an arbitrary NSOP$_1$ field:

\begin{theorem*}[Theorem \ref{thm: NSOP1 and Kim indep}]
	If $\Th(K)$ is NSOP$_1$, then $T := \prescript{}{\Alt}T^{K}_n$ is also NSOP$_1$. Moreover, for any $M\models T$ and small ($\mathcal{L}^{K}_{\theta,f}$-)substructures $A,B\supseteq M$, $\tp(A/B)$ does not Kim-divide over $M$ if and only if $A \ind^{K}_M B$ and $A \ind^{V}_M B$ (in the notation of Section \ref{sec: NSOP1 preserved}). If $K$ is finite, then $T$ is simple (Corollary \ref{cor: multilin fin field simple}).
\end{theorem*}

The Composition Lemma (Theorem \ref{Composition Lemma}) was recently used to obtain some further natural examples of $n$-dependent structures: generic nilpotent groups and Lie algebras over finite fields \cite{d2025model}; generic nilpotent Lie algebras over algebraically closed fields in a two-sorted language \cite{d2024two}; and pseudofinite quadratic geometries \cite{kestner2024some}.
All of these are still explicitly controlled by multilinear maps over the NIP part. In view of this, it is tempting to speculate that $n$-dependence
of a theory should imply some form of ``linearity over the dependent part''. While
formulating this precisely still appears difficult, this motivated our more precise conjecture:
\begin{conj} (\cite[Conjecture 1.1]{chernikov2021n}; also \cite[Problem 4.10]{chernikov2019mekler})
If a field (viewed as a structure in the pure ring language) is $n$-dependent for some $n \geq 1$, then it is already dependent.
\end{conj}
\noindent While this conjecture remains open, some evidence towards it is provided by the following results, which demonstrate that certain known properties of NIP fields also hold for $n$-dependent fields: $n$-dependent fields are Artin--Schreier closed \cite{hempel2016n}; $n$-dependent valued fields of positive characteristic are Henselian \cite{chernikov2021n}; and an analogue of the question for valued fields reduces to pure fields \cite{boissonneau2024nipn}.

Some further evidence is provided by the intersection conditions on the connected components, generalizing their absoluteness in NIP groups and resembling 
modular behavior in the $2$-dependent case. Given an ($\emptyset$-)type-definable group $G$ and a small set of parameters $A$, we denote by $G^{00}_A$ (respectively, $G^{\infty}_A$) the intersection of all subgroups of $G$ of bounded index that are type-definable over $A$ (respectively, $\Aut(\mathbb{M}/A)$-invariant); see Section \ref{sec: inv conn comp} for more details. A crucial fact about definable groups in dependent theories is that for every small set $A$ one has $G^{00}_A = G^{00}_{\emptyset}$ \cite{shelah2008minimal} and $G^{\infty}_A = G^{\infty}_{\emptyset}$ (Shelah \cite{MR3666349} in the abelian case, Gismatullin \cite{Gismatullin2011} in general). Generalizing Shelah's result in the $2$-dependent case (Fact \ref{fac: Sh G00 2-dep}), we established the following in \cite{chernikov2021n}:
\begin{fact*}[Fact \ref{fac: G00 for n-dep}]
	If $T$ is $n$-dependent and $G = G(\mathbb{M})$ is a type-definable group (over $\emptyset$), then for any small model $\M$ and finite tuples $b_1, \ldots, b_{n-1}$ sufficiently independent over $\M$ in an appropriate sense (see Fact \ref{fac: G00 for n-dep} for the exact statement), we have that 
$$G^{00}_{\M \cup b_1 \cup \dots \cup  b_{n-1}} = \bigcap_{i=1}^{n-1} G_{\M \cup  b_1 \cup \ldots \cup b_{i-1} \cup b_{i+1} \cup \ldots \cup  b_{n-1}}^{00} \cap G^{00}_{C \cup  b_1\cup \dots\cup  b_{n-1}}$$ 
for some $C \subseteq \M$ of absolutely bounded size.
\end{fact*}

\noindent In Section \ref{sec: inv conn comp} we prove an analog of Fact \ref{fac: G00 for n-dep} for $G^{\infty}$ in arbitrary $k$-dependent groups:
\begin{theorem*}[Theorem \ref{thm: Ginfty k-dep main}]
	For any $k\geq 1$, let $T$ be a $k$-dependent theory and $G = G(\mathbb{M})$ a ($\emptyset$-)type-definable group. Let $\M$ be a small model and $\bar{b}_1, \ldots, \bar{b}_{k-1}$ finite tuples in $\mathbb{M}$ sufficiently independent over $\M$. Then there is some $C \subseteq \M$ with $|C| \leq \beth_2(|T|)$ such that
	\begin{gather*}
		G^{\infty}_{\M \cup \bar{b}_1 \cup \ldots \cup \bar{b}_{k-1}} = \left( \bigcap_{i=1}^{k-1} G^{\infty}_{\M \cup \bar{b}_1 \cup \ldots \cup \bar{b}_{i-1} \cup \bar{b}_{i+1} \cup \ldots \cup \bar{b}_{k-1}} \right) \cap G^{\infty}_{C \cup \bar{b}_1 \cup \ldots \cup \bar{b}_{k-1}}.
		\end{gather*}
\end{theorem*}

\noindent 
In Section \ref{sec: inv conn prelims} we recall some preliminaries on the connected components and our independence assumption on the tuples $\bar{b}_1, \ldots, \bar{b}_{k-1}$ (\emph{generic position}, see Definition \ref{def: generic position}). In Section \ref{sec: Lstp and thick} we recall Lascar strong types and their description in terms of thick formulas and indiscernible sequences. We note that this description can be strengthened using \emph{uniformly thick formulas}, which allows us to carry out certain compactness arguments. In Section \ref{sec: Ginfty and Lstp} we recall the description of $G^{\infty}_S$ in terms of the Lascar-difference sets $X_{S,<\omega} = \{ a^{-1}b : a,b \in G(\mathbb{M}), a \equiv^{L}_{S} b \}$ and their local approximations, and prove some lemmas on manipulating these sets. In Section \ref{sec: proof of Ginf} we give a proof of Theorem \ref{thm: Ginfty k-dep main}. Compared to our proof of Fact \ref{fac: G00 for n-dep} in \cite{chernikov2021n}, the situation is more complicated since we cannot work with subgroups, but rather only with approximate subgroups given by the powers of the sets $X_S$. Finally, in Section \ref{sec: Ginfty in multilin} we calculate the example of $G^{\infty}$ for the additive group in multilinear forms over prime fields.

\subsection*{Notation}
As usual, given $n \in \mathbb{N}$, we write $[n]$ to denote the set $\{1, 2, \ldots, n \}$. Given a formula $\varphi$, we write $\varphi^1$ to denote $\varphi$, and $\varphi^0$ to denote $\neg \varphi$. Given a sequence of tuples $I$, we write $|I|$ to denote the length of the sequence. As usual, without further context $\delta_{i,j}$ denotes the Kronecker delta symbol, i.e.~$\delta_{i,j}=1$ if $i=j$ and $0$ if $i\neq j$. Throughout the paper, $T$ will denote a first-order theory and $\mathbb{M}$ a monster model of $T$. Given a model $\mathcal{M}$ of $T$ and a tuple of variables $x$, we write $\mathcal{M}^x$ for the product of sorts of $\mathcal{M}$ corresponding to the variables in $x$.

\subsection*{Acknowledgments
}We are grateful to the referee for many valuable suggestions on
improving the paper.
We  thank Yvon Bossut, Jan Dobrowolski, Alex Kruckman and Nick Ramsey for comments on Section  \ref{sec: NSOP1 preserved}. The majority of the results in this paper were obtained by 2020 while both authors were at UCLA, and circulated since then, along with some improvements added later.
Chernikov was partially supported by the NSF
Research Grants DMS-2246598, DMS-2554164 and by the Alexander von Humboldt Foundation. Both authors were partially supported by the NSF CAREER grant DMS-1651321.

\section{Multilinear forms}\label{sec: multilinear forms}

\subsection{Perfect pairings, non-degenerate and generic $n$-linear forms}\label{sec: non-degen}
In this section we consider possible generalizations of \emph{non-degeneracy} from bilinear forms to $n$-linear forms, for an arbitrary $n \in \mathbb{N}_{\geq 2}$. Several other notions of non-degeneracy for multilinear forms appear in the literature, see e.g.~\cite{gelfand2009discriminants, hitchin2001stable, kazhdan2020properties} (see also \cite[Remark 1.1]{harman2024ultrahomogeneous}).

Let $V$ be a vector space over a field $K$. Recall that a bilinear form $\langle -, - \rangle: V^2 \to K$ is \emph{degenerate} if there exists a vector $v \in V, v \neq 0$ such that $\langle v,w \rangle = 0$ for all $w \in V$.

If $V$ has finite dimension $d$, a bilinear form \Bf is non-degenerate if and only if it is a \emph{perfect pairing}, i.e.~the maps $V \rightarrow V^\ast,\ v \mapsto  \langle v, - \rangle$ and $V \rightarrow V^\ast,\ v \mapsto \langle -, v \rangle$ are isomorphisms. In other words, for any basis $v_1, \ldots, v_d$ of $V$ and any $k_1, \ldots, k_d \in K$ there is $w \in V$ such that $\langle v_i, w \rangle =k_i$ for all $i=1, \ldots, d$. This equivalence is no longer true in infinite-dimensional vector spaces for dimensional reasons. However, we can still obtain a ``local'' version: the bilinear form \Bf is non-degenerate if and only if for any $m \in \mathbb{N}$, any linearly independent vectors $v_1, \dots, v_m$ in $V$ and any $k_1, \dots, k_m \in K$ there is $w \in V$ such that $\langle v_i, w \rangle =k_i$ for all $i=1, \dots, m$. In fact, we can find such a vector $w$ in any subspace $W \subseteq V$ such that $W^\perp \cap \Span \left(v_1, \dots, v_m \right) = \{0\}$.

A naive attempt to generalize non-degeneracy to $n$-linear forms $\Bln:V^n \to K$ would be: for any non-zero $v_1, \ldots, v_{n-1} \in V$ there is $w \in V$ such that $ \langle v_1, \dots, v_{n-1}, w \rangle_n \neq 0$. However, this condition typically cannot be satisfied in an $n$-linear form satisfying some additional natural requirements (e.g.~alternating or symmetric). In the case of an alternating form, we have for example that $\langle v,v, v_3, \dots, v_{n-1}, w \rangle_n=0$ regardless of the choice of $v, v_3, \ldots, v_{n-1},w \in V$. To circumvent this issue, we work in the tensor product space $\Ten^{n-1} V$ modulo the subspace $N$ of $\Ten^{n-1} V$ generated by the elements $v_1 \ten \dots \ten v_{n-1}$ for which the map $ V \rightarrow K, \ w \mapsto  \langle v_1, \dots, v_{n-1}, w \rangle_n$ should be the zero map. For example, for alternating $n$-linear forms, we take the subspace $N$ to be
 $$\Alt :=  \Span \left( \left\{v_1 \ten \dots\ten v_{n-1} \mid v_1, \dots, v_{n-1} \text{ are linearly dependent}\right\} \right),$$ 
 and for symmetric $n$-linear forms we let $N$ be
  $$\Sym :=\Span \left( \{v_1\ten \dots \ten v_{n-1} - v_{\sigma(1)} \ten \dots \ten v_{\sigma(n-1)} \mid \sigma \in S_{n-1}  \} \right).$$ 
  In these cases we obtain:
  \begin{gather*}
  	\left( \Ten^{n-1} V\right)\big/\Alt\ = \bigwedge^{n-1} V, \quad \text{ i.e.~the } (n-1)\text{th exterior power of }V, \\
  	\left(\Ten^{n-1} V\right)\big/\Sym\ = \bigvee^{n-1} V, \quad \text{i.e.~the } (n-1)\text{th symmetric power of }V.
  \end{gather*}

   To talk about these cases in a uniform way (as many properties apply to both of them), from now on we let  $\lozenge \in \{\wedge, \vee\}$ and write  
$\lozenge^{n-1} V$ to mean either $\bigwedge^{n-1} V$ or $\bigvee^{n-1} V$. Moreover, given $\sum_{i=1}^m k_i\, (v_{i,1} \ten \dots \ten v_{i,n-1}) \in \Ten^{n-1} V$, we write $\overline{\sum_{i=1}^m k_i\, (v_{i,1} \ten \dots \ten v_{i,n-1})}$ for its equivalence class in $\lozenge^{n-1} V$.

Another way of thinking about these powers is via their universal property. This gives rise to a functor sending a vector space $V$ to its exterior/symmetric power. As the embedding $i$ of any subspace $W$ into a vector space $V$ is split injective, $\lozenge^{n-1}W$ embeds into $\lozenge^{n-1}V$ via $\lozenge^{n-1}i$ and can be naturally viewed as a subspace.

Then the $n$-linear form \Bn  gives rise to a bilinear form \Bi on $\big( \lozenge^{n-1} V \big) \times V$ defined by \[ \left \langle \overline{\sum_{i=1}^m k_i\, (v_{i,1} \ten \dots \ten v_{i,n-1})},v \right\rangle_2 := \sum_{i=1}^m k_i\,\langle v_{i,1}, \dots , v_{i,n-1},v \rangle_n.\]

This remains true upon restriction to $
\left( \lozenge^{n-1} U \right) \times W$ for subspaces $U$ and $W$ of $V$. To ease the notation, we will not distinguish between the bilinear forms 
 \begin{gather*}
 	\Bli:  \big(\lozenge^{n-1} V \big) \times V  \rightarrow K \textrm{ and }\\
 \Bli:  \big( \lozenge^{n-1} U \big)\times W \rightarrow K\ \text{\scriptsize (its restriction to $\lozenge^{n-1}U\times W$)}.
 \end{gather*}

 \textbf{For the rest of the section, we let $\left(V,\Bln \right)$  be an alternating/symmetric $n$-linear space ($n\geq 2$) and let \Bi  be the associated bilinear form on $\big( \lozenge^{n-1} V \big) \times V$.}

\begin{defn}\label{def: non-degen} We say that  the $n$-linear form \Bn is:
\begin{enumerate}
\item \emph{non-degenerate} if for any non-zero  $t \in \lozenge^{n-1} V $ there is $w \in V$ such that $\langle t,w \rangle_2 \neq 0$;
\item a \emph{perfect pairing} if the maps 
$$V \rightarrow  \left(\lozenge^{n-1} V\right)^\ast, \ v \mapsto \langle -, v \rangle_2\quad \text{and}\quad \lozenge^{n-1} V\rightarrow V^\ast, \ t \mapsto \langle t, - \rangle_2$$ are vector space isomorphisms (where as usual $\ast$ denotes the dual space);
\item \emph{generic} if for any $m\in \mathbb{N}$, any linearly independent elements $t_1, \dots, t_m \in \lozenge^{n-1} V$ and any $k_1, \dots, k_m \in K$ there is $w \in V$ such that $\langle t_i, w \rangle_2=k_i$ for all $i \in \set{m}$.	
\end{enumerate}
We also refer to the corresponding alternating/symmetric $n$-linear space $\left(V, \Bln \right)$ as \emph{non-degenerate} or \emph{generic}, respectively. 
\end{defn}

\noindent	It follows immediately from the definitions that any perfect pairing is a generic form, and any generic form is non-degenerate.
 Now, we explore the different characterizations as well as  connections between these  notions.

\begin{defn}
	We define the map
$$\Psi: \lozenge^{n-1} V  \rightarrow V^\ast, \ t \mapsto \langle t, - \rangle_2, $$

\noindent  and, for any subspace $W$ of $V$,
  $$\Phi_W: V \rightarrow  \big(\lozenge^{n-1} W\big)^\ast, \ v \mapsto \langle -, v \rangle_2.$$
We  write $\Phi$ for $\Phi_V$. 
\end{defn}

\noindent The following lemma is immediate from the definitions:

\begin{lemma}\label{L:Non-DegGenChar}
\begin{enumerate}
\item\label{Enu:Non-Deg} The form $\Bln$ on $V$ is non-degenerate if and only if  the linear map $\Psi$ is injective.
\item\label{Enu:Generic}  The following are equivalent
:

\begin{enumerate}
 \item 	\Bn is generic;
 \item $\Phi_W$ is surjective for any finite-dimensional subspace $W$ of $V$;
 \item for any finite-dimensional subspace $W$ of $V$ there is a basis $t_1, \dots, t_l$ of $\lozenge^{n-1} W$ and $u_1, \dots, u_l$ in $V$  such that 	$\langle t_i,u_j \rangle_2= \delta_{ij}$. In this case we say that  the tuples $(t_1, \dots, t_l)$  and  $(u_1, \dots, u_l)$ are \emph{dual}.
\end{enumerate}

\end{enumerate}
\end{lemma}

We consider separately the cases when $V$ has finite or infinite dimension.

\subsubsection{Finite-dimensional case} Let $V$ be of dimension $d \in \mathbb{N}$ and $n>2$.

For an alternating form, $\dim\bigwedge^{n-1}V=\binom d{n-1}$. If $d<n-1$, then $\bigwedge^{n-1}V=0$, so non-degeneracy and genericity hold vacuously. If $d=n-1$, every alternating $n$-form is zero and neither condition holds. If $d=n$, non-degeneracy and genericity hold precisely when the form is non-zero, and in that case all three notions in Definition \ref{def: non-degen} coincide. If $d>n$, neither non-degeneracy nor genericity can hold (in contrast to the bilinear case).

For a symmetric form, $\dim\bigvee^{n-1}V=\binom{d+n-2}{n-1}$. Thus, apart from the vacuous zero-dimensional case, a symmetric $n$-form with $n>2$ can be non-degenerate or generic only when $d=1$, in which case this holds precisely when the form is non-zero.

\subsubsection{Infinite-dimensional case} 

 For an infinite-dimensional vector space $V$ over a field $K$, its dual space $V^\ast$ is never isomorphic to $V$ itself (in fact, if a basis of $V$ has size $\kappa$, i.e.~$V$ is isomorphic to $K^{(\kappa)}$, its dual space is isomorphic to $K^\kappa$). Hence, as the dimension of $\lozenge^{n-1} V$ is at least as big as the dimension of $V$, an alternating/symmetric $n$-linear form \Bn on $V$ can never be a perfect pairing.

First, we want to show that any $n$-linear form can be extended to a non-degenerate $n$-linear form on a larger (and infinite-dimensional in general) vector space. In particular, non-degenerate $n$-linear forms on infinite-dimensional vector spaces exist.

\begin{lemma}\label{lem: ext multilin to gen} For any alternating/symmetric $n$-linear space  $\left(U,\Cln \right)$ there is a vector space $V$ of dimension at most $\aleph_0 + \dim(U)$ (over the same field) containing $U$ and an alternating/symmetric $n$-linear form \Bn on $V$ extending \Cn and such that   $ \left( V,\Bln \right)$ is non-degenerate. 
\end{lemma}
\begin{proof}
Let $\kappa := \aleph_0 + \dim(U)$ and let $W$ be an arbitrary vector space over $K$ of dimension $\kappa^+$ containing $U$. We construct an increasing chain $(V_i : i \in \mathbb N)$ of subspaces of $W$, each of dimension at most $\kappa$ and containing $U$, and define an alternating/symmetric $n$-linear form \Bn on each $V_{i+1}$ extending \Bn on $V_i$ so that the map $ \lozenge^{n-1} V_m \rightarrow V_{m+1}^\ast,\ t \mapsto \langle t, -\rangle_2  $ is injective for all $m$.

 Let $V_0 := U$ and define \Bn on $V_0$ to be \Cn. Assume $\left(V_m, \Bln\right)$ has already been constructed for some $m\in \mathbb N$. 

If the map $\Psi_{V_m} :  \lozenge^{n-1} V_m  \rightarrow V_m^\ast, \ t \mapsto \langle t, -\rangle_2 $ is  injective, then \Bn is non-degenerate on $V_m$ by Lemma \ref{L:Non-DegGenChar}(\ref{Enu:Non-Deg}). We let $V:= V_m$ and  conclude. 

Otherwise we can find a basis $(t_i:i\in I)$ of $\ker(\Psi_{V_m})$, with $|I|\leq\kappa$, and extend it to a basis $(t_i:i\in J)$ of $\lozenge^{n-1}V_m$, where $I\subseteq J$. Let $(w_i:i\in I)$ be elements of $W$ linearly independent over $V_m$, put $W_m:=\Span\{w_i:i\in I\}$, and let $V_{m+1}:=V_m\oplus W_m$.
For each $j\in I$, prescribe the contraction with $w_j$ on $\lozenge^{n-1}V_m$ by $\langle t_i,w_j\rangle_2=\delta_{i,j}$ for every $i\in J$. Extend these prescriptions by linearity and alternation/symmetry, and set every coefficient involving at least two vectors from $W_m$ equal to zero. This defines an alternating/symmetric $n$-linear form on $V_{m+1}$ extending the form on $V_m$.

Then  the map $  \lozenge^{n-1} V_m \rightarrow V_{m+1}^\ast,\ t \mapsto \langle t, -\rangle_2  $ is injective by construction, completing the inductive step. 

Now, let $V := \bigcup_{i\in \mathbb N} V_i$ and  $t \in \lozenge^{n-1} V $ be  non-zero. Then there is $m\in \mathbb N$ and non-zero $s  \in \Ten^{n-1} V_m$ such that $\overline s=t$. As $  \langle \overline s, -\rangle_2 \in V_{m+1}^\ast $ is non-zero by construction,  there is some $v \in V_{m+1}$ such that $\langle t,v\rangle _2=\langle \overline s,v \rangle _2 \neq 0$. Hence \Bn is non-degenerate on $V$ (by Lemma \ref{L:Non-DegGenChar}(1)).
\end{proof}
Next, we show that non-degeneracy and genericity are equivalent:
\begin{lemma}\label{L:NonDegiffGeneric}
The form \Bn is non-degenerate if and only if it is generic.
\end{lemma}
\begin{proof}
Genericity plainly implies non-degeneracy. Conversely, assume that \Bn is non-degenerate. If $t_1,\ldots,t_m$ are linearly independent in $\lozenge^{n-1}V$, then injectivity of $\Psi$ implies that $\Psi(t_1),\ldots,\Psi(t_m)$ are linearly independent elements of $V^*$. Let $F$ be the linear map whose coordinate functionals are precisely $\Psi(t_1),\ldots,\Psi(t_m)$, namely
\[
F:V\longrightarrow K^m,\qquad
w\longmapsto
(\langle t_1,w\rangle_2,\ldots,\langle t_m,w\rangle_2)
\]
If $F$ were not surjective, then
$\operatorname{im}(F)$ would be a proper subspace of the
finite-dimensional space $K^m$, so there would be a non-zero linear
functional $\lambda(x_1,\ldots,x_m)=\sum_{i=1}^m a_i x_i$ on $K^m$
vanishing on $\operatorname{im}(F)$.  Hence, for every $w\in V$,
\[
0=\lambda(F(w))=\sum_{i=1}^m a_i\langle t_i,w\rangle_2
 =\left(\sum_{i=1}^m a_i\Psi(t_i)\right)(w),
\]
contradicting the linear independence of the $\Psi(t_i)$.  Hence the
map $F$ is surjective.  Consequently, for every
$(k_1,\ldots,k_m)\in K^m$ there is $w\in V$ such that
$\langle t_i,w\rangle_2=k_i$ for all $i\in\set{m}$, which is precisely
genericity.
\end{proof}

%
We next record explicitly that non-zero elements of an exterior power are
uniformly first-order recognizable when the number of displayed
pure tensors is fixed.  This will be used to axiomatize non-degeneracy.

\begin{lemma}\label{lem: power nonzero definable}
For every $m,r\geq 1$ there is an $\mathcal{L}^K_{\theta,f}$-formula
\[
 \operatorname{NZ}^{\Alt}_{m,r}
 \big((\lambda_i)_{i\in\set{m}};(v_{i,j})_{i\in\set{m},j\in\set{r}}\big)
\]
such that, in every $K$-vector space $V$, for any $\lambda_i \in K$ and $v_{i,j} \in V$ we have 
\[
\models  \operatorname{NZ}^{\Alt}_{m,r}(\bar\lambda;\bar v)
 \quad\Longleftrightarrow\quad
 \sum_{i=1}^m\lambda_i\,
 \overline{v_{i,1}\ten\cdots\ten v_{i,r}}\neq 0
 \text{ in }\bigwedge^rV.
\]
\end{lemma}

\begin{proof}
Put $P:=\set{m}\times\set{r}$ and fix an ordering of $P$.
For each non-empty $B\subseteq P$, write $B=(b_1<\cdots<b_d)$, where
$b_q=(i_q,j_q)$, and put $u_q:=v_{i_q,j_q}$.  Let
\[
 \operatorname{Bas}_B(\bar v):=
 \theta_d(u_1,\ldots,u_d)\wedge
 \bigwedge_{(i,j)\in P}\neg\theta_{d+1}(u_1,\ldots,u_d,v_{i,j}).
\]
Thus $\operatorname{Bas}_B$ says exactly that $(u_1,\ldots,u_d)$ is a
basis of the span of the vectors $\bar{v}$.  Assume that
$\operatorname{Bas}_B(\bar v)$ holds.  For every $(i,j)\in P$ and
$q\in\set{d}$, put
\[
 c^B_{i,j,q}:=f_q^d(v_{i,j};u_1,\ldots,u_d).
\]
Since $u_1,\ldots,u_d$ is then a basis of the span of the displayed
vectors, the definition of $f_q^d$ gives
\[
 v_{i,j}=\sum_{q=1}^d c^B_{i,j,q}u_q
 \qquad\text{for every $(i,j)\in P$.}
\]
Set
\[
 T:=\sum_{i=1}^m\lambda_i
 \bigl(v_{i,1}\wedge\cdots\wedge v_{i,r}\bigr)\in\bigwedge^rV.
\]
By multilinearity, for each $i\in\set{m}$,
\[
 v_{i,1}\wedge\cdots\wedge v_{i,r}
 =\sum_{p_1,\ldots,p_r\in\set{d}}
   \left(\prod_{j=1}^r c^B_{i,j,p_j}\right)
   u_{p_1}\wedge\cdots\wedge u_{p_r}.
\]
Every term with two equal indices $p_j$ vanishes.  For a fixed
$Q=(q_1<\cdots<q_r)\subseteq\set{d}$, the remaining terms involving
exactly the indices in $Q$ are indexed by the permutations $\pi\in S_r$.
Since
\[
 u_{q_{\pi(1)}}\wedge\cdots\wedge u_{q_{\pi(r)}}
 =\operatorname{sgn}(\pi)
  u_{q_1}\wedge\cdots\wedge u_{q_r},
\]
their total coefficient in $v_{i,1}\wedge\cdots\wedge v_{i,r}$ is
\[
 \sum_{\pi\in S_r}\operatorname{sgn}(\pi)
 \prod_{j=1}^r c^B_{i,j,q_{\pi(j)}}
 =\det\big(c^B_{i,j,q_s}\big)_{1\leq j,s\leq r}.
\]
It follows that
\[
 T=\sum_{Q=(q_1<\cdots<q_r)\subseteq\set{d}}
 C^B_Q(\bar\lambda;\bar v)\,
 u_{q_1}\wedge\cdots\wedge u_{q_r},
\]
where
\[
 C^B_Q(\bar\lambda;\bar v):=
 \sum_{i=1}^m\lambda_i
 \det\big(c^B_{i,j,q_s}\big)_{1\leq j,s\leq r}.
\]
The displayed wedge monomials are linearly independent because
$u_1,\ldots,u_d$ are linearly independent.  Hence $T\neq0$ if and only if
$C^B_Q(\bar\lambda;\bar v)\neq0$ for at least one
$Q\in\binom{\set{d}}{r}$.
Consequently we may take
\[
 \operatorname{NZ}^{\Alt}_{m,r}(\bar\lambda;\bar v):=
 \bigvee_{\emptyset\neq B\subseteq P}
 \left(
   \operatorname{Bas}_B(\bar v)\wedge
   \bigvee_{Q\in\binom{\set{|B|}}{r}} C^B_Q(\bar\lambda;\bar v)\neq 0
 \right),
\]
where the inner disjunction is false when $|B|<r$.  This is a 
first-order formula, and the coordinate calculation above proves the desired
equivalence.
\end{proof}

\begin{remark}\label{rem: pure tensors insufficient}
For $n>2$, testing only non-zero pure tensors does not by itself establish
injectivity of the contraction map $\Psi$: a non-trivial linear combination of
pure tensors may still belong to its kernel.\end{remark}

To carry out a back--and--forth argument in vector spaces with a non-degenerate form in the next subsection, we need the following finer version of genericity:
\begin{lemma}\label{L:findingw}
		Let $\left(V, \Bln \right)$ be an infinite-dimensional non-degenerate alternating/symmetric $n$-linear space, and let $U$ be a finite-dimensional subspace of $V$. Then for any linearly independent elements $t_1, \dots, t_m \in \lozenge^{n-1} V$ and $k_1, \dots, k_m \in K$ there is \emph{$w \in V$ linearly independent from $U$} such that $ \langle t_i, w \rangle_2=k_i$ for all $i \in \set{m}$.  
\end{lemma}
\begin{proof} Let $\left(v_j : j \in \set{l} \right)$ be linearly independent  vectors in $V$ such that for all $i\in \set{m}$  we have
$$t_i \in \lozenge^{n-1} \underbrace{\Span (  v_j : j \in \set{l} )}_{=:W}.$$
%
Let $f\in \left( \lozenge^{n-1} W\right)^\ast$ be such that
\[f \left( t_i \right) = k_i\]
for all $i \in \set{m}$. As \Bn is generic by Lemma \ref{L:NonDegiffGeneric}, $\Phi_W$ is surjective by Lemma \ref{L:Non-DegGenChar}(\ref{Enu:Generic}). Moreover,
since $\left( \lozenge^{n-1}W\right)^\ast$ is finite-dimensional, the kernel of $\Phi_W$ has infinite dimension. Thus, one can find $w$ linearly independent from $U$ such that $f= \langle -,w \rangle_2$, and we can conclude.
\end{proof}

\subsection{Quantifier elimination and completeness}\label{sec: QE}
In this section we generalize (and correct) a quantifier-elimination result of Granger \cite{granger1999stability} from non-degenerate bilinear forms to non-degenerate alternating multilinear forms.

\begin{defn}\label{def: language of bilin forms}
		We consider $n$-linear spaces as structures in the language $\mathcal{L}$ consisting of two sorts $V$ and $K$, the field language $\mathcal{L}_{\text{field}}=\{+_K, \cdot_K, -_K, {}^{-1}_K, 0_K, 1_K\}$ on $K$, the group language $\mathcal{L}_{G}= \{+_V, 0_V\}$ on $V$, a scalar-multiplication function $\cdot_V: K \times V \to V$, and a function symbol $\langle - , \dots, - \rangle_n$ for an $n$-linear form $V^n \rightarrow K$. The language $\mathcal{L}_{\theta, f}$ is obtained from $\mathcal{L}$ by adding the following ($\mathcal{L}(\emptyset)$-definable, with quantifiers) relation and function symbols:
\begin{itemize}
	\item for each $p \in \omega$, a $p$-ary predicate $\theta_p$, interpreted on $V^p$ as the relation of linear independence of $p$ vectors in $V$ over $K$;
	\item for each $p \in \omega$ and $i \in \set{p}$, a $(p+1)$-ary function symbol $f_i^p:V^{p+1}\rightarrow K$. If $v_1,\ldots,v_p$ are linearly independent and $v=\sum_{j=1}^p\lambda_jv_j$ 
	for $\lambda_1,\ldots,\lambda_p\in K$, then set $f_i^p(v;v_1,\ldots,v_p):=\lambda_i$. 
	In every other case, set $f_i^p(v;v_1,\ldots,v_p):=0$.
\end{itemize} 

Let $\mathcal{L}^K$ be an expansion of the language of fields by relations on $K^p$, for $p\in\omega$ (so there are no new constant or function symbols), that are $\emptyset$-definable in the language of fields and such that $K$ eliminates quantifiers in $\mathcal{L}^K$ (e.g.~we can always take the Morleyization of $K$). Let $\mathcal{L}_{\theta,f}^K := \mathcal{L}_{\theta,f} \cup \mathcal{L}^K $. Finally, let $\mathcal{L}^K_{\VS} := \mathcal{L}_{\theta,f}^K \setminus \left\{ \Bln \right\}$ be the language of two-sorted vector spaces.
\end{defn}

\begin{remark}
	Note that the functions $f_i^p$ are similar to those in the language introduced by Delon to study pairs of algebraically closed fields in \cite{delon2012elimination}.
\end{remark}

\begin{prop}\label{prop: non-degen axiom}
	The class of all infinite-dimensional non-degenerate alternating $n$-linear spaces (with the field sort a model of $\Th(K)$) in the language $\mathcal{L}^{K}_{\theta,f}$ is elementary. We will denote its theory by $T := \prescript{}{\Alt}T^K_n$ (adapting the notation in Granger \cite{granger1999stability}).
\end{prop}
\begin{proof}
Being alternating is expressed by the axiom
$$\forall v_1, \dots, v_{n}\ (\neg \theta_{n}(v_1, \dots, v_{n}) \rightarrow \langle v_1, \dots, v_{n} \rangle_n = 0), $$
and infinite dimension is expressed by the usual axiom scheme.  For every
$m\geq 1$, add the sentence
\begin{gather*}
(\mathrm{ND})_m:\quad
\forall (\lambda_i)_{i\in\set{m}}\,
\forall (v_{i,j})_{i\in\set{m},j\in\set{n-1}}
\Bigg(
\operatorname{NZ}^{\Alt}_{m,n-1}(\bar\lambda;\bar v)\ \longrightarrow\\
\exists w\ 
\sum_{i=1}^m\lambda_i
\langle v_{i,1},\ldots,v_{i,n-1},w\rangle_n\neq 0
\Bigg).
\end{gather*}
The formulas $\operatorname{NZ}^{\Alt}_{m,n-1}$ are supplied by Lemma
\ref{lem: power nonzero definable}.  Every element of
$\bigwedge^{n-1}V$ is a finite sum of pure tensors, so the scheme
$((\mathrm{ND})_m)_{m\geq1}$ says exactly that the contraction map
$\Psi:\bigwedge^{n-1}V\to V^*$ has zero kernel.  By Lemma
\ref{L:Non-DegGenChar}(\ref{Enu:Non-Deg}), this is precisely
non-degeneracy.
\end{proof}

The following is the main result of the section:
\begin{prop}\label{prop: backandforth} The set of isomorphisms between finitely generated $\mathcal{L}_{\theta,f}^K$-substructures of any two $\omega$-saturated models of $T$ is non-empty and has the back-and-forth property.
\end{prop}

Before proving it, we first show some auxiliary lemmas for a model $\mathcal V=(V,K)$ of $T$ and an $\mathcal{L}_{\theta,f}^K$-substructure $\mathcal A=(A_V,A_K)$ of $\mathcal V$ that are needed in the proof of the proposition. Note that, because we use the field language (rather than the language of rings) on the $K$-sort, the $K$-sort of every substructure is a field.

\begin{lemma}\label{lem QE1}
Let   $S$ be a  maximal $K$-linearly independent subset of $A_V$. Then every $v \in A_V$ is an $A_K$-linear combination of elements in $S$.
\end{lemma}

\begin{proof}
Let $v \in A_V$. Then $v \in \Span_K(S)$ by maximality. Thus 
$$ v = \lambda_1 v_1 + \dots +\lambda_m v_m$$
for some $\lambda_1, \dots, \lambda_m \in K$ and $v_1, \ldots, v_m \in S$. Since $\mathcal V \models  \theta_m( v_1, \dots, v_m)$, we have that $f_i^m (v,v_1, \dots, v_m) = \lambda_i$ for all $i \in \set{m}$. As $\mathcal A$ is a substructure, we can conclude that $\lambda_i \in A_K$.
\end{proof}

\begin{lemma}\label{lem QE2} 
	Let $U \subseteq  V$ be a $K$-linearly independent set, say $U = \{v_i : i < \kappa \}$, and let  $L\subseteq K$ be a  subfield containing \[\left\{\left \langle v_{i_1}, \dots, v_{i_n} \right \rangle_n : 0\leq i_1 <  \dots <  i_{n} < \kappa \right\}.\]  
	Then  the substructure generated by $U \cup L $ is equal to $(\Span_{L}(U) , L)$.
\end{lemma}
\begin{proof}
First note that $(\Span_{L}(U), L)$ is clearly contained in the substructure generated by $U \cup L $.

For the other inclusion, clearly $(\Span_L(U),L)$ contains $U\cup L$ and is closed under $+_V$, $\cdot_V$, $+_K$, $-_K$, $\cdot_K$, and ${}^{-1}_K$. Hence it remains to show that $(\Span_L(U),L)$ is closed under the maps $f_i^p$ for every $p\in\omega$ and $i\in\set{p}$, as well as under the map $\langle - , \dots, - \rangle_n$.

\noindent \underline{Closure under $\Bln$.}  Let $w_1, \ldots, w_n \in \Span_{ L}( v_{t_0}, \dots, v_{t_m})$ for some $m \in \omega$ and $t_0 < \ldots < t_m < \kappa$ and $\mu_{i,j} \in L$  be such that $$w_j = \sum_{i=0}^m \mu_{i,j} v_{t_i}. $$
Now
\begin{gather*}\langle w_1,\dots, w_n \rangle_n = \left\langle \sum_{i=0}^m \mu_{i,1} v_{t_i}, \dots , \sum_{i=0}^m \mu_{i,n} v_{t_i} \right\rangle_n\\
	= \sum_{0 \leq i_1< \dots < i_n\leq m}  \left( \Bigg(\sum_{\sigma \in \Sym(\set{n})} \sign (\sigma) \mu_{i_1, \sigma(1)} \dots \mu_{i_n, \sigma(n)}\Bigg)\langle v_{t_{i_1}}, \dots ,v_{t_{i_n}} \rangle_n  \right).
\end{gather*}
As  $ \langle v_{t_{i_1}}, \dots ,v_{t_{i_n}} \rangle_n  \in L$ for all $0 \leq i_1< \dots < i_n \leq m$ by assumption, we can conclude that $\langle w_1,\dots, w_n \rangle_n \in L$.
\vspace{12pt}

\noindent \underline{Closure under the $f_i^p$.}
Let $w, w_1, \dots, w_p \in \Span_L(v_{t_0}, \dots, v_{t_m})$ for some $t_0 < \ldots < t_m < \kappa$. For each $i\in\set{p}$, let $\lambda_i := f_i^p(w;w_1, \dots, w_p)$. There are coefficients
\[
\mu_i,\nu_{i,j}\in L
\qquad(0\leq i\leq m,\ 1\leq j\leq p)
\]
such that
  \begin{equation}\label{eqn:defww1...wp}
  w = \sum_{i=0}^m \mu_i v_{t_i}  \quad \mbox{and}\quad  w_j = \sum_{i=0}^m \nu_{i,j}v_{t_i}.
  \end{equation}	
  
We may assume that $\mathcal{V} \models \theta_p(w_1, \dots, w_p)$ and $w \in \Span_K(w_1, \dots, w_p)$, as otherwise all $f_i^p$ give zero. In particular, we have that $p\leq m+1$.

By definition of $f_i^p$ we have that
  \begin{equation}\label{eqn:ww1...wp}
  	w= \lambda_1 w_1+ \dots +\lambda_p w_p,
  \end{equation}
and moreover these $\lambda_i$'s are the unique solution in $K$ to this equation. Replacing $w$ and the $w_i$'s with the corresponding expressions in \eqref{eqn:defww1...wp}, \eqref{eqn:ww1...wp} is equivalent to
$$ \left(-\mu_0+ \sum_{i=1}^p \lambda_{i} \nu_{0,i} \right) v_{t_0} + \dots +  \left(-\mu_m+ \sum_{i=1}^p \lambda_{i} \nu_{m,i} \right) v_{t_m} =0.$$
Since the tuple $(v_{t_0}, \dots, v_{t_m} )$ is linearly independent, this holds if and only if
$$-\mu_j+ \sum_{i=1}^p \lambda_{i} \nu_{j,i}=0 \ \  \mbox{ for all } 0\leq j\leq m, $$
equivalently if 
$$M  \begin{pmatrix} \lambda_1 \\ \vdots \\ \lambda_p \end{pmatrix} =  \begin{pmatrix} \mu_0 \\ \vdots \\ \mu_m \end{pmatrix} \mbox{, where }(M)_{ji} = \nu_{j,i}.
$$
Since the system has a unique solution, $M$ has rank $p$. Choose linearly independent rows with indices $l_1,\ldots,l_p$. Then
\[
N\begin{pmatrix}x_1\\ \vdots\\ x_p\end{pmatrix}
=\begin{pmatrix}\mu_{l_1}\\ \vdots\\ \mu_{l_p}\end{pmatrix},
\qquad N_{ji}=\nu_{l_j,i},
\]
has $(\lambda_1,\ldots,\lambda_p)$ as its unique solution. Thus $N$ is invertible. Since all entries of $N$ and all $\mu_i$ belong to $L$, we obtain $\lambda_i\in L$ for every $i\in\set{p}$.
%
\end{proof}

Additionally, let $\mathcal W=(W,L)$ be another model of $T$ and let $\mathcal B=(B_V,B_K)$ be a substructure of $\mathcal W$.
\begin{lemma}\label{lem QE3}
Let $\eta:(A_V,A_K)\rightarrow(B_V,B_K)$ be a bijective map preserving $+_V$ and scalar multiplication, and suppose that $\eta\upharpoonright A_K:A_K\to B_K$ is a field isomorphism. Then $\eta$ preserves $\theta_p$ and $f_i^p$ for every $p\in\omega$ and $i\in\set{p}$.
\end{lemma}

\begin{proof}
By preserving $+_V$ and scalar multiplication, we have that $\eta(0_V) = 0_W$.

\noindent \underline{Preservation of $\theta_p$.} Let $w_1, \dots, w_p  \in A_V $ be such that 
 $\mathcal{V} \models \theta_p (w_1, \dots, w_p).$ We want to show that $\mathcal{W} \models \theta_p(\eta(w_1), \dots, \eta(w_p))$, in other words that $ (\eta(w_1), \dots, \eta(w_p))$ is $L$-linearly independent. Assume otherwise and let $m \in \set{p-1}$ be minimal such that $\eta(w_{m+1}) \in \Span_L \left( \left\{ \eta(w_i): i \in \set{m} \right\}\right)$. Then 
 \[\eta(w_{m+1}) = \sum_{i=1}^{m}\underbrace{ f_i^m(\eta(w_{m+1});\eta(w_1), \dots, \eta(w_m) )}_{=: \mu_i}\eta(w_{i}).   \]
As $\mathcal{B}$ is closed under the $f_i^m$, we have $\mu_1, \dots, \mu_m \in B_K$.
As $\eta\upharpoonright A_K$ is bijective, there are $\lambda_1, \dots, \lambda_m \in A_K$ such that $\mu_i=\eta(\lambda_i)$. Then preservation of $+_V$ and scalar multiplication, together with the fact that $\eta\upharpoonright A_K$ is a field isomorphism, gives:
\begin{align*}
	& \mu_1 \eta(w_1)+ \dots+ \mu_m \eta(w_m) - \eta(w_{m+1})= 0_W \\
	\Leftrightarrow\quad & \eta(\lambda_1) \eta(w_1)+ \dots+ \eta(\lambda_m)\eta(w_m)- \eta(w_{m+1}) = 0_W\\
	\Leftrightarrow \ \ \ &   \eta \left( \lambda_1 w_1 + \dots +\lambda_m w_m - w_{m+1}\right)=0_W  \\
	\Leftrightarrow \ \ \ &   \lambda_1 w_1 + \dots +\lambda_m w_m - w_{m+1}=0_V  \ \ \ (\mbox{as } \eta^{-1}(0_W)= 0_V)
	\end{align*}
--- a contradiction to $\mathcal{V} \models \theta_p(w_1, \dots, w_p)$. Applying the same argument to $\eta^{-1}$ shows that linear independence is preserved in both directions.
\vspace{12pt}

\noindent \underline{Preservation of $f_i^p$.} Let $w, w_1, \dots, w_p \in  A_V$.
We want to show that
$$ \eta( f_i^p( w, w_1, \dots, w_p)) = f_i^p( \eta(w),\eta( w_1), \dots, \eta(w_p)).$$
By the first part of the proof we already know that $(w_1, \dots, w_p )$  is $K$-linearly independent if and only if $ (\eta( w_1), \dots, \eta(w_p)) $  is $L$-linearly  independent, and similarly for  $(w, w_1, \dots, w_p )$ and $ (\eta(w),\eta( w_1), \dots, \eta(w_p)) $.  Hence we may assume that $(w_1, \dots, w_p )$ and  $ (\eta( w_1), \dots, \eta(w_p)) $ are linearly independent as well as $w \in \Span_K( w_1, \dots, w_p)$ and $\eta(w) \in \Span_L(\eta( w_1), \dots, \eta(w_p))$, as otherwise the values of the $f_i^p$ are all $0$.

So suppose that $\nu_1, \dots, \nu_p \in L$ are such that 
$$\eta(w)= \nu_1 \eta(w_1)+ \dots+ \nu_p \eta(w_p),  $$
or, equivalently, $\mathcal{W}\models f_i^p(\eta(w);\eta(w_1),\dots,\eta(w_p))=\nu_i$. Since $(B_V,B_K)$ is a substructure of $\mathcal{W}$, we have $\nu_i\in B_K$ for all $i\in\set{p}$. As $\eta$ is bijective and preserves $+_V$ and scalar multiplication, we conclude that
$$w= \mu_1w_1+ \dots+ \mu_pw_p $$
for $\mu_1, \dots, \mu_p \in A_K$ such that $\eta(\mu_i)=\nu_i$. In other words, 
$$\eta( f_i^p( w,w_1, \dots, w_p))=\eta( \mu_i) = \nu_i= f_i^p( \eta(w),\eta( w_1), \dots, \eta(w_p))$$
for all $i \in \set{p}$.
\end{proof}

We are now ready to prove the main proposition:

\begin{proof}[Proof of Proposition \ref{prop: backandforth}] Let $\mathcal V =(V, K )$ and $\mathcal W=(W,L)$ be two $\omega$-saturated models of $T$. 

Let $P$ and $Q$ be the prime fields of $K$ and $L$, respectively. Then $\left(\{0_V\},P\right)$ and $\left(\{0_W\},Q\right)$ are isomorphic substructures of $\mathcal V=(V,K)$ and $\mathcal W=(W,L)$, respectively, and therefore the set of partial isomorphisms is non-empty. 

Now, consider an  $\mathcal{L}^{K}_{\theta,f}$-isomorphism  $g : \left(\bar a_V, \bar a_K \right)  \rightarrow \left(\bar b_V, \bar b_K \right)$ between two finitely generated substructures  $\mathcal A = (\bar a_V, \bar a_K )$ of $\mathcal V$ and $\mathcal B =(\bar b_V, \bar b_K)$ of $\mathcal W$. Note that, as we work in the field language, $\bar a_K$ and $\bar b_K$ are fields. Moreover, $\bar a_V$ and $\bar b_V$ are finite-dimensional vector spaces over $\bar a_K$ and $\bar b_K$, respectively. Indeed, if $\mathcal A$ is finitely generated by  $v_1, \dots, v_m \in  V$ (we may assume these are linearly independent) and $k_0, \dots, k_\ell \in K$, as $\mathcal A$ is an $\mathcal{L}^{K}_{\theta,f}$-substructure, the subfield $\bar a_K$ contains \[\left\{\left \langle v_{i_1}, \dots, v_{i_n} \right \rangle_n : 1\leq i_1 <  \dots <  i_{n}\leq m \right\}.\] 
	Thus Lemma \ref{lem QE2} ensures that
\[
\mathcal A=(\Span_{\bar a_K}(v_1,\dots,v_m),\bar a_K),
\]
so $\bar a_V$ is a finite-dimensional vector space over $\bar a_K$.

For the rest of the proof, fix a maximal $K$-linearly independent subset $S=\{v_1,\dots,v_m\}$ of $\bar a_V$. Then $g(S):=\{g(v_1),\dots,g(v_m)\}$ is a maximal $L$-linearly independent subset of $\bar b_V$, since $g$ preserves $\theta_m$ and $\theta_{m+1}$ in both directions.
Fix an arbitrary $x\in V\cup K$. We need to extend $g$ to a partial isomorphism between finitely generated substructures of $\mathcal V$ and $\mathcal W$ whose domain contains $x$.

\textbf{Suppose first that $x \in K$.} 

Choose a finite tuple $\bar c$ generating the $\mathcal L^K$-substructure $\bar a_K$, and put $\bar d:=g(\bar c)$. Quantifier elimination for $\Th(K)$ implies that the quantifier-free type of $x$ over $\bar c$, transported via $g$, is a complete type over $\bar d$. It is finitely satisfiable in $L$, and hence, by $\omega$-saturation, is realized by some $x'\in L$. The map $g\upharpoonright\bar a_K$ therefore extends to an $\mathcal L^K$-isomorphism $h:\bar a_K'\to\bar b_K'$ between finitely generated substructures with $\bar a_K\cup\{x\}\subseteq\bar a_K'\leq K$ and $\bar b_K\cup\{x'\}\subseteq\bar b_K'\leq L$. 
As $\mathcal A$ and $\mathcal B$ are substructures, we have 
\begin{gather*}
	\left\{ \langle v_{i_1}, \dots, v_{i_n} \rangle_n : 1\leq i_1 <  \dots <  i_{n}\leq m \right\}\subseteq \bar a_K \subseteq \bar a_K'  \textrm{ and}\\
	\left\{ \left \langle g(v_{i_1}), \dots, g(v_{i_n}) \right \rangle_n : 1\leq i_1 <  \dots <  i_{n}\leq m \right\}\subseteq \bar b_K \subseteq \bar b_K'.
\end{gather*}
Thus Lemma \ref{lem QE2} ensures that $ \left(\Span_{\bar a_K'}(\bar a_V) , \bar a_K' \right) $ is a substructure of $\mathcal V$ generated by $ \bar a_V,\bar a_K'$, and $\left(\Span_{\bar b_K'} \left(\bar b_V \right), \bar b_K' \right)$  is a substructure of $\mathcal W$ generated by $\bar b_V, \bar b_K' $.
Using Lemma \ref{lem QE1}, we have
 \[ \left(\Span_{\bar a_K'}(\bar a_V) ,\bar a_K' \right) = \left(\Span_{\bar a_K'}(v_1, \dots, v_m) , \bar a_K' \right), \] 
and similarly
\[ \left( \Span_{\bar b_K'} \left(\bar b_V \right) , \bar b_K' \right) = \left(\Span_{\bar b_K'} \left(g(v_1), \dots, g(v_m) \right) , \bar b_K' \right). \] 
 So we want to extend the partial $\mathcal{L}^{K}_{\theta,f}$-isomorphism $g$ to these larger finitely generated substructures.
 \begin{claim}\label{cla: multilin QE proof 1}
The map $\eta: \left(\Span_{\bar a_K'}(S), \bar a_K' \right)  \rightarrow \left(\Span_{\bar b_K'}(g(S)), \bar b_K' \right)$ that maps any
\begin{gather*}
	v =\lambda_1 v_1 + \dots + \lambda_m v_m \in \Span_{\bar a_K'}(S)  \mbox{ with } \lambda_1, \dots, \lambda_m \in \bar a_K'\\
	\textrm{ to } h(\lambda_1)  g(v_1) +  \dots + h(\lambda_m) g(v_m)
\end{gather*}
and any $\lambda\in \bar a_K'$ to $ h(\lambda)$ is an isomorphism of $\mathcal{L}_{\theta,f}^K $-structures.
\end{claim}
\begin{proof}
Note first that any element of $\Span_{\bar a_K'}(S)$ is a unique linear combination of elements in the $K$-linearly independent set $S$. Thus the map is well defined. Moreover, by $L$-linear independence of $g(S)$ and bijectivity of $g$ and $h$, we have that $\eta$ is bijective. In particular,  only $0_V$ gets mapped to $0_W$ by $\eta$.

\noindent \underline{$\eta$ preserves $+_V$ and $\cdot_V$.} Let $\ell\in\bar a_K'$ and write $a=\sum_{i=1}^m\lambda_iv_i$, $a'=\sum_{i=1}^m\mu_iv_i$ 
for arbitrary $a,a'\in\Span_{\bar a_K'}(\bar a_V)$. Then
\begin{gather*}
	\eta(\ell a+a')
 =\eta\left(\sum_{i=1}^m(\ell\lambda_i+\mu_i)v_i\right) 
 =\sum_{i=1}^m h(\ell\lambda_i+\mu_i)g(v_i)\\
 =h(\ell)\sum_{i=1}^m h(\lambda_i)g(v_i)
   +\sum_{i=1}^m h(\mu_i)g(v_i)
 =\eta(\ell)\eta(a)+\eta(a').
\end{gather*}

\noindent \underline{$\eta$ preserves $\theta_p$ and $f_i^p$.} Since $h$ is a field isomorphism and $\eta$ preserves vector addition and scalar multiplication, this follows from Lemma \ref{lem QE3}.

\noindent \underline{$\eta $ preserves $  \langle -, \dots, - \rangle_n$.} 
Let $w_1, \ldots, w_n \in \Span_{\bar a_K'}(\bar a_V)$ be given. For $i\in\set{m}$ and $j\in\set{n}$, fix $\mu_{i,j} \in \bar a_K'$ such that $w_j = \sum_{i=1}^m \mu_{i,j} v_i$.
Then
\begin{align*} \left \langle \eta(w_1), \dots, \eta( w_n) \right \rangle_n&= \left\langle \eta\left( \sum_{i=1}^m \mu_{i,1} v_i\right), \dots,  \eta\left(  \sum_{i=1}^m \mu_{i,n} v_i\right)\right\rangle_n\\
	&= \left\langle \sum_{i=1}^m h(\mu_{i,1})g( v_i), \dots,  \sum_{i=1}^m h( \mu_{i,n}) g(v_i) \right\rangle_n\\
	&= \sum_{i_1=1}^m \dots  \sum_{i_n=1}^m h(\mu_{i_1,1})  \dots h( \mu_{i_n,n}) \left\langle g(v_{i_1}), \dots, g(v_{i_n}) \right  \rangle_n\\
		&= \sum_{i_1=1}^m \dots  \sum_{i_n=1}^m h(\mu_{i_1,1})  \dots h ( \mu_{i_n,n})g \big(\langle v_{i_1}, \dots, v_{i_n}  \rangle_n \big)\\
	&= \sum_{i_1=1}^m \dots  \sum_{i_n=1}^m h \big(\mu_{i_1,1})  \dots h ( \mu_{i_n,n}) h \big(\langle v_{i_1}, \dots, v_{i_n}   \rangle_n \big)\\
	&= h\left(\sum_{i_1=1}^m \dots \sum_{i_n=1}^m \mu_{i_1,1} \dots   \mu_{i_n,n} \langle v_{i_1}, \dots, v_{i_n}  \rangle_n\right)\\
	&= \eta\left(\left \langle \sum_{i=1}^m  \mu_{i,1}v_i, \dots, \sum_{i=1}^m   \mu_{i,n} v_i \right \rangle_n \,\right) 
		= \eta\left(\langle w_1, \dots,w_n  \rangle_n\right). \qedhere
\end{align*}
\end{proof}
\noindent This finishes the case in which the element $x$ that we wanted to add to the domain of our partial isomorphism belongs to the field.

\textbf{Now suppose that $x=v \in V$.}

\noindent \textbf{Case 1: }$v \in \Span_K(\bar a_V)$.

\noindent Let $w_1, \dots, w_r \in \bar a_V$ and $\lambda_1, \dots, \lambda_r \in K$ be such that
$ v = \lambda_1 w_1 + \dots + \lambda_r w_r.$
Applying Claim \ref{cla: multilin QE proof 1} repeatedly, we can add $\lambda_1, \dots, \lambda_r$ to the field part. Let $\bar a_K'$ be the field substructure generated by $\bar a_K\cup\{\lambda_1,\ldots,\lambda_r\}$, and let $\bar b_K'$ be its image under the resulting extension of $g$. We obtain an isomorphism between the substructures generated by $\bar a_V, \bar a_K'$ and $\bar b_V, \bar b_K'$; moreover, $v \in \Span_{\bar a_K'}(\bar a_V)$, as witnessed by $\lambda_1, \dots, \lambda_r$. Thus $v$ belongs to the generated substructure, and we have obtained the desired map.

\noindent \textbf{Case 2: } $v \not\in \Span_K(\bar a_V)$.

\noindent  Applying Claim \ref{cla: multilin QE proof 1} repeatedly, we may assume that the elements $\{ \langle v_{i_1}, \dots, v_{i_{n-1}},v \rangle_n : 1\leq  i_1 < \dots < i_{n-1} \leq m \}$
 belong to $\bar a_K$.
  Since $\left( W, \Bln \right)$ is non-degenerate and the set
$ \{t_{i_1, \dots, i_{n-1}} := g(v_{i_1}) \wedge \dots \wedge g(v_{i_{n-1}}): 1\leq i_1 <\dots <i_{n-1}\leq m \}$ 
 is linearly independent in $\bigwedge^{n-1}W$, by Lemma \ref{L:findingw} we can find $w \in W$ which is $L$-linearly independent from $   \Span_L \left(\bar b_V \right)$ and such that $ \langle g(v_{i_1}), \dots, g(v_{i_{n-1}}), w \rangle_n = \left \langle t_{i_1, \dots, i_{n-1}}, w \right \rangle_2 = g \left(\langle v_{i_1}, \dots, v_{i_{n-1}}, v \rangle_n \right)$ 
for all $1 \leq i_1 < \ldots < i_{n-1} \leq m$. Now, using that \Bn is alternating and $g$ is a partial isomorphism, it is easy to check that, for any $u_1, \dots, u_{n-1} \in \bar a_V$ and every $r\in\set{n}$,
\begin{equation}\label{E:preservingNlinform}g\left(\langle u_1, \dots, u_{r-1}, v, u_r, \dots, u_{n-1} \rangle_n \right) = \langle g(u_1), \dots,g(u_{r-1}), w, g(u_r), \dots, g(u_{n-1}) \rangle_n.
\end{equation}
Indeed, for any $i_1,\dots,i_{n-1}\in\set{m}$ we have $
\langle g(v_{i_1}),\dots,g(v_{i_{n-1}}),w\rangle_n
=g\left(\langle v_{i_1},\dots,v_{i_{n-1}},v\rangle_n\right)$ 
(if two coincide then both sides are $0$, and otherwise we can use the permutation that orders the $v_i$'s in ascending order). Then for any $u_1= \sum_{i=1}^m \mu_{i,1}v_i\,,\ \dots,\ u_{n-1}=\sum_{i=1}^m \mu_{i,n-1}v_i$ in $\bar{a}_V$ we have 
\begin{gather*}
g\left(\langle u_1, \dots, u_{r-1}, v, u_r, \dots, u_{n-1} \rangle_n \right) = \sign(\sigma) g\left(\langle u_1, \dots, u_{n-1},v \rangle_n \right)\\
= \sign(\sigma) g\left(\left \langle \sum_{i=1}^m \mu_{i,1}v_i, \dots, \sum_{i=1}^m \mu_{i,n-1}v_i,v \right\rangle_n \right)\\
= \sign(\sigma) g\left(\sum_{i_1=1}^m \dots \sum_{i_{n-1}=1}^m \prod_{j=1}^{n-1} \mu_{i_j,j}  \left\langle v_{i_1}, \dots, v_{i_{n-1}},v \right\rangle_n \right)\\
= \sign(\sigma) \left(\sum_{i_1=1}^m \dots \sum_{i_{n-1}=1}^m \prod_{j=1}^{n-1} g(\mu_{i_j,j}) g \left( \left \langle v_{i_1}, \dots, v_{i_{n-1}},v \right\rangle_n \right)\right)\\
= \sign(\sigma) \left(\sum_{i_1=1}^m \dots \sum_{i_{n-1}=1}^m \prod_{j=1}^{n-1} g(\mu_{i_j,j})  \left \langle g(v_{i_1}), \dots, g(v_{i_{n-1}}),w \right\rangle_n \right)\\
= \sign(\sigma) \left \langle \sum_{i=1}^m g(\mu_{i,1}) g(v_i), \dots, \sum_{i=1}^m g(\mu_{i,n-1}) g(v_i),w \right\rangle_n \\
= \sign(\sigma) \left \langle g\left (\sum_{i=1}^m \mu_{i,1} v_i\right ), \dots, g\left (\sum_{i=1}^m \mu_{i,n-1}v_i\right),w \right\rangle_n \\
 =\sign(\sigma) \langle g(u_1), \dots, g(u_{n-1}),w \rangle_n 
 = \langle g(u_1), \dots,g(u_{r-1}), w, g(u_r), \dots, g(u_{n-1}) \rangle_n,
\end{gather*}
where $\sigma$ is the permutation that sends $(u_1, \dots, u_{r-1},v,u_r, \dots, u_{n-1})$ to $(u_1, \dots, u_{n-1},v)$.

As in the first part of this proof, Lemmas \ref{lem QE1} and \ref{lem QE2} yield:
\begin{itemize}
	\item the substructure generated by $ v,  \bar  a_V, \bar a_K$ is equal to 
	 \[\left(\Span_{\bar a_K}(S,v) , \bar a_K \right) = \left(\bar a_V \oplus \Span_{\bar a_K}(v), \bar a_K \right);\]
	\item the substructure generated by $w,  \bar  b_V, \bar b_K$ is equal to
	\[\left(\Span_{\bar b_K}(g(S),w) , \bar b_K \right) = \left(\bar b_V \oplus \Span_{\bar b_K}(w), \bar b_K \right).\]
\end{itemize}

\begin{claim}
	The map 
	\begin{align*}
		\eta: \left(\bar a_V \oplus \Span_{\bar a_K}(v), \bar a_K \right) &\rightarrow \left(\bar b_V \oplus \Span_{\bar b_K}(w) , \bar b_K \right)\\[1ex]
		(u  + \lambda v, \mu) &\mapsto  ( g( u)+ g( \lambda) w, g(\mu))
	\end{align*} 
 is an isomorphism of $\mathcal{L}_{\theta,f}^K $-structures.
\end{claim}
\begin{proof}
Since $g$ is bijective, we have that $\eta$ is bijective as well. In particular,  only $0_V$ gets mapped to $0_W$ by $\eta$. 

\noindent \underline{$\eta$ preserves $+_V$ and $\cdot_V$.} Let $\ell \in \bar a_K$, $ u_1 + \lambda_1 v, u_2 + \lambda_2 v\in  \bar a_V \oplus \Span_{\bar a_K}(v)$. Then  
\begin{gather*}
	\eta(\ell(u_1 + \lambda_1 v  ) + (u_2 + \lambda_2 v) )
	  = \  \eta((\ell u_1+u_2)+ (\ell \lambda_1 +\lambda_2) v  )\\
	= \   g(\ell u_1+u_2 ) +g(\ell \lambda_1 +\lambda_2)  w 
	= \ g(\ell)g(u_1)+ g(u_2)+ g(\ell) g(\lambda_1)  w +   g(\lambda_2)  w  \\
	=\  g(\ell) \left(g(u_1)+ g(\lambda_1) w  \right) + (g( u_2) +g(\lambda_2) w )
	=\  \eta(\ell) \eta(u_1+ \lambda_1 v  ) + \eta( u_2 +\lambda_2 v ).
\end{gather*}

\noindent \underline{$\eta$ preserves $\theta_p$ and $f_i^p$.} The map $\eta$ preserves $\theta_p$ and $f_i^p$ by linearity and Lemma \ref{lem QE3}.

\noindent \underline{$\eta$ preserves $  \langle - , \dots, -\rangle_n$.} Let $w_1, \ldots, w_n \in \bar a_V \oplus \Span_{\bar a_K}(v)$. Let $\mu_1, \dots, \mu_n \in \bar a_K$ and $u_1, \dots, u_n \in  \bar a_V$ be such that $w_j =   u_j + \mu_j v $
for all $j$. Then
\begin{gather*}
	\langle \eta(w_1), \dots, \eta( w_n) \rangle_n 
	= \ \left\langle \eta( u_1+ \mu_1 v ), \dots,  \eta( u_n +\mu_n v )\right\rangle_n\\
	= \ \left\langle  g(u_1)+ g(\mu_1)w , \dots,  g(u_n) + g(\mu_n)w \right\rangle_n\\
	= \  \sum_{i=1}^n  g(\mu_i)\Big \langle g(u_1),\dots, g(u_{i-1}) , w , g(u_{i+1}), \dots,  g(u_n) \Big\rangle_n 
	 +\langle g(u_1),\dots,g(u_n) \rangle_n \quad \\\textrm{(as $\Bln$  is alternating)} \\
	= \  \sum_{i=1}^n  g(\mu_i)g\Big( \Big \langle u_1,\dots, u_{i-1} , v , u_{i+1}, \dots, u_n \Big\rangle_n\Big)  +g(\langle u_1,\dots,u_n \rangle_n) \\
	\textrm{(as $g$ preserves $\langle-, \ldots, - \rangle_n$ and \eqref{E:preservingNlinform})}\\
	= \ g\Big( \sum_{i=1}^n  \mu_i \Big \langle u_1,\dots, u_{i-1} , v , u_{i+1}, \dots, u_n \Big\rangle_n + \langle u_1,\dots,u_n \rangle_n \Big) \\
	\textrm{(as $g\upharpoonright \bar{a}_K$ is a field isomorphism)}\\
	= \ g( \langle u_1 +\mu_1 v, \dots, u_n + \mu_n v \rangle_n) 
	\textrm{ (as $\Bln$  is alternating)} \\
	= \ g( \langle w_1, \dots, w_n \rangle_n) 
	= \ \eta( \langle w_1, \dots, w_n \rangle_n). \qedhere
\end{gather*}
\end{proof}
This concludes the proof of Proposition \ref{prop: backandforth}.
\end{proof}

Proposition \ref{prop: backandforth} implies immediately:
\begin{theorem}\label{thm: QE for multilinear forms}
		For any $n \geq 2$ and field $K$, the theory $\prescript{}{\Alt}T^K_n$ of infinite-dimensional alternating non-degenerate $n$-linear spaces over $K$ has quantifier elimination (in the language $\mathcal{L}^{K}_{\theta,f}$) and is complete.
\end{theorem} 
\begin{remark}\label{rem: VS complete and omega-cat}
\begin{enumerate}
\item Let $T^K_{\VS,\infty}$ denote the two-sorted theory of infinite-dimensional vector spaces over fields elementarily equivalent to $K$, in the language
\[
\mathcal L^K_{\VS}:=\mathcal L^K_{\theta,f}\setminus\{\Bln\}.
\]
Our proof of Theorem \ref{thm: QE for multilinear forms} demonstrates, in particular, that $T^K_{\VS,\infty}$ is complete and has quantifier elimination in $\mathcal L^K_{\VS}$.
	\item The proof also demonstrates that if the field $K$ is finite, then the theory $\prescript{}{\Alt}T^K_n$ is $\omega$-categorical. Indeed, the proof builds an isomorphism between any two countable models of $\prescript{}{\Alt}T^K_n$ by back--and--forth, starting with an isomorphism of their corresponding finite fields and using that in this case finitely generated structures are finite.
\end{enumerate}
\end{remark}
%

\begin{prop}\label{prop: multilinear not pseudofinite}
For every finite field $K$ and every $n\geq2$, the theory
$\prescript{}{\Alt}T_n^K$ is pseudofinite if and only if $n=2$.
\end{prop}

\begin{proof}
Assume first that $n \geq 3$. Every model of \(\prescript{}{\Alt}T^K_n\) satisfies the sentence 
\begin{equation}\label{eq: no singular n-1 space}
 \forall v_1,\ldots,v_{n-1}
 \left(
 \theta_{n-1}(v_1,\ldots,v_{n-1})
 \ \longrightarrow\
 \exists w\ 
 \langle v_1,\ldots,v_{n-1},w\rangle_n\neq0
 \right).
\end{equation}
Indeed, for an independent tuple \(v_1,\ldots,v_{n-1}\), the element
\(v_1\wedge\cdots\wedge v_{n-1}\) of
\(\bigwedge^{n-1}V\) is nonzero, so \eqref{eq: no singular n-1 space}
follows directly from non-degeneracy as defined in Definition
\ref{def: non-degen}.

By the Chevalley--Warning theorem
\cite{chevalley1935demonstration,warning1935bemerkung}, every finite
field is quasi-algebraically closed.  Draisma and Shaw
\cite[Theorem~1.1, applied with $e=n$]{draisma2010singular} prove that, over any such
field, every alternating \(n\)-form on a finite-dimensional vector
space of dimension strictly greater than \(n\) has a singular
\((n-1)\)-dimensional subspace $U$, i.e.~if $v_1, \ldots, v_{n-1}$ is a basis for $U$, then 
$\langle v_1,\ldots,v_{n-1},w\rangle_n=0$
for all $w \in V$ --- so the sentence in \eqref{eq: no singular n-1 space} does not hold. 
To make the pseudofiniteness obstruction a single sentence, take the
conjunction of \eqref{eq: no singular n-1 space}, the assertion that
\(V\) contains \(n+1\) linearly independent vectors, namely
\[
 \exists u_1,\ldots,u_{n+1}\ \theta_{n+1}(u_1,\ldots,u_{n+1}),
\]
and the finite collection of axioms saying that the field sort is exactly
\(K\), that \(V\) is a vector space over it, that \(\theta_{n-1}\) and
\(\theta_{n+1}\) express linear independence, and that
\(\langle- ,\ldots,-\rangle_n\) is alternating (in the usual
repeated-argument sense) and \(n\)-linear.  This
sentence belongs to \(\prescript{}{\Alt}T^K_n\), but the preceding
paragraph shows that it has no finite model.  Hence
\(\prescript{}{\Alt}T^K_n\) is not pseudofinite.

For $n=2$, in contrast, for a finite field \(K\), the bilinear theory
\(\prescript{}{\Alt}T^K_2\) is pseudofinite. Indeed, expand the
finite-dimensional symplectic \(K\)-spaces of dimension \(2m\) by
the canonical interpretations of the symbols in
\(\mathcal L^K_{\theta,f}\).  Any nonprincipal ultraproduct as
\(m\to\infty\) is, by \L o\'s's theorem, an infinite-dimensional non-degenerate alternating bilinear space and hence a model of \(\prescript{}{\Alt}T^K_2\). The completeness supplied by Theorem \ref{thm: QE for multilinear forms} identifies its complete theory with \(\prescript{}{\Alt}T^K_2\).
\end{proof}

\begin{problem}
	Here we focused on alternating multilinear forms. It remains interesting to extend this quantifier elimination/completeness result to other kinds of multilinear forms (and multilinear maps between different $K$-vector spaces), in particular in the symmetric case (some additional assumptions on the field $K$ are necessary, even in the bilinear case). See also \cite{harman2024ultrahomogeneous, bik2022isogeny, harman2024tensor, neretin2023oligomorphic}.
\end{problem}

\section{Composition Lemma for functions of arbitrary arity and NIP relations}\label{sec: Composition lemma}

\subsection{NIP formulas on indiscernible sequences and arrays}\label{sec: array shrink NIP}
Let $T$ be a complete theory in a language $\mathcal{L}$. We recall some standard definitions and facts.
\begin{defn}
	Given a formula $\varphi(x_1, \ldots, x_k)$ we say that it is NIP if every formula obtained from it by partitioning the variables into two groups is NIP.
\end{defn}

\begin{defn}\label{def: indiscernible array}
Given $k \in \mathbb{N}$, let $I_1, \ldots, I_k$ be linear orders, and let $\Delta$ be a set of formulas. We say that a $k$-dimensional array of tuples $\left( a_{\bar{i}} : \bar{i} \in I_1 \times \ldots \times I_k \right)$ from $\mathbb{M}$ is \emph{$\Delta$-indiscernible} if for any $\varphi(x_1, \ldots, x_n) \in \Delta$ and any $ \bar{i}^{\alpha}_t \in I_1 \times \ldots \times I_k$ for $\alpha \in \{0,1\}$ and $1 \leq t \leq n$, with $\bar{i}^\alpha_t = (i^\alpha_{t,1}, \ldots, i^\alpha_{t,k})$, we have 
\begin{align*}
&\bigwedge_{s=1}^{k}
\Bigl(\qftp_{<}(i^0_{1,s},\ldots,i^0_{n,s})
=\qftp_{<}(i^1_{1,s},\ldots,i^1_{n,s})\Bigr)\\
&\hspace{30mm}\Longrightarrow
\bigl(\models\varphi(a_{\bar i^0_1},\ldots,a_{\bar i^0_n})
\leftrightarrow\varphi(a_{\bar i^1_1},\ldots,a_{\bar i^1_n})\bigr).
\end{align*}
If $\Delta$ consists of all $\mathcal{L}$-formulas, then we simply say \emph{indiscernible}.
\end{defn}

\begin{remark}
		Indiscernible arrays were studied in various contexts, including Zilber's trichotomy in Zariski geometries \cite{hrushovski1996zariski, levi2015indiscernible}, dividing in simple theories \cite{kim1996simple} and NTP$_2$ theories \cite{yaacov2014independence}, simplicity in compact abstract theories \cite{ben2003simplicity}, and thorn-forking \cite{adler2009thorn}.
\end{remark}

The following fact is the so-called UDTFS property of NIP formulas (i.e.~\emph{Uniform Definability of Types over Finite Sets}), established for an arbitrary formula in an NIP structure in \cite{chernikov2015externally}, and more recently for an NIP formula in an arbitrary structure in \cite{eshel2021uniform}.

\begin{fact}\label{fac: UDTFS}
		Let $\varphi(x,y) \in \mathcal{L}$ be an NIP formula. Then there exists a formula $\theta(y;\bar{y}) \in \mathcal{L}$, with $\bar{y} = (y_1, \ldots, y_d)$ for some $d \in \mathbb{N}$, satisfying the following: for every finite set $A \subseteq \mathbb{M}^y$ with $|A| \geq 2$ and every $b \in \mathbb{M}^x$, there exists a tuple $c  \in A^d$ such that for every $a \in A$ we have $\models \varphi(b,a) \iff \models \theta(a, c)$.
\end{fact}
This implies that every subset of an indiscernible array relatively definable by an instance of an NIP formula is a union of boundedly many boxes (this is a generalization of the finitary version of the ``shrinking of indiscernibles'' (\cite{baldwin2000stability, shelah59classification}; see also \cite[Section 3]{adler2008introduction}) from sequences to arrays; the case $k=1$ appears as \cite[Lemma 3.1]{chernikov2021ramsey}).

\begin{lemma}\label{lem: fin shrinking}
	Let $\varphi(x;y_1, \ldots, y_r) \in \mathcal{L}$ be an NIP formula. Then there exist some $d' = d'(\varphi) \in \mathbb{N}$ and a finite set of $\mathcal{L}$-formulas $\Delta$ such that the following holds.  
	Let $k \in \mathbb{N}$, $b \in \mathbb{M}^x$ and a $k$-dimensional $\Delta$-indiscernible array $\left( a_{\bar{i}} : \bar{i} \in I_1 \times \ldots \times I_k\right)$, with $I_s$ finite for $1 \leq s \leq k$, be arbitrary. Then there exist some tuples $\bar{j}_1, \ldots, \bar{j}_{d'} \in I_1 \times \ldots \times I_k$ satisfying the following: for any $\bigl(\bar{i}^\alpha_t: \alpha\in\{0,1\},\ 1\leq t\leq r\bigr)$ elements of $I_1 \times \ldots \times I_k$, we have
\begin{align*}
&\bigwedge_{s=1}^{k}\Bigl(
\qftp_{<}(i^0_{1,s},\ldots,i^0_{r,s}/j_{1,s},\ldots,j_{d',s})
=\qftp_{<}(i^1_{1,s},\ldots,i^1_{r,s}/j_{1,s},\ldots,j_{d',s})
\Bigr)\\
&\hspace{25mm}\Longrightarrow
\bigl(\models\varphi(b;a_{\bar i^0_1},\ldots,a_{\bar i^0_r})
\leftrightarrow\varphi(b;a_{\bar i^1_1},\ldots,a_{\bar i^1_r})\bigr).
\end{align*}
	
\end{lemma}
\begin{proof}

Put $y' := (y_1, \ldots, y_r)$. By Fact \ref{fac: UDTFS}, applied to
$\varphi(x;y')$, choose a formula $\theta(y';\bar z)$, where
$\bar z=(z_1,\ldots,z_d)$ and each $z_q$ has the same arity as $y'$. Let
	$\Delta :=\{ \theta \}$. Let $b$ and the finite set
	$A := \{a_{\bar{i}}: \bar{i} \in I_1 \times \ldots \times I_k\}$ be
given. If $|A|=1$, the conclusion is immediate. Otherwise
$|A'|\geq2$ for $A':=A^r$, so applying Fact \ref{fac: UDTFS} to $b$
	and $A'$, there are tuples
	\[
	(\bar j^q_1,\ldots,\bar j^q_r)
	\in(I_1\times\cdots\times I_k)^r
	\qquad(1\leq q\leq d)
	\]
	such that for all $\bar{i}_1, \ldots, \bar{i}_r \in I_1\times \ldots \times I_k$ we have
		\begin{equation}\label{theta iff phi}
		\begin{aligned}
		\models \varphi(b;a_{\bar i_1},\ldots,a_{\bar i_r})
		\quad\Longleftrightarrow\quad
		\models \theta\bigl(&a_{\bar i_1},\ldots,a_{\bar i_r}; 
		(a_{\bar j_t^q}:1\leq q\leq d,\ 1\leq t\leq r)\bigr).
		\end{aligned}
		\end{equation}
	
		Let $\bar{j}_1, \ldots, \bar{j}_{d'}$ with $d' := rd$ enumerate the list $ \left( \bar{j}_t^q : 1 \leq q \leq d, 1 \leq t \leq r \right)$. Now consider any family $\bigl(\bar{i}^\alpha_t: \alpha\in\{0,1\},\ 1\leq t\leq r\bigr)$ of elements of $I_1 \times \ldots \times I_k$ such that
	$$\qftp_{<}\left( i^0_{1,s}, \ldots, i^0_{r,s} / j_{1,s}, \ldots, j_{d',s}\right) = \qftp_{<} \left( i^1_{1,s}, \ldots, i^1_{r,s} / j_{1,s}, \ldots, j_{d',s}\right)$$
	for all $1 \leq s \leq k$, by indiscernibility of the array (with respect to $\theta$) we have 
	$$\models \theta\left(a_{\bar{i}^0_1}, \ldots, a_{\bar{i}^0_r}; \left( a_{\bar{j}_t^q} : 1 \leq q \leq d, 1 \leq t \leq r \right) \right)$$
	$$ \iff \models \theta\left(a_{\bar{i}^1_1}, \ldots, a_{\bar{i}^1_r}; \left( a_{\bar{j}_t^q} : 1 \leq q \leq d, 1 \leq t \leq r \right) \right),$$
	hence, by \eqref{theta iff phi},
		\begin{gather*}
			\models \varphi(b; a_{\bar{i}^0_1}, \ldots, a_{\bar{i}^0_r}) \iff \models \varphi(b; a_{\bar{i}^1_1}, \ldots, a_{\bar{i}^1_r}).\qedhere
		\end{gather*}
\end{proof}
We will also use the following lemma on the existence of indiscernible subarrays in sufficiently large arrays.
\begin{lemma}\label{lem: indisc subarray}
	For every $k,n \in \mathbb{N}$ and a finite set of formulas $\Delta$ there exists some $r = r(k,\Delta,n)$ satisfying the following. If $\left(a_{\bar{i}} : \bar{i} \in [r]^k \right)$ is an arbitrary array of finite tuples of the same length, then there exists some $\Delta$-indiscernible subarray $(a_{\bar{i}} : \bar{i} \in I_1 \times \ldots \times I_k)$ with $I_t \subseteq [r]$ and $|I_t| \geq n$ for all $1 \leq t \leq k$.
\end{lemma}
\begin{proof}
	By a repeated application of the finite Ramsey theorem.
\end{proof}

\subsection{$n$-dependence and generalized indiscernibles}\label{sec: n-dep and gen indisc}

We recall some basic properties of $n$-dependent formulas and theories.

\begin{fact}\label{fac: props of n-dependent formulas} 
\begin{enumerate}
	\item Let $\varphi(x;y_1,\ldots,y_n)$ and $\psi(x;y_1,\ldots,y_n)$ be $n$-dependent. Then $\neg\varphi$, $\varphi\land\psi$, and $\varphi\lor\psi$ are $n$-dependent. See \cite[Proposition 6.5(1)]{chernikov2019n}.
	\item Let $\varphi(x,y_1, \ldots, y_n)$ be a formula. Suppose that $(w, z_1, \ldots, z_n)$ is any permutation of the tuple $(x,y_1, \ldots, y_n)$. Then $\psi(w, z_1, \ldots, z_n) := \varphi(x,y_1, \ldots, y_n)$ is $n$-dependent if and only if $\varphi(x,y_1, \ldots, y_n)$ is $n$-dependent. See \cite[Proposition 6.5(2)]{chernikov2019n}.
	\item A theory $T$ is $n$-dependent if and only if every formula $\varphi(x,y_1, \ldots, y_n)$ such that all but at most one of the tuples $x, y_1, \ldots, y_n$ are singletons is $n$-dependent. See \cite[Theorem 2.12]{chernikov2021n}.
	\end{enumerate}
	\end{fact}
We recall a characterization of $n$-dependence in terms of generalized indiscernible sequences from \cite{chernikov2019n}.

\begin{defn}
Fix a language $\mathcal{L}_{\opg}^n=
\{R_n(x_1,\ldots ,x_{n}),<, P_1(x),\ldots , P_{n}(x)\}$. 
	An \emph{ordered $n$-partite hypergraph} is an $\mathcal{L}^{n}_{\opg}$-structure $ \mathcal{A} = \left(A; <, R_n, P_1, \ldots , P_{n} \right)$ such that
\begin{enumerate}
\item
$A$ is the disjoint union $P^{\mathcal{A}}_1 \sqcup\ldots \sqcup P^{\mathcal{A}}_{n}$,
\item $R_n^{\mathcal{A}}$ is a symmetric relation so that if $(a_1,\ldots ,a_{n})\in R_n^{\mathcal{A}}$ then $P^{\mathcal{A}}_i\cap \{a_1, \ldots, a_{n}\}$ is a singleton for every $1 \leq i \leq n$,
\item
$<^{\mathcal{A}}$ is a linear ordering on $A$ with $P^{\mathcal{A}}_1 <\ldots <P^{\mathcal{A}}_{n}$.
\end{enumerate}

\end{defn}

\begin{fact} \label{fac: Gnp} \cite[Fact 4.4 + Remark 4.5]{chernikov2019n}
Let $\mathcal{K}$ be the class of all finite ordered $n$-partite hypergraphs. Then $\mathcal{K}$ is a Fra\"{i}ss\'e class, and its limit is called the \emph{generic ordered $n$-partite hypergraph}, denoted by $G_{n,p}$.  An ordered $n$-partite hypergraph 
$\mathcal{A}$ is a model of $\Th(G_{n,p})$ if and only if:
\begin{itemize}
\item
$(P^{\mathcal{A}}_i, <)$ is a dense linear order without endpoints for each $1 \leq i \leq n$,
\item
for every $1 \leq j \leq n$, finite disjoint sets 
$A_0,A_1\subset {\prod_{1 \leq i \leq n, i\neq j}P^{\mathcal{A}}_i}$
 and
$b_0<b_1\in P^{\mathcal{A}}_j$, there is some $b \in P^{\mathcal{A}}_j$ such that $b_0<b<b_1$, $R_n(b,\bar{a})$ holds for every $\bar{a} \in A_0$, and $\neg R_n(b, \bar{a})$ holds for every $\bar{a} \in A_1$.
\end{itemize}
\end{fact}

\begin{remark}\label{rem : induced copy}
		It is easy to see from the axiomatization that, given $G_{n,p}$ and non-empty intervals $I_t\subseteq P_t$ for $t=1,\ldots,n$, the induced substructure on $I_1\cup\cdots\cup I_n$ contains a copy of $G_{n,p}$ whose $t$th part lies in $I_t$.
\end{remark}

\noindent We denote by $O_{n,p}$ the reduct of $G_{n,p}$ to the language $\mathcal{L}^{n}_{\op} = \{<, P_1(x),\ldots , P_{n}(x)\}$.

\begin{defn}
Let $T$ be a theory in a language $\mathcal{L}$, and $\C$ a monster
model of $T$.
\begin{enumerate}
\item Let $I$ be a structure in the language $\mathcal{L}_0$.
We say that $\bar{a}=\left(a_{i}\right)_{i\in I}$, 
with $a_{i}$ a tuple in $\C$, is \emph{$I$-indiscernible} over a set of parameters $C \subseteq \C$ if for
all $n\in\omega$ and all $i_{0},\ldots,i_{n}$ and $j_{0},\ldots,j_{n}$
from $I$ we have:
	\[
	\qftp_{\mathcal{L}_0}\left(i_{0}, \ldots, i_{n}\right)=\qftp_{\mathcal{L}_0}\left(j_{0},\ldots,j_{n}\right)
	\ \Longrightarrow\
	\tp_{\mathcal{L}}\left(a_{i_{0}}, \ldots, a_{i_{n}} / C\right)=\tp_{\mathcal{L}}\left(a_{j_{0}}, \ldots, a_{j_{n}} / C \right).
	\]
\item For $\mathcal{L}_0$-structures $I$ and $J$, we say that $\left(b_{i}\right)_{i\in J}$
is \emph{locally based on} $\left(a_{i}\right)_{i\in I}$ over a set of parameters $C \subseteq \C$ if for any finite
set $\Delta$ of $\mathcal{L}(C)$-formulas,  and for any finite tuple $\left(j_{0},\ldots,j_{n}\right)$
from $J$ there is a tuple $\left(i_{0},\ldots,i_{n}\right)$ from
$I$ such that:

\begin{itemize}
\item $\qftp_{\mathcal{L}_0}\left(j_{0},\ldots,j_{n}\right)=\qftp_{\mathcal{L}_0}\left(i_{0},\ldots,i_{n}\right)$ and
\item $\tp_{\Delta}\left(b_{j_{0}},\ldots,b_{j_{n}}\right)=\tp_{\Delta}\left(a_{i_{0}},\ldots,a_{i_{n}}\right)$.
\end{itemize}
\item Let $I$ be an $\mathcal{L}_0$-structure. We say that $I$ has the \emph{modeling property} if given any $\bar{a} = (a_i)_{i \in I}$ there exists an $I$-indiscernible $\bar{b} = (b_i)_{i \in I}$ locally based on $\bar{a}$.
\end{enumerate}
\end{defn}

The following is a standard method for finding generalized indiscernibles using structural Ramsey theory  \cite{scow2015indiscernibles} (see also \cite[Fact 4.7]{chernikov2019n}):
\begin{fact}
		Let $\mathcal{K}$ be a class of finite $\mathcal{L}_0$-structures, with  $\mathcal{L}_0$ a finite relational language, and let $I$ be a countable $\mathcal{L}_0$-structure such that $\operatorname{Age}(I)=\mathcal{K}$. Then $\mathcal{K}$ is a Ramsey class if and only if $I$ has the modeling property.
\end{fact}

\begin{fact}\cite[Corollary 4.8]{chernikov2019n}\label{fac: random hypergraph indiscernibles exist}
Let $C \subseteq \C$ be a small set of parameters.
\begin{enumerate}
\item For any $n \in \omega$ and $\bar{a}=\left(a_{g}\right)_{g\in O_{n,p}}$, there is some $\left(b_{g}\right)_{g\in O_{n,p}}$ which is $O_{n,p}$-indiscernible over $C$ and is locally based on $\bar{a}$ over $C$.
\item For any $n \in \omega$ and $\bar{a}=\left(a_{g}\right)_{g\in G_{n,p}}$, there is some $\left(b_{g}\right)_{g\in G_{n,p}}$ which is $G_{n,p}$-indiscernible over $C$ and is locally based on $\bar{a}$ over $C$.
\end{enumerate}
\end{fact}
%

%

\begin{fact}\cite[Proposition 2.8]{chernikov2021n}\label{fac: indisc witness to IPn}
Let $n\geq1$ and let $\varphi(x;y_1,\ldots,y_n)$ be a formula. The following are equivalent in any theory $T$.
\begin{enumerate}
	\item $\varphi(x;y_1, \ldots, y_n)$ is not $n$-dependent.
	\item There are tuples $b$ and $(a_g)_{g \in G_{n,p}}$ such that
	\begin{enumerate}
	\item $(a_g)_{g \in G_{n,p}}$ is $O_{n,p}$-indiscernible over $\emptyset$ and $G_{n,p}$-indiscernible over $b$;
	\item $\models \varphi(b;a_{g_1}, \ldots, a_{g_{n}}) \iff G_{n,p}\models R_n(g_1, \ldots, g_n)$, for all $g_i 
	\in P_i$.
	\end{enumerate}
	\item Item (2) holds for any small $G'_{n,p} \equiv G_{n,p}$ in the place of $G_{n,p}$.
	\end{enumerate}	
\end{fact}

We will need the following easy lemma.

\begin{lemma}
	Let $n\geq1$ and let $\varphi(x;y_1,\ldots,y_n)$ be a formula. Let $b$ and $(a_g)_{g \in G_{n,p}}$ be such that
$(a_g)_{g \in G_{n,p}}$ is $O_{n,p}$-indiscernible over $\emptyset$ and $G_{n,p}$-indiscernible over $b$ and $\models \varphi(b;a_{g_1}, \ldots, a_{g_{n}}) \iff G_{n,p}\models R_n(g_1, \ldots, g_n)$, for all $g_i 
	\in P_i$. Then the following hold:
	\begin{enumerate}
	\item for any $1 \leq m \leq n$, the sequences 
	$$(a_g)_{g \in P_1}, \ldots,(a_g)_{g \in P_{m-1}},  (a_g)_{g \in P_{m+1}}, \ldots, (a_g)_{g \in P_{n}}$$
	 are mutually indiscernible over $b$; in particular, 
	 $$(\gamma_{\bar{g}} : \bar{g} \in P_1 \times \ldots \times P_{m-1} \times P_{m+1} \times \ldots \times P_n)$$
		  with $\gamma_{\bar{g}} := (b, a_{g_1}, \ldots, a_{g_{m-1}}, a_{g_{m+1}},\ldots,a_{g_n})$ is an indiscernible array;
		\item the sequences $(a_{g})_{g \in P_1}, \ldots, (a_{g})_{g \in P_n}$ are mutually indiscernible; in particular, $(\gamma_{\bar{g}} : \bar{g} \in P_1 \times \ldots \times P_n)$ with $\gamma_{\bar{g}} := (a_{g_1}, \ldots, a_{g_n})$ is an indiscernible array.
	\end{enumerate}
\end{lemma}
\begin{proof}
		Part (1) is immediate from $G_{n,p}$-indiscernibility over $b$, as for any finite tuples $\bar{g},\bar{g}'$ from $P_1 \times \ldots \times P_{m-1} \times P_{m+1} \times \ldots \times P_n$ we have 
	$$\qftp_{\mathcal{L}^{n}_{\op}}(\bar{g}) = \qftp_{\mathcal{L}^{n}_{\op}}(\bar{g}') \implies \qftp_{\mathcal{L}^{n}_{\opg}}(\bar{g}) = \qftp_{\mathcal{L}^{n}_{\opg}}(\bar{g}')$$
		 (as the edge relation $R_n$ can never hold).
		Part (2) is immediate from $O_{n,p}$-indiscernibility.	
\end{proof}

\subsection{Finitary type-counting characterization of $n$-dependence}\label{sec: fin type count}
We recall a generalization of the Sauer--Shelah lemma to higher-arity VC-dimension from \cite{chernikov2019n}.

\begin{defn}
	Let $X_1, \ldots, X_k$ be arbitrary sets, and let $X := \prod_{i=1}^k X_i$. Let $\mathcal{F} \subseteq \mathcal{P}(X)$ be a family of subsets of $X$. 
	\begin{enumerate}
	\item Given $A \subseteq X$, we write $\mathcal{F} \cap A := \{ B \subseteq A : B = A \cap S \textrm{ for some } S \in \mathcal{F} \}$.
		\item We say that $\mathcal{F}$ \emph{shatters} a set $A \subseteq X$ if for every $B \subseteq A$ there is some $S \in \mathcal{F}$ such that $B = A \cap S$.
		\item By an \emph{$n$-box} we mean a set $A = \prod_{i=1}^{k} A_i$ with $|A_i| = n$.
		\item The \emph{$\VC_k$-dimension} of $\mathcal{F}$, denoted $\VC_k(\mathcal{F})$, is the largest $n \in \mathbb{N}$ such that $\mathcal{F}$ shatters an $n$-box, or $\infty$ if there is no largest such $n$.
	\end{enumerate} 
\end{defn}

\begin{fact} \cite[Proposition 3.9]{chernikov2019n}\label{fac: shatter} For every $k, d \in \mathbb{N}$ there exist some $\varepsilon \in \mathbb{R}_{>0}$ and $n_0 \in \mathbb{N}$ such that the following holds for all $n>n_0$.
	If $X_1, \ldots, X_k$ are finite sets with $|X_i| = n$, and $\mathcal{F} \subseteq \mathcal{P}(\prod_{i=1}^kX_i)$ with $|\mathcal{F}| \geq 2^{n^{k - \varepsilon}}$, then $\mathcal{F}$ shatters some $d$-box $A \subseteq  \prod_{i=1}^k X_i$.
\end{fact}

\begin{remark}
	Given a formula $\varphi(x;y_1, \ldots, y_k)$ and $a \in \C^x$, we consider:
	\begin{enumerate}
		\item $\varphi(a; \C) := \left\{(b_1, \ldots, b_k) \in \C^{y_1} \times \ldots \times \C^{y_k} : \models \varphi(a;b_1, \ldots, b_k) \right\}$,
		\item  $\mathcal{F}_{\varphi} := \left\{ \varphi(a; \C): a \in \C^x \right\} \subseteq \mathcal{P}\left( \C^{y_1} \times \ldots \times \C^{y_k} \right)$,
		\item $\VC_k(\varphi) := \VC_k(\mathcal{F}_{\varphi})$.
	\end{enumerate}
	 \end{remark}

We now establish a finitary version of the type-counting criterion for $k$-dependence from \cite[Section 5.1]{chernikov2021n}. This provides a partial answer to \cite[Problem 5.11]{chernikov2021n}.
\begin{defn}
Given a partitioned formula $\varphi(x; y_1, \ldots, y_k)$ and sets $A_0 \subseteq \mathbb{M}^x$,  $A_i \subseteq \mathbb{M}^{y_i}$ for $i \in \set{k-1}$ and  $B \subseteq \mathbb{M}^{y_k}$, we denote by $S_{\varphi, B}\left(A_0, A_1, \ldots, A_{k-1}  \right)$ the set of all $\varphi$-types over $A_0 \times A_1 \times \ldots \times A_{k-1}$ in the variable $y_k$ that are realized in $B$. Here, by a $\varphi$-type over $A_0 \times A_1 \times \ldots \times A_{k-1}$ in the variable $y_k$ we mean a maximal consistent collection of formulas of the form $\varphi(a_0;a_1, \ldots, a_{k-1}, y_k)$ or $\neg \varphi(a_0;a_1, \ldots, a_{k-1}, y_k)$ with $(a_0, a_1, \ldots, a_{k-1}) \in A_0 \times A_1 \times \ldots \times A_{k-1}$. If $A_0 = \{b \}$ is a singleton, we will simply write $S_{\varphi, B}\left(b, A_1, \ldots, A_{k-1}  \right)$.
\end{defn}

\begin{defn}
	Given a formula $\varphi(x;y_1, \ldots, y_k)$, $\varepsilon \in \mathbb{R}_{>0}$ and a function $f: \mathbb{N} \to \mathbb{N}$, we consider the following condition.
\begin{itemize}
	\item[$(\dagger)_{f,\varepsilon}$] There exists some $n^* \in \mathbb{N}$ such that the following holds for all $n^* \leq n \leq m \in \mathbb{N}$: for any mutually indiscernible sequences $I_1, \ldots, I_k$  of finite length, with $I_i \subseteq \mathbb{M}^{y_i}$, $n = |I_1| = \ldots = |I_{k-1}|$, $m =|I_k|$, and $b \in \mathbb{M}^x$ an arbitrary tuple, there exists an interval $J \subseteq I_k$ with $|J| \geq \frac{m}{f(n)}-1$
		satisfying $|S_{\varphi,J}(b, I_1, \ldots, I_{k-1})| < 2^{n^{k-1-\varepsilon}}$.
\end{itemize}
\end{defn}

\begin{prop}\label{prop: type count}
	The following are equivalent for a formula $\varphi(x; y_1, \ldots, y_k)$, with $k \geq 2$:
	\begin{enumerate}
		\item $\varphi(x; y_1, \ldots, y_k)$ is $k$-dependent.
		\item There exist some $\varepsilon > 0$ and $d \in \mathbb{N}$  such that $\varphi$ satisfies $(\dagger)_{f,\varepsilon}$ with respect to the function $f(n) = n^d$. 
				\item There exist some $\varepsilon>0$ and some function $f: \mathbb{N} \to \mathbb{N}$ such that $\varphi$ satisfies $(\dagger)_{f,\varepsilon}$.
	\end{enumerate}
	\end{prop}
\begin{proof}
(1) implies (2). Assume that $\varphi(x;y_1, \ldots, y_k)$ is $k$-dependent, and put $d_0:=\VC_k(\varphi)+1$, so that no $d_0$-box is shattered by $\mathcal F_\varphi$. Let $\varepsilon>0$ and $n_0$ be as given by Fact \ref{fac: shatter} for $k-1$ and $d_0$. Let $d:=d_0(k-1)+1$ and $n^*:=\max\{n_0+1,d_0\}$. Assume that $n^*\leq n\leq m$, together with $b$ and $I_1,\ldots,I_k$, witnesses that (2) fails with respect to $d$ and $\varepsilon$ (in particular, $m>f(n)$, since otherwise (2) would hold using $J:=\emptyset$). We can partition $I_k$ into $n':=f(n)=n^d\geq d_0n^{d_0(k-1)}$ disjoint intervals $J_1<\ldots<J_{n'}$, with $|J_i|\geq m/f(n)-1$ for all $i\in[n']$. Then $|S_{\varphi,J_i}(b,I_1,\ldots,I_{k-1})|\geq2^{n^{k-1-\varepsilon}}$ for each $i$. Hence, by Fact \ref{fac: shatter}, for each $i$ there is a $d_0$-box $A_i\subseteq I_1\times\cdots\times I_{k-1}$ shattered by the family $\{\varphi(b,y_1,\ldots,y_{k-1},a):a\in J_i\}$. There are at most $\binom{n}{d_0}^{k-1}$ choices for $A_i$, so by the pigeonhole principle there are a $d_0$-box $A'=A'_1\times\cdots\times A'_{k-1}$ and indices $1\leq i_1<\cdots<i_{d_0}\leq n'$ such that $A'$ is shattered by each family $\{\varphi(b,y_1,\ldots,y_{k-1},a):a\in J_{i_t}\}$, for $t\in[d_0]$.

For each $1 \leq i \leq k$, let $I'_i$ be the initial segment of $I_i$ of length $d_0$, and let $R \subseteq I'_1 \times \ldots \times I'_{k}$ be arbitrary. By the choice of $A'$ we can select points $c_{t} \in J_{i_t}$ such that the ordered $k$-partite hypergraph $(R; I'_1, \ldots, I'_k)$ is isomorphic to the ordered $k$-partite hypergraph $(\varphi(b; y_1, \ldots, y_k); A'_1, \ldots, A'_{k-1}, \{ c_1, \ldots, c_{d_0}\})$. By mutual indiscernibility of $I_1, \ldots, I_k$ there is an automorphism sending $A'_i$ to $I'_i$ for $1 \leq i \leq k-1$ and $\{c_1, \ldots, c_{d_0} \}$ to $I'_k$; then the image $b_R$ of $b$ under this automorphism satisfies 
$$\models \varphi(b_R; a_1, \ldots, a_k) \iff (a_1, \ldots, a_k) \in R$$
for all $a_i \in I'_i$, $1 \leq i \leq k$. That is, the $d_0$-box $I'_1 \times \ldots \times I'_k$ is shattered by the instances of $\varphi$, contradicting the assumption on the $\VC_k$-dimension of $\varphi$.

(2) implies (3) is obvious.

	(3) implies (1). 	
		Let $k\geq 2$ and $d,n\geq1$ be arbitrary. We define a $k$-partite ordered hypergraph $(R^{k-1}_{d,n}; V_1, \ldots, V_{k})$ as follows. Let $d' := 2d$, and let $s_1, \ldots, s_{2^{n^{k-1}}}$ be an arbitrary enumeration of $\mathcal{P}([n]^{k-1})$. We take $V_i := [n]$ for $1 \leq i \leq k-1$ with the natural ordering, and we let $V_k := [d'] \times [2^{n^{k-1}}]$, with the order given by the lexicographic product of the natural orders on $[d']$ and $[2^{n^{k-1}}]$. For any $(a_1, \ldots, a_{k-1}) \in V_1 \times \ldots \times V_{k-1}$ and $(t,a) \in V_k$, we define
$$(a_1, \ldots, a_{k-1}, (t,a)) \in R^{k-1}_{d,n} :\iff (a_1, \ldots, a_{k-1}) \in s_{a}.$$
\begin{claim} \label{cla: bad hypergraph} The $k$-partite ordered hypergraph $(R^{k-1}_{d,n}; V_1, \ldots, V_{k})$ satisfies the following.
\begin{enumerate}
	\item[(a)] $|V_i| = n$ for $1 \leq i \leq k-1$ and $|V_k| = 2d\cdot2^{n^{k-1}}$.
	\item[(b)] For any interval $J \subseteq V_k$ with $|J| \geq \frac{|V_k|}{d}-1$, the family of sets 
	$$\mathcal{F}_J := \left\{ R^{k-1}_{d,n}(y_1, \ldots, y_{k-1},c) \cap V_1 \times \ldots \times V_{k-1} : c\in J \right\}$$
	 shatters $V_1 \times \ldots \times V_{k-1}$. In particular, $|\mathcal{F}_J| = 2^{n^{k-1}}$.
\end{enumerate}	
\end{claim}
\begin{proof}
(a) is obvious by construction. For (b), put $N:=2^{n^{k-1}}$. An interval in the lexicographically ordered set $[2d]\times[N]$ of length at least $|V_k|/d-1=2N-1$ contains an entire block $\{t\}\times[N]$ for some $t\in[2d]$. Hence (b) holds by construction.
\end{proof}

		Now assume that $\varphi(x;y_1,\ldots,y_k)$ is not $k$-dependent. By Fact \ref{fac: indisc witness to IPn}, there are tuples $b$ and $(a_g)_{g\in G_{k,p}}$ such that $(a_g)_{g\in G_{k,p}}$ is $O_{k,p}$-indiscernible and, for all $g_i\in P_i$,
\[
\models\varphi(b;a_{g_1},\ldots,a_{g_k})
\quad\Longleftrightarrow\quad
G_{k,p}\models R_k(g_1,\ldots,g_k).
\]
In particular, taking $I_i:=(a_g:g\in P_i)$, the sequences $I_1,\ldots,I_k$ are mutually indiscernible (with the ordering induced from each $P_i$).

We fix an arbitrary $f: \mathbb{N}\to \mathbb{N}$, $\varepsilon > 0$ and $n^* \in \mathbb{N}$. 
		As $G_{k,p}$ contains any finite ordered $k$-partite hypergraph as an induced substructure, restricting $I_1, \ldots, I_k$ to the corresponding finite subsequences we find finite mutually indiscernible sequences $I'_1, \ldots, I'_k$ such that the $k$-partite ordered hypergraph $\left( \varphi(b; y_1, \ldots, y_k); I'_1, \ldots, I'_k \right)$ is isomorphic to $(R^{k-1}_{f(n^*),n^*}; V_1, \ldots, V_{k})$. Then, by Claim \ref{cla: bad hypergraph}, $|I'_i| = n^*$ for $1 \leq i \leq k-1$ and $|I'_k| = m := 2f(n^*)\cdot2^{(n^*)^{k-1}}$. If $J$ is an arbitrary interval in $I'_k$ with $|J| \geq \frac{m}{f(n^*)}-1$, then $I'_1 \times \ldots \times I'_{k-1}$ is shattered by the family $\{\varphi(b; y_1, \ldots, y_{k-1}, c) :c \in J \}$, so $|S_{\varphi, J}(b, I'_1, \ldots, I'_{k-1})| = 2^{(n^*)^{k-1}}$. This contradicts $(\dagger)_{f,\varepsilon}$.
	\end{proof}

\subsection{Array shattering lemma in NIP structures}\label{sec: array shattering}

The following is our key technical lemma. In this section, given a tuple $\bar{i} = (i_1, \ldots, i_k)$ and $s \subseteq [k]$, we write $\bar{i}_{s}$ to denote the subtuple $(i_t : t \in s)$. If $\bar{i}' = (i'_1, \ldots, i'_{k'})$ is another tuple, we write $\bar{i}+\bar{i}'$ to denote the concatenated tuple $(i_1, \ldots, i_k, i'_1, \ldots, i'_{k'})$.
\begin{lemma}[Array shattering lemma]\label{lem: comp lemma induction}
	Assume that every $\mathcal{L}$-formula is NIP. Then $(1)_k$ below holds for all $k \geq 2$ and $(2)_k$ holds for all $k \geq 1$ (below all of the $y_i$'s are tuples of variables of arbitrary finite length).
	
	\begin{enumerate}
		\item[$(1)_k$] 	For every $\varphi(y_0, \ldots, y_{k}) \in \mathcal{L}$ there exist some $f_\varphi: \mathbb{N} \to \mathbb{N}$, $n_\varphi \in \mathbb{N}$, a sequence $(\Delta_{\varphi,n} : n \in \mathbb{N})$ with each $\Delta_{\varphi,n}$ a finite set of $\mathcal{L}$-formulas,  and $\varepsilon_{\varphi} > 0$ satisfying the following:
		
		Let $n_\varphi \leq n \leq m \in \mathbb{N}$ be arbitrary. Let		
		$\bar{\gamma} = (\gamma_{\bar{i}} : \bar{i}  \in [n]^{k-1} \times [m])$ be a $\Delta_{\varphi,n}$-indiscernible array with $\gamma_{\bar{i}} \in \mathbb{M}^{y_k}$. Let $\bar{\alpha}= (\alpha_{\bar{i}} : \bar{i} \in [n]^{k-1})$ with $\alpha_{\bar{i}} \in  \mathbb{M}^{y_0}$ and $\bar{\beta}^t=(\beta^t_{\bar{i}} : \bar{i} \in [n]^{k-2} \times [m])$ with $\beta^t_{\bar{i}} \in \mathbb{M}^{y_t}$ for $1 \leq t \leq k-1$ be arbitrary arrays.
		 For $j \in [m]$, let 
		 $$S^{\varphi}_{j} := \left\{\bar{i} \in [n]^{k-1}: \models \varphi \left(\alpha_{\bar{i}}, \beta^1_{\bar{i}_{[k-1]\setminus \{1\}} + (j)}, \ldots, \beta^{k-1}_{\bar{i}_{[k-1]\setminus \{k-1\}} + (j)}, \gamma_{\bar{i} + (j)} \right) \right\}.$$
		  Then there exists some interval $J  \subseteq [m]$ with $|J| \geq \frac{m}{f_\varphi(n)}-1$ such that the family of sets $\mathcal{F}^\varphi_J := \left\{S^\varphi_j : j \in J \right\}$
			 has cardinality $<2^{n^{k-1-\varepsilon_\varphi}}$.
		
		\item[$(2)_k$] For every $\varphi(y_1, \ldots, y_{k+1}) \in \mathcal{L}$ there exist some $n'_\varphi \in \mathbb{N}$ and $\varepsilon'_{\varphi} > 0$ satisfying the following:
		
		Let $n'_\varphi \leq n \in \mathbb{N}$ and
		$\bar{\delta}=(\delta_{\bar{i}} : \bar{i}  \in [n]^{k})$ be an array with $\delta_{\bar{i}} \in \mathbb{M}^{y_{k+1}}$. For a sequence of  arrays $\left(\bar{\zeta}^t=(\zeta^t_{\bar{i}} : \bar{i} \in [n]^{k-1})\right)_{1 \leq t \leq k}$ with $\zeta^t_{\bar{i}} \in \mathbb{M}^{y_t}$, we define 
		 $$S^{\varphi,\bar{\delta}}_{\bar{\zeta}^1, \ldots, \bar{\zeta}^k} := \left\{\bar{i} \in [n]^{k}: \models \varphi \left(\zeta^1_{\bar{i}_{[k]\setminus \{1\}}}, \ldots, \zeta^k_{\bar{i}_{[k]\setminus \{k\}}}, \delta_{\bar{i}} \right) \right\}.$$
		  Then the family 
		  $$\mathcal{F}^{\varphi,\bar{\delta}} = \left\{S^{\varphi,\bar{\delta}}_{\bar{\zeta}^1, \ldots, \bar{\zeta}^k} \subseteq [n]^{k} : \bar{\zeta}^1, \ldots, \bar{\zeta}^k \textrm{ arbitrary arrays} \right\}$$  has cardinality $<2^{n^{k - \varepsilon'_{\varphi}}}$.

	\end{enumerate}
\end{lemma}
\begin{proof}
We already saw an illustration of the properties $(2)_2$ and $(2)_3$ in the introduction.
The following is an illustration of properties $(1)_2$ and $(1)_3$:

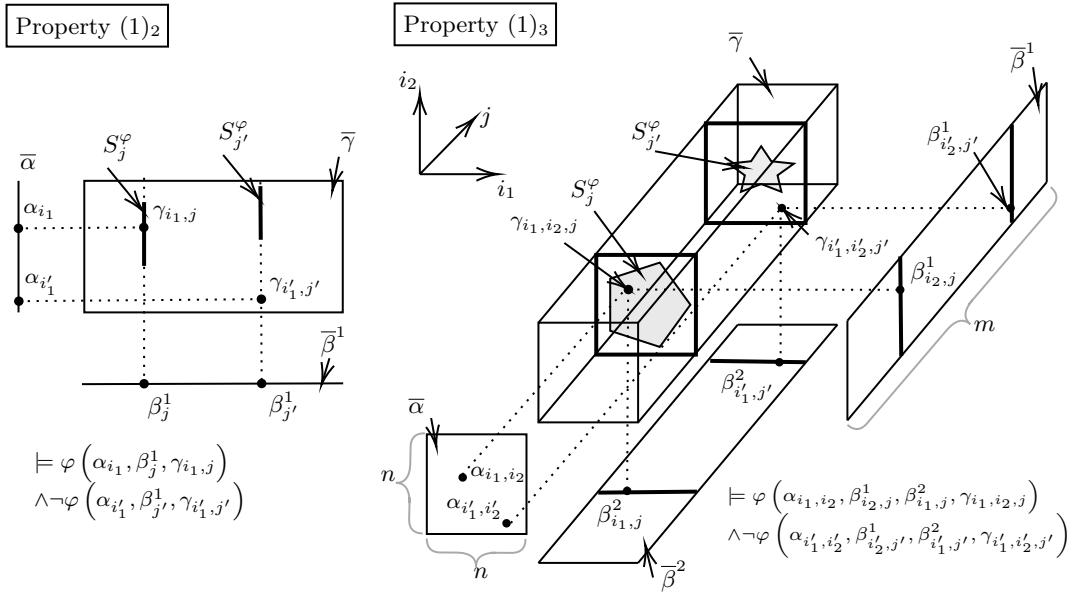
\tikzset{every picture/.style={line width=0.75pt}} 

\begin{tikzpicture}[x=0.75pt,y=0.75pt,yscale=-1,xscale=1]

\draw   (268.17,161.17) -- (318.17,161.17) -- (318.17,211.17) -- (268.17,211.17) -- cycle ;
\draw   (368.17,41.67) -- (418.17,41.67) -- (418.17,91.67) -- (368.17,91.67) -- cycle ;
\draw    (368.17,41.67) -- (268.17,161.17) ;
\draw    (418.17,41.67) -- (318.17,161.17) ;
\draw    (368.17,91.67) -- (268.17,211.17) ;
\draw    (418.17,91.67) -- (318.17,211.17) ;
\draw    (368.17,162.17) -- (268.17,281.67) ;
\draw    (418.17,162.17) -- (318.17,281.67) ;
\draw    (268.17,281.67) -- (318.17,281.67) ;
\draw    (368.17,162.17) -- (418.17,162.17) ;
\draw   (212.17,217.17) -- (262.17,217.17) -- (262.17,267.17) -- (212.17,267.17) -- cycle ;
\draw  [line width=1.5]  (297.17,127.17) -- (347.17,127.17) -- (347.17,177.17) -- (297.17,177.17) -- cycle ;
\draw  [line width=1.5]  (352.17,61.17) -- (402.17,61.17) -- (402.17,111.17) -- (352.17,111.17) -- cycle ;
\draw [line width=1.5]    (298.17,246.42) -- (347.67,245.92) ;
\draw    (523.17,40.67) -- (423.17,160.17) ;
\draw    (523.17,90.67) -- (423.17,210.17) ;
\draw    (423.17,160.17) -- (423.17,210.17) ;
\draw    (523.17,40.67) -- (523.17,90.67) ;
\draw [line width=1.5]    (449.67,127.92) -- (450.17,177.92) ;
\draw  [fill={rgb, 255:red, 0; green, 0; blue, 0 }  ,fill opacity=0.07 ] (379.92,72.67) -- (384.99,80.18) -- (396.32,81.39) -- (388.12,87.24) -- (390.06,95.51) -- (379.92,91.6) -- (369.78,95.51) -- (371.71,87.24) -- (363.51,81.39) -- (374.85,80.18) -- cycle ;
\draw  [fill={rgb, 255:red, 0; green, 0; blue, 0 }  ,fill opacity=1 ] (315.17,144.54) .. controls (315.17,143.51) and (314.33,142.67) .. (313.29,142.67) .. controls (312.26,142.67) and (311.42,143.51) .. (311.42,144.54) .. controls (311.42,145.58) and (312.26,146.42) .. (313.29,146.42) .. controls (314.33,146.42) and (315.17,145.58) .. (315.17,144.54) -- cycle ;
\draw  [dash pattern={on 0.84pt off 2.51pt}]  (313.29,144.54) -- (449.67,144.67) ;
\draw  [dash pattern={on 0.84pt off 2.51pt}]  (313.17,145.17) -- (312.67,247.17) ;
\draw  [fill={rgb, 255:red, 0; green, 0; blue, 0 }  ,fill opacity=0.09 ] (344.42,152.17) -- (329.04,173.33) -- (304.17,165.24) -- (304.17,139.09) -- (329.04,131.01) -- cycle ;
\draw    (281.17,121.17) -- (311.67,143.36) ;
\draw [shift={(313.29,144.54)}, rotate = 216.04] [color={rgb, 255:red, 0; green, 0; blue, 0 }  ][line width=0.75]    (10.93,-3.29) .. controls (6.95,-1.4) and (3.31,-0.3) .. (0,0) .. controls (3.31,0.3) and (6.95,1.4) .. (10.93,3.29)   ;
\draw    (296.17,105.17) -- (318.54,135.37) ;
\draw [shift={(319.73,136.98)}, rotate = 233.47] [color={rgb, 255:red, 0; green, 0; blue, 0 }  ][line width=0.75]    (10.93,-3.29) .. controls (6.95,-1.4) and (3.31,-0.3) .. (0,0) .. controls (3.31,0.3) and (6.95,1.4) .. (10.93,3.29)   ;
\draw    (326.67,75.17) -- (369.7,83.06) ;
\draw [shift={(371.67,83.42)}, rotate = 190.39] [color={rgb, 255:red, 0; green, 0; blue, 0 }  ][line width=0.75]    (10.93,-3.29) .. controls (6.95,-1.4) and (3.31,-0.3) .. (0,0) .. controls (3.31,0.3) and (6.95,1.4) .. (10.93,3.29)   ;
\draw  [dash pattern={on 0.84pt off 2.51pt}]  (313.29,142.67) -- (230.17,237.17) ;
\draw  [fill={rgb, 255:red, 0; green, 0; blue, 0 }  ,fill opacity=1 ] (228.54,238.79) .. controls (228.54,237.89) and (229.27,237.17) .. (230.17,237.17) .. controls (231.06,237.17) and (231.79,237.89) .. (231.79,238.79) .. controls (231.79,239.69) and (231.06,240.42) .. (230.17,240.42) .. controls (229.27,240.42) and (228.54,239.69) .. (228.54,238.79) -- cycle ;
\draw    (374.17,28.17) -- (383.15,43.44) ;
\draw [shift={(384.17,45.17)}, rotate = 239.53] [color={rgb, 255:red, 0; green, 0; blue, 0 }  ][line width=0.75]    (10.93,-3.29) .. controls (6.95,-1.4) and (3.31,-0.3) .. (0,0) .. controls (3.31,0.3) and (6.95,1.4) .. (10.93,3.29)   ;
\draw  [color={rgb, 255:red, 0; green, 0; blue, 0 }  ,draw opacity=0.34 ] (423.17,212.42) .. controls (426.72,215.45) and (430.01,215.2) .. (433.04,211.65) -- (472.98,164.94) .. controls (477.31,159.87) and (481.25,158.86) .. (484.8,161.89) .. controls (481.25,158.86) and (481.65,154.81) .. (485.98,149.74)(484.03,152.02) -- (525.93,103.03) .. controls (528.96,99.48) and (528.71,96.19) .. (525.16,93.16) ;
\draw  [color={rgb, 255:red, 0; green, 0; blue, 0 }  ,draw opacity=0.34 ] (212.17,269.42) .. controls (212.14,274.09) and (214.46,276.43) .. (219.13,276.45) -- (226.88,276.49) .. controls (233.55,276.52) and (236.87,278.87) .. (236.85,283.54) .. controls (236.87,278.87) and (240.21,276.56) .. (246.88,276.59)(243.88,276.58) -- (254.63,276.63) .. controls (259.3,276.66) and (261.64,274.34) .. (261.67,269.67) ;
\draw  [color={rgb, 255:red, 0; green, 0; blue, 0 }  ,draw opacity=0.34 ] (211.17,216.67) .. controls (206.5,216.58) and (204.12,218.86) .. (204.03,223.53) -- (203.87,231.65) .. controls (203.74,238.32) and (201.34,241.6) .. (196.67,241.51) .. controls (201.34,241.6) and (203.6,244.98) .. (203.47,251.65)(203.53,248.65) -- (203.31,259.78) .. controls (203.22,264.45) and (205.5,266.83) .. (210.17,266.92) ;
\draw  [fill={rgb, 255:red, 0; green, 0; blue, 0 }  ,fill opacity=1 ] (448.04,144.67) .. controls (448.04,143.77) and (448.77,143.04) .. (449.67,143.04) .. controls (450.56,143.04) and (451.29,143.77) .. (451.29,144.67) .. controls (451.29,145.56) and (450.56,146.29) .. (449.67,146.29) .. controls (448.77,146.29) and (448.04,145.56) .. (448.04,144.67) -- cycle ;
\draw  [fill={rgb, 255:red, 0; green, 0; blue, 0 }  ,fill opacity=1 ] (311.04,245.54) .. controls (311.04,244.64) and (311.77,243.92) .. (312.67,243.92) .. controls (313.56,243.92) and (314.29,244.64) .. (314.29,245.54) .. controls (314.29,246.44) and (313.56,247.17) .. (312.67,247.17) .. controls (311.77,247.17) and (311.04,246.44) .. (311.04,245.54) -- cycle ;
\draw [line width=1.5]    (505.67,61.67) -- (505.67,111) ;
\draw [line width=1.5]    (354.33,180.33) -- (402.33,181) ;
\draw    (516.33,33) -- (518.73,50.35) ;
\draw [shift={(519,52.33)}, rotate = 262.15] [color={rgb, 255:red, 0; green, 0; blue, 0 }  ][line width=0.75]    (10.93,-3.29) .. controls (6.95,-1.4) and (3.31,-0.3) .. (0,0) .. controls (3.31,0.3) and (6.95,1.4) .. (10.93,3.29)   ;
\draw    (214.33,209) -- (216.64,221.7) ;
\draw [shift={(217,223.67)}, rotate = 259.7] [color={rgb, 255:red, 0; green, 0; blue, 0 }  ][line width=0.75]    (10.93,-3.29) .. controls (6.95,-1.4) and (3.31,-0.3) .. (0,0) .. controls (3.31,0.3) and (6.95,1.4) .. (10.93,3.29)   ;
\draw    (325.67,289) -- (322.13,274.28) ;
\draw [shift={(321.67,272.33)}, rotate = 76.5] [color={rgb, 255:red, 0; green, 0; blue, 0 }  ][line width=0.75]    (10.93,-3.29) .. controls (6.95,-1.4) and (3.31,-0.3) .. (0,0) .. controls (3.31,0.3) and (6.95,1.4) .. (10.93,3.29)   ;
\draw    (208.67,86.67) -- (208.67,48) ;
\draw [shift={(208.67,46)}, rotate = 90] [color={rgb, 255:red, 0; green, 0; blue, 0 }  ][line width=0.75]    (10.93,-3.29) .. controls (6.95,-1.4) and (3.31,-0.3) .. (0,0) .. controls (3.31,0.3) and (6.95,1.4) .. (10.93,3.29)   ;
\draw    (208.67,86.67) -- (246,86.67) ;
\draw [shift={(248,86.67)}, rotate = 180] [color={rgb, 255:red, 0; green, 0; blue, 0 }  ][line width=0.75]    (10.93,-3.29) .. controls (6.95,-1.4) and (3.31,-0.3) .. (0,0) .. controls (3.31,0.3) and (6.95,1.4) .. (10.93,3.29)   ;
\draw    (208.67,86.67) -- (235.27,59.43) ;
\draw [shift={(236.67,58)}, rotate = 134.33] [color={rgb, 255:red, 0; green, 0; blue, 0 }  ][line width=0.75]    (10.93,-3.29) .. controls (6.95,-1.4) and (3.31,-0.3) .. (0,0) .. controls (3.31,0.3) and (6.95,1.4) .. (10.93,3.29)   ;
\draw  [fill={rgb, 255:red, 0; green, 0; blue, 0 }  ,fill opacity=1 ] (388.67,103.67) .. controls (388.67,102.77) and (389.39,102.04) .. (390.29,102.04) .. controls (391.19,102.04) and (391.92,102.77) .. (391.92,103.67) .. controls (391.92,104.56) and (391.19,105.29) .. (390.29,105.29) .. controls (389.39,105.29) and (388.67,104.56) .. (388.67,103.67) -- cycle ;
\draw  [dash pattern={on 0.84pt off 2.51pt}]  (389.67,181) -- (389.67,103.67) ;
\draw  [dash pattern={on 0.84pt off 2.51pt}]  (505.67,103.67) -- (391.92,103.67) ;
\draw  [fill={rgb, 255:red, 0; green, 0; blue, 0 }  ,fill opacity=1 ] (503.42,103.67) .. controls (503.42,102.77) and (504.14,102.04) .. (505.04,102.04) .. controls (505.94,102.04) and (506.67,102.77) .. (506.67,103.67) .. controls (506.67,104.56) and (505.94,105.29) .. (505.04,105.29) .. controls (504.14,105.29) and (503.42,104.56) .. (503.42,103.67) -- cycle ;
\draw  [fill={rgb, 255:red, 0; green, 0; blue, 0 }  ,fill opacity=1 ] (387.42,181) .. controls (387.42,180.1) and (388.14,179.38) .. (389.04,179.38) .. controls (389.94,179.38) and (390.67,180.1) .. (390.67,181) .. controls (390.67,181.9) and (389.94,182.63) .. (389.04,182.63) .. controls (388.14,182.63) and (387.42,181.9) .. (387.42,181) -- cycle ;
\draw    (403.67,117) -- (391.71,105.08) ;
\draw [shift={(390.29,103.67)}, rotate = 44.91] [color={rgb, 255:red, 0; green, 0; blue, 0 }  ][line width=0.75]    (10.93,-3.29) .. controls (6.95,-1.4) and (3.31,-0.3) .. (0,0) .. controls (3.31,0.3) and (6.95,1.4) .. (10.93,3.29)   ;
\draw  [dash pattern={on 0.84pt off 2.51pt}]  (254.33,261.67) -- (390.29,103.29) ;
\draw  [fill={rgb, 255:red, 0; green, 0; blue, 0 }  ,fill opacity=1 ] (250.54,261.79) .. controls (250.54,260.89) and (251.27,260.17) .. (252.17,260.17) .. controls (253.06,260.17) and (253.79,260.89) .. (253.79,261.79) .. controls (253.79,262.69) and (253.06,263.42) .. (252.17,263.42) .. controls (251.27,263.42) and (250.54,262.69) .. (250.54,261.79) -- cycle ;
\draw   (40.33,90) -- (170,90) -- (170,156) -- (40.33,156) -- cycle ;
\draw    (7.33,90) -- (7.33,156) ;
\draw    (39.17,192) -- (170.2,191.6) ;
\draw    (168.07,77.93) -- (165.11,93.3) ;
\draw [shift={(164.73,95.27)}, rotate = 280.89] [color={rgb, 255:red, 0; green, 0; blue, 0 }  ][line width=0.75]    (10.93,-3.29) .. controls (6.95,-1.4) and (3.31,-0.3) .. (0,0) .. controls (3.31,0.3) and (6.95,1.4) .. (10.93,3.29)   ;
\draw  [dash pattern={on 0.84pt off 2.51pt}]  (70.33,88.46) -- (70.33,191.79) ;
\draw  [dash pattern={on 0.84pt off 2.51pt}]  (128.67,90.67) -- (129.33,192.67) ;
\draw [line width=1.5]    (70.33,100.67) -- (70.33,132.67) ;
\draw [line width=1.5]    (128.67,92.67) -- (128.86,119.57) ;
\draw  [fill={rgb, 255:red, 0; green, 0; blue, 0 }  ,fill opacity=1 ] (68.3,113.37) .. controls (68.3,112.34) and (69.14,111.5) .. (70.18,111.5) .. controls (71.21,111.5) and (72.05,112.34) .. (72.05,113.37) .. controls (72.05,114.41) and (71.21,115.25) .. (70.18,115.25) .. controls (69.14,115.25) and (68.3,114.41) .. (68.3,113.37) -- cycle ;
\draw  [fill={rgb, 255:red, 0; green, 0; blue, 0 }  ,fill opacity=1 ] (68.46,191.79) .. controls (68.46,190.76) and (69.3,189.92) .. (70.33,189.92) .. controls (71.37,189.92) and (72.21,190.76) .. (72.21,191.79) .. controls (72.21,192.83) and (71.37,193.67) .. (70.33,193.67) .. controls (69.3,193.67) and (68.46,192.83) .. (68.46,191.79) -- cycle ;
\draw  [dash pattern={on 0.84pt off 2.51pt}]  (70.18,113.37) -- (7.29,113.86) ;
\draw  [fill={rgb, 255:red, 0; green, 0; blue, 0 }  ,fill opacity=1 ] (5.41,113.86) .. controls (5.41,112.82) and (6.25,111.98) .. (7.29,111.98) .. controls (8.32,111.98) and (9.16,112.82) .. (9.16,113.86) .. controls (9.16,114.89) and (8.32,115.73) .. (7.29,115.73) .. controls (6.25,115.73) and (5.41,114.89) .. (5.41,113.86) -- cycle ;
\draw    (52,83.5) -- (68.73,103.95) ;
\draw [shift={(70,105.5)}, rotate = 230.71] [color={rgb, 255:red, 0; green, 0; blue, 0 }  ][line width=0.75]    (10.93,-3.29) .. controls (6.95,-1.4) and (3.31,-0.3) .. (0,0) .. controls (3.31,0.3) and (6.95,1.4) .. (10.93,3.29)   ;
\draw    (110,77) -- (126.73,97.45) ;
\draw [shift={(128,99)}, rotate = 230.71] [color={rgb, 255:red, 0; green, 0; blue, 0 }  ][line width=0.75]    (10.93,-3.29) .. controls (6.95,-1.4) and (3.31,-0.3) .. (0,0) .. controls (3.31,0.3) and (6.95,1.4) .. (10.93,3.29)   ;
\draw  [fill={rgb, 255:red, 0; green, 0; blue, 0 }  ,fill opacity=1 ] (127.3,149.37) .. controls (127.3,148.34) and (128.14,147.5) .. (129.18,147.5) .. controls (130.21,147.5) and (131.05,148.34) .. (131.05,149.37) .. controls (131.05,150.41) and (130.21,151.25) .. (129.18,151.25) .. controls (128.14,151.25) and (127.3,150.41) .. (127.3,149.37) -- cycle ;
\draw  [dash pattern={on 0.84pt off 2.51pt}]  (129.18,149.37) -- (7.5,150.38) ;
\draw  [fill={rgb, 255:red, 0; green, 0; blue, 0 }  ,fill opacity=1 ] (5.62,150.38) .. controls (5.62,149.34) and (6.46,148.5) .. (7.5,148.5) .. controls (8.54,148.5) and (9.38,149.34) .. (9.38,150.38) .. controls (9.38,151.41) and (8.54,152.25) .. (7.5,152.25) .. controls (6.46,152.25) and (5.62,151.41) .. (5.62,150.38) -- cycle ;
\draw  [fill={rgb, 255:red, 0; green, 0; blue, 0 }  ,fill opacity=1 ] (127.46,191.67) .. controls (127.46,190.63) and (128.3,189.79) .. (129.33,189.79) .. controls (130.37,189.79) and (131.21,190.63) .. (131.21,191.67) .. controls (131.21,192.7) and (130.37,193.54) .. (129.33,193.54) .. controls (128.3,193.54) and (127.46,192.7) .. (127.46,191.67) -- cycle ;
\draw    (163.33,179) -- (159.91,190.42) ;
\draw [shift={(159.33,192.33)}, rotate = 286.7] [color={rgb, 255:red, 0; green, 0; blue, 0 }  ][line width=0.75]    (10.93,-3.29) .. controls (6.95,-1.4) and (3.31,-0.3) .. (0,0) .. controls (3.31,0.3) and (6.95,1.4) .. (10.93,3.29)   ;
\draw    (486.67,75.67) -- (502.39,101.95) ;
\draw [shift={(503.42,103.67)}, rotate = 239.11] [color={rgb, 255:red, 0; green, 0; blue, 0 }  ][line width=0.75]    (10.93,-3.29) .. controls (6.95,-1.4) and (3.31,-0.3) .. (0,0) .. controls (3.31,0.3) and (6.95,1.4) .. (10.93,3.29)   ;

\draw (252.67,106.75) node [anchor=north west][inner sep=0.75pt]  [font=\scriptsize] [align=left] {$\displaystyle \gamma _{i_{1} ,i_{2} ,j}$};
\draw (282.67,87.42) node [anchor=north west][inner sep=0.75pt]  [font=\footnotesize] [align=left] {$\displaystyle S_{j}^{\varphi }$};
\draw (313.17,58.42) node [anchor=north west][inner sep=0.75pt]  [font=\footnotesize] [align=left] {$\displaystyle S_{j'}^{\varphi }$};
\draw (362.17,13.67) node [anchor=north west][inner sep=0.75pt]  [font=\footnotesize] [align=left] {$\displaystyle \overline{\gamma }$};
\draw (485.17,158.92) node [anchor=north west][inner sep=0.75pt]  [font=\footnotesize] [align=left] {$\displaystyle m$};
\draw (234.17,282.42) node [anchor=north west][inner sep=0.75pt]  [font=\footnotesize] [align=left] {$\displaystyle n$};
\draw (187.67,233.67) node [anchor=north west][inner sep=0.75pt]  [font=\footnotesize] [align=left] {$\displaystyle n$};
\draw (453.33,126.67) node [anchor=north west][inner sep=0.75pt]  [font=\footnotesize] [align=left] {$\displaystyle \beta _{i_{2} ,j}^{1}$};
\draw (503.5,19.33) node [anchor=north west][inner sep=0.75pt]  [font=\footnotesize] [align=left] {$\displaystyle \overline{\beta }^{1}$};
\draw (201.5,197.33) node [anchor=north west][inner sep=0.75pt]  [font=\footnotesize] [align=left] {$\displaystyle \overline{\alpha }$};
\draw (328.17,279.33) node [anchor=north west][inner sep=0.75pt]  [font=\footnotesize] [align=left] {$\displaystyle \overline{\beta }^{2}$};
\draw (296,248.67) node [anchor=north west][inner sep=0.75pt]  [font=\footnotesize] [align=left] {$\displaystyle \beta _{i_{1} ,j}^{2}$};
\draw (232.54,231.79) node [anchor=north west][inner sep=0.75pt]  [font=\scriptsize] [align=left] {$\displaystyle \alpha _{i_{1} ,i_{2}}$};
\draw (246,88.07) node [anchor=north west][inner sep=0.75pt]  [font=\footnotesize]  {$i_{1}$};
\draw (196.67,33.4) node [anchor=north west][inner sep=0.75pt]  [font=\footnotesize]  {$i_{2}$};
\draw (237.67,51.4) node [anchor=north west][inner sep=0.75pt]  [font=\footnotesize]  {$j$};
\draw (354.67,238.4) node [anchor=north west][inner sep=0.75pt]  [font=\tiny]  {$ \begin{array}{l}
\models \varphi \left( \alpha _{i_{1} ,i_{2}} ,\beta _{i_{2} ,j}^{1} ,\beta _{i_{1} ,j}^{2} ,\gamma _{i_{1} ,i_{2} ,j}\right)\\
\land \neg \varphi \left( \alpha _{i'_{1} ,i'_{2}} ,\beta _{i'_{2} ,j'}^{1} ,\beta _{i'_{1} ,j'}^{2} ,\gamma _{i'_{1} ,i'_{2} ,j'}\right)
\end{array}$};
\draw (358.33,183.33) node [anchor=north west][inner sep=0.75pt]  [font=\scriptsize] [align=left] {$\displaystyle \beta _{i'_{1} ,j'}^{2}$};
\draw (461.67,59) node [anchor=north west][inner sep=0.75pt]  [font=\scriptsize] [align=left] {$\displaystyle \beta _{i'_{2} ,j'}^{1}$};
\draw (405.67,116.42) node [anchor=north west][inner sep=0.75pt]  [font=\scriptsize] [align=left] {$\displaystyle \gamma _{i'_{1} ,i'_{2} ,j'}$};
\draw (220.54,247.79) node [anchor=north west][inner sep=0.75pt]  [font=\scriptsize] [align=left] {$\displaystyle \alpha _{i'_{1} ,i'_{2}}$};
\draw    (196,1) -- (277,1) -- (277,22) -- (196,22) -- cycle  ;
\draw (199,5) node [anchor=north west][inner sep=0.75pt]  [font=\footnotesize] [align=left] {Property $\displaystyle ( 1)_{3}$ };
\draw (157.53,161.67) node [anchor=north west][inner sep=0.75pt]  [font=\footnotesize]  {$\overline{\beta }^{1}$};
\draw (6.67,74.73) node [anchor=north west][inner sep=0.75pt]  [font=\footnotesize]  {$\overline{\alpha }$};
\draw (167.73,63.67) node [anchor=north west][inner sep=0.75pt]  [font=\footnotesize]  {$\overline{\gamma }$};
\draw (73.43,101.26) node [anchor=north west][inner sep=0.75pt]  [font=\scriptsize]  {$\gamma _{i_{1} ,j}$};
\draw (70.46,195.19) node [anchor=north west][inner sep=0.75pt]  [font=\footnotesize]  {$\beta _{j}^{1}$};
\draw (8.29,100.11) node [anchor=north west][inner sep=0.75pt]  [font=\scriptsize]  {$\alpha _{i_{1}}$};
\draw (48.17,64.92) node [anchor=north west][inner sep=0.75pt]  [font=\footnotesize] [align=left] {$\displaystyle S_{j}^{\varphi }$};
\draw (106.17,58.42) node [anchor=north west][inner sep=0.75pt]  [font=\footnotesize] [align=left] {$\displaystyle S_{j'}^{\varphi }$};
\draw (131.93,137.26) node [anchor=north west][inner sep=0.75pt]  [font=\scriptsize]  {$\gamma _{i'_{1} ,j'}$};
\draw (9,135.11) node [anchor=north west][inner sep=0.75pt]  [font=\scriptsize]  {$\alpha _{i'_{1}}$};
\draw (131.33,194.07) node [anchor=north west][inner sep=0.75pt]  [font=\footnotesize]  {$\beta _{j'}^{1}$};
\draw (7,220.4) node [anchor=north west][inner sep=0.75pt]  [font=\scriptsize]  {$ \begin{array}{l}
\models \varphi \left( \alpha _{i_{1}} ,\beta _{j}^{1} ,\gamma _{i_{1} ,j}\right)\\
\land \neg \varphi \left( \alpha _{i'_{1}} ,\beta _{j'}^{1} ,\gamma _{i'_{1} ,j'}\right)
\end{array}$};
\draw    (1.33,1.67) -- (82.33,1.67) -- (82.33,22.67) -- (1.33,22.67) -- cycle  ;
\draw (4.33,5.67) node [anchor=north west][inner sep=0.75pt]  [font=\footnotesize] [align=left] {Property $\displaystyle ( 1)_{2}$ };

\end{tikzpicture}

Note that $(2)_1$ holds by the Sauer--Shelah lemma. In this case, given an NIP formula $\varphi(y_1,y_2)$ and an arbitrary $\bar{\delta} = \left(\delta_i : i \in [n]\right)$, for any $\bar{\zeta}^1 = (\zeta)$ with $\zeta \in \mathbb{M}^{y_1}$, the set 
$$S^{\varphi,\bar{\delta}}_{\bar{\zeta}^1} := \left\{i \in [n]: \models \varphi \left(\zeta, \delta_{i} \right) \right\}$$
 is determined by the $\varphi$-type of $\zeta$ over $\bar{\delta}$. By the Sauer--Shelah lemma, there are $c,n_0\in\mathbb N$ such that the number of $\varphi$-types over every set of size $n\geq n_0$ is bounded by $n^c$. For any $0<\varepsilon'_\varphi<1$, there is $n'_\varphi\geq n_0$ such that $n^c<2^{n^{1-\varepsilon'_\varphi}}$ for all $n\geq n'_\varphi$, so $(2)_1$ holds.

 We will show that $(1)_k$ implies $(2)_k$ for all $k \geq 2$, and that $(2)_k$ implies $(1)_{k+1}$ for all $k \geq 1$, which is sufficient.

\textbf{$(1)_k$ implies $(2)_k$.} 
Let $\varepsilon_{\psi} > 0, n_{\psi}, \left( \Delta_{\psi,n} \right)_{n \in \mathbb{N}}$ and $f_{\psi}$ be as given by $(1)_k$ for 
$$\psi(y_0, \ldots, y_k) := \varphi(y_1, \ldots, y_{k-1}, y_0, y_k).$$
 We consider the $k$-partite ordered hypergraph $(R^{k-1}_{f_{\psi}(n_{\psi}),n_{\psi}}; V_1, \ldots, V_{k})$ defined in the proof of Proposition \ref{prop: type count}. By Claim \ref{cla: bad hypergraph}(a) we have $|V_i|=n_{\psi}$ for $1 \leq i \leq k-1$ and $|V_k|=m:=2f_{\psi}(n_{\psi})\cdot2^{n_{\psi}^{k-1}}$.
Let $r = r(k,\Delta_{\psi,n_{\psi}},m) \in \mathbb{N}$ be as given by Lemma \ref{lem: indisc subarray}. Finally, let $\varepsilon'_\varphi>0$ and $n_0\in\mathbb N$ be as given by Fact \ref{fac: shatter} with $d:=r$, and put $n'_\varphi:=\max\{n_0+1,n_\psi+1,m+1\}$.

We claim that $(2)_k$ holds for $\varphi$ with respect to $n'_\varphi, \varepsilon'_{\varphi}$ (note that they were chosen only depending on $\varphi$). If not, then there are $n \geq n'_\varphi$ and an array $\bar{\delta}=(\delta_{\bar{i}} : \bar{i}  \in [n]^{k})$ such that the family $\mathcal{F}^{\varphi,\bar{\delta}} = \left\{S^{\varphi,\bar{\delta}}_{\bar{\zeta}^1, \ldots, \bar{\zeta}^k} \right\} \subseteq \mathcal{P}([n]^k)$ has cardinality $\geq 2^{n^{k - \varepsilon'_{\varphi}}}$. By the choice of $n'_\varphi$, we have $n>n_\psi$, $n>m$, and $n>n_0$.
Then, by Fact \ref{fac: shatter} and the choice of $n'_\varphi, \varepsilon'_\varphi$, there exists some $I = I_1 \times \ldots \times I_k \subseteq [n]^k$ with $|I_1| = \ldots = |I_k| = r$ shattered by $\mathcal{F}^{\varphi,\bar{\delta}}$. By Lemma \ref{lem: indisc subarray} and the choice of $r$, there exist some $I'_t \subseteq I_t$ for $1 \leq t \leq k$, with $|I'_t| =n_{\psi}$ for $1 \leq t \leq k-1$ and $|I'_k| = m$, and such that the array $(\delta_{\bar{i}} : \bar{i} \in I')$ is $\Delta_{\psi,n_{\psi}}$-indiscernible, where $I' = I'_1 \times \ldots \times I'_k$. Note that $I'$ is still shattered by $\mathcal{F}^{\varphi,\bar{\delta}}$. In particular, there exist some arrays $\bar{\zeta}^t=(\zeta^t_{\bar{i}} : \bar{i} \in [n]^{k-1})$ for $1 \leq t \leq k$ such that, identifying $V_t$ with $I'_t$ in an order-preserving way for $1 \leq t \leq k$, for every $\bar{i} \in I'_1 \times \ldots \times I'_k$ we have
\begin{equation}\label{eq: shattering bad graph}
	\models \varphi \left(\zeta^1_{\bar{i}_{[k]\setminus \{1\}}}, \ldots, \zeta^k_{\bar{i}_{[k]\setminus \{k\}}}, \delta_{\bar{i}} \right) \iff \bar{i} \in R^{k-1}_{f_{\psi}(n_{\psi}),n_{\psi}}.
\end{equation}

Identifying $I'_t$ with $[n_{\psi}]$ in an order-preserving way for $1 \leq t \leq k-1$ and $I'_k$ with $[m]$, we define:

\begin{itemize}
	\item $\alpha_{\bar{i}} := \zeta^k_{\bar{i}}$ for $\bar{i} \in [n_{\psi}]^{k-1}$;
	\item for each $1 \leq t \leq k-1$, $\beta^t_{\bar{i}} := \zeta^t_{\bar{i}}$ for $\bar{i} \in [n_{\psi}]^{k-2} \times [m]$;
	\item $\gamma_{\bar{i}} := \delta_{\bar i}$ for all $\bar{i} \in [n_{\psi}]^{k-1} \times [m]$.
\end{itemize} 
Then, by \eqref{eq: shattering bad graph}, for any tuple $\bar{i} \in [n_{\psi}]^{k-1}$ and $j \in [m]$ we have
\begin{equation}
	\models \psi \left(\alpha_{\bar{i}}, \beta^1_{\bar{i}_{[k-1]\setminus \{1\}} + (j)}, \ldots, \beta^{k-1}_{\bar{i}_{[k-1]\setminus \{k-1\}} + (j)}, \gamma_{\bar{i} + (j)} \right)  \iff \bar{i}+(j) \in R^{k-1}_{f_{\psi}(n_{\psi}),n_{\psi}}.
\end{equation}

But then by Claim \ref{cla: bad hypergraph}(b), for any interval $J \subseteq [m]$ with $|J| \geq \frac{m}{f_{\psi}(n_{\psi})}-1$, the family of sets $S_j^{\psi} = \left\{ \bar{i} \in [n_{\psi}]^{k-1} : \bar{i}+(j) \in R^{k-1}_{f_{\psi}(n_{\psi}),n_{\psi}} \right\}$
	 with $j$ varying over $J$ shatters $[n_{\psi}]^{k-1}$, so in particular has cardinality $2^{n_{\psi}^{k-1}}$.
This is a contradiction to the choice of $\varepsilon_{\psi}, n_{\psi}, \Delta_{\psi, n_{\psi}}, f_{\psi}$, i.e.~to $(1)_k$.

\textbf{$(2)_k$ implies $(1)_{k+1}$.}
Let $\varphi(y_0, \ldots, y_{k+1}) \in \mathcal{L}$ be arbitrary.
For $n\geq1$, we identify $2^{n^k}$ with the set of functions
$\eta:[n]^k\to\{0,1\}$, and all tuples indexed by $[n]^k$ are listed
in lexicographic order.  For $\eta\in2^{n^k}$, consider the formula
\[
\begin{aligned}
&\xi_{\varphi,n,\eta}
 \big((x_{\bar i}:\bar i\in[n]^k);(y_{\bar i}:\bar i\in[n]^k)\big)
 :=\\
&\quad \exists(z^1_{\bar r}:\bar r\in[n]^{k-1})\cdots
\exists(z^k_{\bar r}:\bar r\in[n]^{k-1})
\bigwedge_{\bar i\in[n]^k}
\varphi\left(x_{\bar i},z^1_{\bar i_{[k]\setminus\{1\}}},\ldots,
z^k_{\bar i_{[k]\setminus\{k\}}},y_{\bar i}\right)^{\eta(\bar i)}.
\end{aligned}
\]
Here $x_{\bar i}$, $y_{\bar i}$, and $z^t_{\bar r}$ are pairwise
disjoint copies of $y_0$, $y_{k+1}$, and $y_t$, respectively.
Since $\xi_{\varphi,n,\eta}$ is NIP, let
$d'_{\varphi,n,\eta}\in\mathbb N$ and a finite set of formulas
$\Delta_{\varphi,n,\eta}$ be as chosen in the proof of Lemma
\ref{lem: fin shrinking} in the one-dimensional case, with $r=1$.

We use one convention concerning these formulas: whenever a tuple
variable $Y=(y_{\bar i}:\bar i\in[n]^k)$ occurs in a member of
$\Delta_{\varphi,n,\eta}$, we regard it as the lexicographically
ordered list of its entries.  For several such variables, the blocks
are listed in their displayed order.  This merely ungroups tuple
variables; it does not change the formula.

Let $\psi(y_1, \ldots, y_{k}, y'_{k+1}) := \varphi(y_0, \ldots, y_{k+1})$, where $y'_{k+1} := y_0^{\frown}y_{k+1}$. By assumption, it satisfies $(2)_k$ with respect to some $n'_{\psi}$ and $\varepsilon'_{\psi} > 0$.

We will show that $(1)_{k+1}$ holds for $\varphi$ with respect to:
\begin{itemize}
\item the function $f_\varphi:\mathbb N\to\mathbb N$ defined by
\[
 f_\varphi(0):=1,
 \qquad
 f_\varphi(n):=1+\sum_{\eta\in2^{n^k}}
 d'_{\varphi,n,\eta}\quad(n\geq1);
\]
\item the finite sets
\[
 \Delta_{\varphi,0}:=\varnothing,
 \qquad
 \Delta_{\varphi,n}:=
 \bigcup_{\eta\in2^{n^k}}\Delta_{\varphi,n,\eta}
 \quad(n\geq1),
\]
where the members of $\Delta_{\varphi,n,\eta}$ are understood according
to the preceding convention;
\item $\varepsilon_{\varphi} := \varepsilon'_{\psi}$,
\item $n_{\varphi} := \max\{1,n'_{\psi}\}$.
\end{itemize}

Assume that $n_\varphi \leq n \leq m$ are arbitrary. Let		
		$\bar{\gamma} = (\gamma_{\bar{i}} : \bar{i}  \in [n]^{k} \times [m])$ be a $\Delta_{\varphi,n}$-indiscernible array. Let $\bar{\alpha} = (\alpha_{\bar{i}} : \bar{i} \in [n]^{k})$ and $\bar{\beta}^t=\left(\beta^t_{\bar{i}} : \bar{i} \in [n]^{k-1} \times [m] \right)$ for $1 \leq t \leq k$ be arbitrary arrays. For each $j \in [m]$, let
		\[
  \bar\gamma^j:=(\gamma_{\bar i^{\frown}(j)}:\bar i\in[n]^k)
 \quad \bar\delta^j:=
 \big(\alpha_{\bar i}^{\frown}\gamma_{\bar i^{\frown}(j)}:
 \bar i\in[n]^k\big).
\]

For each $\eta\in2^{n^k}$, the sequence
$(\bar\gamma^j:j\in[m])$ is
$\Delta_{\varphi,n,\eta}$-indiscernible (from the definition and the $\Delta_{\varphi,n}$-indiscernibility
of the array $\bar\gamma$). We also have
 \begin{equation}\label{eq: arr ind 1}
 	S^{\varphi}_j \in \mathcal{F}^{\psi, \bar{\delta}^j} \textrm{ for all } j \in [m],
 \end{equation}
 where $S^{\varphi}_j$ is as in the definition of $(1)_{k+1}$, and $\mathcal{F}^{\psi,\bar{\delta}^j}$ as in the definition of $(2)_k$.
		
			Applying Lemma \ref{lem: fin shrinking} separately for each
$\eta\in2^{n^k}$ and taking the union of the resulting exceptional
sets, we obtain $J_0\subseteq[m]$ such that, with $t:=|J_0|$,
\[
 t\leq\sum_{\eta\in2^{n^k}}d'_{\varphi,n,\eta}
 =f_\varphi(n)-1,
\]
and, for every $\eta\in2^{n^k}$ and $j,j'\in[m]$, we have:
\begin{equation}\nonumber
	\qftp_{<}(j/J_0)=\qftp_<(j'/J_0) \implies \models \xi_{\varphi,n,\eta}(\bar{\alpha}, \bar{\gamma}^j)\leftrightarrow \xi_{\varphi,n,\eta}(\bar{\alpha}, \bar{\gamma}^{j'}).
\end{equation}
If $[m]\setminus J_0\neq\varnothing$, let $J$ be a largest interval
component of $[m]\setminus J_0$.  There are at most $t+1$ such
components and their total size is $m-t$, so
\[
 |J|\geq\frac{m-t}{t+1}>\frac{m}{t+1}-1
 \geq\frac{m}{f_\varphi(n)}-1.
\]
If $[m]\setminus J_0=\varnothing$, let $J$ be any singleton; the same
lower bound follows from $m=t\leq f_\varphi(n)-1$.  In either case all elements of $J$
have the same quantifier-free order type over $J_0$.  Consequently,
\begin{equation}\label{eq: arr ind 2}
 \models \xi_{\varphi,n,\eta}(\bar{\alpha}, \bar{\gamma}^j)\leftrightarrow \xi_{\varphi,n,\eta}(\bar{\alpha}, \bar{\gamma}^{j'}) \textrm{ for all } \eta \in 2^{n^k} \textrm{ and } j,j' \in J.
\end{equation}

By the choice of the formulas $\xi_{\varphi,n,\eta}$, \eqref{eq: arr ind 2} implies 
\begin{equation}\label{eq: arr ind 3}
	\mathcal{F}^{\psi, \bar{\delta}^j} =  \mathcal{F}^{\psi, \bar{\delta}^{j'}} \textrm{ for all } j,j' \in J.
\end{equation}

For an arbitrary fixed $j_0 \in J$, we have $|\mathcal{F}^{\psi, \bar{\delta}^{j_0}}| < 2^{n^{k - \varepsilon'_{\psi}}} = 2^{n^{k - \varepsilon_\varphi}}$ by $(2)_k$ and the choice of $\varepsilon_\varphi$ and $n_\varphi$. Hence, combining \eqref{eq: arr ind 1} and \eqref{eq: arr ind 3}, we conclude that 
$$|\mathcal{F}^{\varphi}_J| = \left|\left\{ S^{\varphi}_j : j \in J \right\}\right| \leq \left| \mathcal{F}^{\psi,\bar{\delta}^{j_0}} \right| < 2^{n^{k-\varepsilon_\varphi}},$$
as wanted.
\end{proof}

\subsection{Composition Lemma}\label{sec: comp lemma proof}
We are ready to prove the main result of this section (generalizing the binary case $k=2$ established in \cite{chernikov2021n}): composing a relation definable in an NIP structure with \emph{arbitrary} functions of arity at most $k$ produces a $k$-dependent relation.
As before, all the variables below are allowed to be tuples of arbitrary finite length.
\begin{theorem}[Composition Lemma]\label{Composition Lemma}
		Let $\mathcal{M} = (M, \ldots)$ be an $\mathcal{L}'$-structure such that its reduct to a language $\mathcal{L} \subseteq \mathcal{L}'$ is NIP. Let $d,k\geq1$, let $\varphi(x_1, \ldots, x_d)$ be an $\mathcal{L}$-formula, and let $y_0, \ldots, y_k$ be arbitrary tuples of variables. For each $1 \leq t \leq d$, let $0 \leq i^{t}_{1}, \ldots, i^{t}_k \leq k$ be arbitrary, and let $f_t: M^{y_{i^t_{1}}} \times \ldots \times M^{y_{i^{t}_{k}}} \to M^{x_t}$ be an arbitrary $\mathcal{L}'$-definable $k$-ary function. Then the formula 
	$$\psi\left( y_0; y_1, \ldots, y_k \right) := \varphi \left(f_1 \left(y_{i^1_1}, \ldots, y_{i^1_k} \right), \ldots, f_d \left(y_{i^d_1}, \ldots, y_{i^d_k} \right) \right)$$
	 is $k$-dependent.
\end{theorem}

\begin{proof}

If $k=1$, group the coordinates $x_t$ according to whether $f_t$ depends on $y_0$ or on $y_1$. Then $\psi$ is obtained from an NIP formula by coordinatewise preimages, and is therefore NIP. Hence assume $k\geq2$.

Let $\varphi(x_1, \ldots, x_d)$ and $(f_t)_{1 \leq t \leq d}$ be given. For $t\in[d]$, let
\[
r(t):=\min\bigl(\{0,\ldots,k\}\setminus\{i^t_1,\ldots,i^t_k\}\bigr),
\]
which is well-defined because the displayed set of indices has at most $k$ elements. For $0\leq r\leq k$, put
\[
S_r:=\{t\in[d]:r(t)=r\}.
\]
Thus the sets $S_0,\ldots,S_k$ partition $[d]$.

Let $x'_r:=(x_t:t\in S_r)$, with $t$ taken in increasing order. Define the $k$-ary $\mathcal L'$-definable function
\[
f'_r:\prod_{\substack{0\leq s\leq k\\s\neq r}}M^{y_s}\longrightarrow M^{x'_r}
\]
by
\[
f'_r\bigl((y_s:s\in\{0,\ldots,k\}\setminus\{r\})\bigr)
:=\bigl(f_t(y_{i^t_1},\ldots,y_{i^t_k}):t\in S_r\text{ in increasing order}\bigr).
\]

Consider the $\mathcal{L}$-formula $\varphi'(x'_0, \ldots, x'_k) := \varphi(x_1, \ldots, x_d)$. Then 
\begin{equation}\label{eq: comp lem 1}
	\models \psi(y_0; y_1, \ldots, y_k) \leftrightarrow \varphi'\left( f'_0 \left((y_s)_{ s \in \{0,\ldots,k\} \setminus \{0\}} \right),  \ldots, f'_k \left((y_s)_{ s \in \{0,\ldots,k\} \setminus \{k\}} \right)\right).
\end{equation}

In order to show that $\psi(y_0;y_1, \ldots, y_k)$ is $k$-dependent, we check that it satisfies the criterion in Proposition \ref{prop: type count}(3). 
Consider the $\mathcal{L}$-formula 
$$\varphi''(x'_k, x'_1, \ldots, x'_{k-1} , x'_0) := \varphi'(x'_0, x'_1, \ldots, x'_{k-1}, x'_k).$$

Let $f_{\varphi''}: \mathbb{N} \to \mathbb{N}$, $n_{\varphi''} \in \mathbb{N}, (\Delta_{\varphi'',n})_{ n \in \mathbb{N}}$ with each $\Delta_{\varphi'',n}$ a finite set of $\mathcal{L}$-formulas,  and $\varepsilon_{\varphi''} > 0$ be as given for $\varphi''$ by $(1)_k$ in Lemma \ref{lem: comp lemma induction}. We claim that $\psi(y_0;y_1, \ldots, y_k)$ satisfies $(\dagger)_{f_{\varphi''}, \varepsilon_{\varphi''}}$ with respect to $n^* := n_{\varphi''}$.

Indeed, assume that $n^* \leq n \leq m \in \mathbb{N}$ and $I_1, \ldots, I_k$ are mutually $\mathcal{L}'$-indiscernible sequences with $I_t \subseteq M^{y_t}$ for $1 \leq t \leq k$, $I_t = \left(a^t_i : i \in [n] \right)$ for $1 \leq t \leq k-1$, $I_k = \left(a^k_i : i \in [m] \right)$, and $b \in M^{y_0}$ is an arbitrary tuple. We define:
\begin{itemize}
	\item $\alpha_{\bar{i}} := f'_k(b, a^1_{i_1}, \ldots, a^{k-1}_{i_{k-1}})$ for every $\bar{i}=(i_1, \ldots, i_{k-1}) \in [n]^{k-1}$;
	\item $\gamma_{\bar{i}} := f'_0(a^1_{i_1}, \ldots, a^k_{i_k})$ for every $\bar{i} = (i_1, \ldots, i_k) \in [n]^{k-1} \times [m]$;
	\item $\beta^t_{\bar{i}} := f'_t(b, a^1_{i_1}, \ldots, a^{t-1}_{i_{t-1}}, a^{t+1}_{i_{t+1}}, \ldots, a^{k}_{i_{k}})$ for every $1 \leq t \leq k-1$ and every tuple $\bar{i} = (i_1, \ldots, i_{t-1}, i_{t+1}, \ldots, i_k) \in [n]^{k-2} \times [m]$.
\end{itemize}

The array $\bar{\gamma}=(\gamma_{\bar{i}}:\bar{i}\in[n]^{k-1}\times[m])$ is $\mathcal{L}'$-indiscernible by mutual $\mathcal{L}'$-indiscernibility of the $I_t$, so in particular it is $\Delta_{\varphi'',n}$-indiscernible. By \eqref{eq: comp lem 1} and the definition of $\varphi''$, for every $\bar{i}=(i_1, \ldots, i_{k-1}) \in [n]^{k-1}$ and $j \in [m]$ we have
$$\models \psi\left(b;a^1_{i_1}, \ldots, a^{k-1}_{i_{k-1}}, a^{k}_{j} \right) \iff \models \varphi'' \left(\alpha_{\bar{i}}, \beta^1_{\bar{i}_{[k-1]\setminus \{1\}} + (j)}, \ldots, \beta^{k-1}_{\bar{i}_{[k-1]\setminus \{k-1\}} + (j)}, \gamma_{\bar{i} + (j)} \right).$$

In particular, for every $j \in [m]$, the type $\tp_{\psi}(a^k_j /b, I_1, \ldots, I_{k-1})$ is determined by the set $S^{\varphi''}_{j}$ as in Lemma \ref{lem: comp lemma induction}, $(1)_k$. Hence there exists an interval $J\subseteq[m]$ with $|J|\geq m/f_{\varphi''}(n)-1$ such that the family $\mathcal{F}^{\varphi''}_J:=\{S^{\varphi''}_j:j\in J\}$ has cardinality $<2^{n^{k-1-\varepsilon_{\varphi''}}}$. Thus $|S_{\psi,J}(b,I_1,\ldots,I_{k-1})|<2^{n^{k-1-\varepsilon_{\varphi''}}}$, as wanted.
\end{proof}

\subsection{Discussion}\label{sec: comp lem discussion}

		In our proof of Lemma \ref{lem: comp lemma induction} we did not try to optimize any of the bounds, and we expect that substantially better bounds are possible.

We show how to obtain a polynomial bound for both $f_\varphi$ and $|\mathcal{F}^\varphi_J|$ in $(1)_2$ of Lemma \ref{lem: comp lemma induction} with a different argument (but do not pursue the general case here).

\begin{lemma}\label{lem: array shatter poly bound}
	Let $\varphi(x,y,z) \in \mathcal{L}$, and assume that every $\mathcal{L}$-formula is NIP. There are $d=d(\varphi)\in\mathbb N$ and a finite set of formulas $\Delta=\Delta(\varphi)$ such that the following holds. Let $n\leq m$, let $(\alpha_i:i\in[n])$ and $(\beta_j:j\in[m])$ be arbitrary sequences, and let $(\gamma_{i,j}:(i,j)\in[n]\times[m])$ be a $\Delta$-indiscernible array. For $j\in[m]$, let $S_j:=\{i\in[n]:\models\varphi(\alpha_i,\beta_j,\gamma_{i,j})\}$. If $n$ is sufficiently large, then there is an interval $J\subseteq[m]$ of length at least $m/n^d$ such that $|\{S_j:j\in J\}|\leq n^d$.
\end{lemma}
\begin{proof}
 Apply Fact \ref{fac: UDTFS} to the NIP formula $\varphi'(y;xz):=\varphi(x,y,z)$. There are $d_1\in\mathbb N$ and an $\mathcal L$-formula $\theta(x,z;(x_1,z_1),\ldots,(x_{d_1},z_{d_1}))$ such that, for each $j\in[m]$, there are indices $i_{j,1},\ldots,i_{j,d_1}\in[n]$ satisfying, for every $i\in[n]$,
 \[
 i\in S_j\quad\Longleftrightarrow\quad
 \models\theta\left(\alpha_i,\gamma_{i,j};
 \alpha_{i_{j,1}},\gamma_{i_{j,1},j},\ldots,
 \alpha_{i_{j,d_1}},\gamma_{i_{j,d_1},j}\right).
 \]
 
 By assumption, $\theta$ is NIP under every partition of its variables into two groups. Let $d_2=d_2(\theta)$ and a finite set $\Delta_0$ be as given by Lemma \ref{lem: fin shrinking}. Choose the finite set $\Delta$ in the statement large enough that, for every equality pattern among $i,i_1,\ldots,i_{d_1}$, $\Delta$-indiscernibility of $(\gamma_{i,j})$ implies $\Delta_0$-indiscernibility of the grouped sequence below. Fix $i\in[n]$ and $\bar i=(i_1,\ldots,i_{d_1})\in[n]^{d_1}$. The sequence of tuples
 $$\bigl((\gamma_{i,j},\gamma_{i_1,j},\ldots,\gamma_{i_{d_1},j}):j\in[m]\bigr)$$
 is $\Delta_0$-indiscernible by indiscernibility of the array $(\gamma_{i,j})$. Hence Lemma \ref{lem: fin shrinking} provides a tuple $\bar j^{i,\bar i}$ of at most $d_2$ elements of $[m]$ such that equality of the order types of $j,j'$ over $\bar j^{i,\bar i}$ implies
 \[
 \theta(\alpha_i,\gamma_{i,j};\alpha_{i_1},\gamma_{i_1,j},\ldots,
 \alpha_{i_{d_1}},\gamma_{i_{d_1},j})
 \leftrightarrow
 \theta(\alpha_i,\gamma_{i,j'};\alpha_{i_1},\gamma_{i_1,j'},\ldots,
 \alpha_{i_{d_1}},\gamma_{i_{d_1},j'}).
 \]
There are $n^{d_1+1}$ choices of $(i,\bar i)$. Let $\bar j$ enumerate all elements appearing in the corresponding tuples $\bar j^{i,\bar i}$. Then $|\bar j|\leq d_2n^{d_1+1}$. A largest $\qftp_<(-/\bar j)$-class $J\subseteq[m]$ is an interval and has
\[
|J|\geq\frac{m}{2d_2n^{d_1+1}+1}\geq\frac{m}{n^{d_1+2}}
\]
for all sufficiently large $n$. For $j,j'\in J$, the preceding equivalence holds for every $i\in[n]$ and every $\bar i\in[n]^{d_1}$. Consequently,
\[
(i_{j,1},\ldots,i_{j,d_1})=(i_{j',1},\ldots,i_{j',d_1})
\quad\Longrightarrow\quad S_j=S_{j'}.
\]
There are at most $n^{d_1}$ possible anchor tuples. Taking $d:=d_1+2$ proves the lemma for all sufficiently large $n$.
 \end{proof}
 
 \begin{problem}
 
 \begin{enumerate}
	 	\item Obtain better bounds in $(1)_k$ for general $k\geq2$.
	 \item In the Composition Lemma, is it sufficient to assume that the formula $\varphi$ is NIP, as opposed to assuming that the complete theory of the $\mathcal{L}$-reduct of $\mathcal{M}$ is NIP?
 \end{enumerate}
  \end{problem}
 
 \noindent We conjecture a generalization of Theorem \ref{Composition Lemma} (which corresponds to the case $n = 1$; the bound $kn$ is clearly optimal): 
 \begin{conj}
 Let $\mathcal{M}$ be an $\mathcal{L}'$-structure such that its reduct to a language $\mathcal{L} \subseteq \mathcal{L}'$ is $n$-dependent. Let $d,k \in \mathbb{N}$, let $\varphi(x_1, \ldots, x_d)$ be an $\mathcal{L}$-formula, and let $y_0, \ldots, y_{kn}$ be arbitrary tuples of variables. For each $1 \leq t \leq d$, let $0 \leq i^{t}_{1}, \ldots, i^{t}_k \leq kn$ be arbitrary, and let $f_t: M^{y_{i^t_{1}}} \times \ldots \times M^{y_{i^{t}_{k}}} \to M^{x_t}$ be an arbitrary $\mathcal{L}'$-definable $k$-ary function. Then the formula 
	$$\psi\left( y_0; y_1, \ldots, y_{kn} \right) := \varphi \left(f_1 \left(y_{i^1_1}, \ldots, y_{i^1_k} \right), \ldots, f_d \left(y_{i^d_1}, \ldots, y_{i^d_k} \right) \right)$$
	 is $kn$-dependent.
 \end{conj}
 
 \begin{remark}
 An analogue of the Composition Lemma for NFOP$_2$, with $k=2$ and the $\mathcal{L}$-reduct of $\mathcal{M}$ assumed stable, was proved in \cite{abd2023higher}. A preliminary version of this
paper also contained an analogous result for the stronger property NOP$_2$ (in the sense of
Takeuchi), this will be included in future work.
 \end{remark}

\section{Alternating non-degenerate $n$-linear forms are $n$-dependent and NSOP$_1$}\label{sec: multilin n-dep and NSOP1}
\subsection{$n$-dependence}
In this section we demonstrate $n$-dependence of certain theories of alternating non-degenerate $n$-linear forms over NIP fields, generalizing the $n=2$ case established in \cite[Theorem 6.3]{chernikov2021n}. 
The proof is similar to the binary case, with the two main new ingredients being the relative quantifier elimination result for non-degenerate $n$-linear forms (Theorem \ref{thm: QE for multilinear forms}, generalizing Granger's result \cite{granger1999stability} in the case $n=2$) and the $n$-ary Composition Lemma (Theorem \ref{Composition Lemma}, generalizing \cite[Theorem 5.12]{chernikov2021n} in the case $n=2$). Our proof in \cite[Theorem 6.3]{chernikov2021n} utilized additional analysis of generalized indiscernibles in bilinear spaces and a certain simplification of terms procedure in order to reduce the question to an application of the Composition Lemma (see \cite[Section 6.3]{chernikov2021n}), working in a language without the functions $f_i^p$ (which turned out to be insufficient for Granger's quantifier elimination, see the discussion in the introduction).
Later \cite{abd2023higher} proposed a streamlined argument simplifying these last two points using existing results on pairs of fields and vector spaces. While a direct generalization of our proof can be carried out in the language expanded by the functions $f_i^p$, we use the same simplification here.

 The following is the main result of the section (see Section \ref{sec: QE} for notation).
\begin{theorem}\label{thm: Granger}
Let $n \geq 2$ and let $T$ be a theory of \emph{infinite-dimensional} $n$-linear $K$-spaces (in the language $\mathcal{L}^{K}_{\theta,f}$).
\begin{enumerate}
	\item Assume that $T$ eliminates quantifiers in the language $\mathcal{L}^{K}_{\theta,f}$. If $K$ is NIP, then $T$ is $n$-dependent (and strictly $n$-dependent if the form is generic).  In particular, if $K$ is NIP, the theory of non-degenerate alternating forms $\prescript{}{\Alt}T^K_n$ is $n$-dependent.
	\item For every $k\geq1$, if $K$ has $\IP_k$ and the form is generic, then $T$ has $\IP_{nk}$.
\end{enumerate}
\end{theorem}
\begin{remark}
If we instead consider $T$ to be a theory of \emph{$m$-dimensional} (for some $m \in \omega$) $n$-linear $K$-spaces $(V,K,\langle - , \dots, - \rangle_n)$, then, for every $k\geq1$, the field $K$ is $k$-dependent if and only if $T$ is $k$-dependent. Indeed, a model of $T$ is interpretable in $K$ using $K^m\cong V$ as follows. Interpreting the vector-space structure is immediate. Let $(e_1, \dots, e_m)$ be the standard basis of $K^m$. The $n$-linear form is completely determined by fixing $k_{i_1, \dots, i_n} := \langle e_{i_1}, \dots , e_{i_{n}} \rangle_n$ for all $1 \leq i_1, \dots, i_n \leq m$. Let $\pi_i: K^m \rightarrow K$ be the projection map onto the $i$th coordinate. Then, for any $v_1, \dots,v_n \in K^m$, we have $$\langle v_1, \dots, v_n \rangle_n= \sum_{i_1, \dots,i_n=1}^m \pi_{i_1}(v_1) \dots \pi_{i_n}(v_{n})k_{i_1, \dots, i_n},$$ which is definable using $\left\{k_{i_1, \dots, i_n}: 1 \leq i_1, \dots, i_n \leq m\right\}$ as parameters.
\end{remark}

We will use a result on tameness of algebraically closed fields with a distinguished subfield named by a predicate from \cite{d2024algebraically}.
\begin{fact}\label{fac: pairs of fields}
Let $T$ be a theory of fields (possibly incomplete), in a language $\mathcal{L}$ expanding the language $\mathcal{L}_{\ring}$ of rings. Let $\mathcal{L}^P := \mathcal{L} \cup \{ P \} $ with $P$ a new unary predicate, and let $\ACF_T$ be an expansion of $\ACF$ in the language $\mathcal{L}^P$ by the axioms expressing that all relations in $\mathcal{L} \setminus \mathcal{L}_{\ring}$ are trivial unless all of the variables are in $P$, that $P$ is a subfield and a model of $T$, and that the degree of the field extension of the universe over $P$ is infinite.
\begin{enumerate}
	\item \cite[Proposition 4.1]{d2024algebraically} If $T$ is complete, then $\ACF_T$ is also complete.
	\item \cite[Theorems 5.24/5.34/5.9/5.13]{d2024algebraically} If $T$ is stable/NIP/simple/NSOP$_1$, then $\ACF_T$ is also stable/NIP/simple/NSOP$_1$.
\end{enumerate}

\end{fact}
From this we immediately have:
\begin{cor}\label{cor: VS's are NIP}
If $\Th(K)$ is stable/NIP/simple/NSOP$_1$, then the theory $T^K_{\VS, \infty}$ of infinite-dimensional vector spaces over fields elementarily equivalent to $K$ (see Remark \ref{rem: VS complete and omega-cat}) is also stable/NIP/simple/NSOP$_1$.
\end{cor}
\begin{proof}
		Given a field $K$, take an algebraically closed extension $F/K$ of infinite degree. Then $(F,K)\models\ACF_{\Th(K)}$ (with $\mathcal L$ in Fact \ref{fac: pairs of fields} taken to be $\mathcal L^K$ from Definition \ref{def: language of bilin forms}), and $(F,K)$ interprets a model of $T^K_{\VS,\infty}$. The claim follows from completeness of $T^K_{\VS,\infty}$ (Remark \ref{rem: VS complete and omega-cat}) and preservation of stability, NIP, simplicity, and NSOP$_1$ under interpretations and reducts.
\end{proof}

The following lemma is an elaboration of the term-reduction
argument that we have used in our proof of the bilinear case in \cite[Section 6.3]{chernikov2021n}.

\begin{fact}\cite[Lemma 5.9]{abd2023higher}\label{fac: term reduction}
	Let $n\geq2$, and let $T$ be a theory of $n$-linear $K$-spaces in $\mathcal L^K_{\theta,f}$. For any $\mathcal{L}^{K}_{\theta,f}$-term $t(\bar{x})$, with $\bar{x} = \bar{x}_V^\frown \bar{x}_K$, variables of sorts $V$ and $K$ respectively, $\bar{x}_V = \left(x^V_1, \ldots, x^V_m \right)$ with $x_i^V$ singleton variables, there is some $\mathcal{L}^K_{\VS} = \left( \mathcal{L}^{K}_{\theta,f} \setminus \left\{ \Bln \right\} \right)$-term $t'(\bar{x}, \bar{u})$ with $\bar{u}$ of sort $K$ so that $$T \models \forall \bar{x} \left( t(\bar{x}) = t'\left( \bar{x}, \left\langle \bar{x}_V  \right\rangle_n \right)  \right),$$
	where $\left\langle \bar{x}_V  \right\rangle_n := \left( \left \langle x^V_{i_1}, \ldots, x^V_{i_n}\right \rangle_n : (i_1, \ldots, i_n) \in \set{m}^n \right)$.
\end{fact}

We can now prove the main result of the section.
\begin{proof}[Proof of Theorem \ref{thm: Granger}]
\textbf{(1)}	By assumption $T$ has quantifier elimination in the language $\mathcal{L}^{K}_{\theta,f}$, and $n$-dependent formulas are closed under Boolean combinations (Fact \ref{fac: props of n-dependent formulas}). Hence it suffices to show that every atomic $\mathcal{L}^{K}_{\theta,f}$-formula $\varphi(y_1, \ldots, y_{n+1})$ with $y_1, 
	\ldots, y_{n+1}$ arbitrary finite tuples of variables, is $n$-dependent in $T$. From the definition of $\mathcal{L}^{K}_{\theta,f}$, $\varphi(y_1, \ldots, y_{n+1})$ has the form $R \left(t_1 \left(\bar{x}\right), \ldots, t_d \left(\bar{x}\right) \right)$ for some  $d \in \omega$, $\mathcal{L}^{K}_{\theta,f}$-relation symbol $R$, $\mathcal{L}^{K}_{\theta,f}$-terms $t_1, \ldots, t_d$ and $\bar{x} := y_1^{\frown} \ldots^{\frown} y_{n+1}$. Permuting the variables if necessary, we have $\bar{x}  = \bar{x}_V^\frown \bar{x}_K$, variables of sorts $V$ and $K$ respectively, and $\bar{x}_V = \left(x^V_1, \ldots, x^V_m \right), \bar{x}_K = \left(x^K_1, \ldots, x^K_{\ell} \right)$ with $m, \ell \in \omega$ and  $x_i^V, x_i^K$ singleton variables. 
By Fact \ref{fac: term reduction} there are $\mathcal{L}^K_{\VS}$-terms $t'_j(\bar{x}, \bar{u})$ for $j \in \set{d}$  so that
\begin{gather*}
		\text{for every }j\in\set d,\quad
		T \models \forall \bar{x} \Big( t_j(\bar{x}) = t'_j\left( \bar{x}, \left\langle \bar{x}_V  \right\rangle_n \right)  \Big)\text{; hence}\\
T \models \forall \bar{x} \Big( R \left(t_1 \left(\bar{x}\right), \ldots, t_d \left(\bar{x}\right) \right) \leftrightarrow R \left(t'_1\left( \bar{x}, \left\langle \bar{x}_V  \right\rangle_n \right), \ldots, t'_d\left( \bar{x}, \left\langle \bar{x}_V  \right\rangle_n \right) \right) \Big).
\end{gather*}
Let $\theta (\bar{x}, \bar{u})$ be the $\mathcal{L}^K_{\VS}$-formula $R \left(t'_1\left( \bar{x}, \bar{u} \right), \ldots, t'_d\left( \bar{x}, \bar{u} \right) \right)$. Then $\varphi(y_1, \ldots, y_{n+1})$ is equivalent to $\theta \left( \bar{x}, \left\langle \bar{x}_V  \right\rangle_n \right)$ in $T$.

Each singleton $x^V_i, i \in \set{m}$ and $x^K_i, i \in \set{\ell}$ appears in some tuple $y_{t_i}, t_i \in \set{n+1}$. Hence for each $(i_1, \ldots, i_n) \in \set{m}^n$, adding dummy variables if necessary, each of the terms $\left \langle x^V_{i_1}, \ldots, x^V_{i_n}\right \rangle _n, x^V_i, x^K_i$ can be viewed as a term $g(y_{t_1}, \ldots, y_{t_n})$ for some $(t_1, \ldots, t_n) \in \set{n+1}^n$. 

As $\varphi(y_1, \ldots, y_{n+1})$ is $T$-equivalent to $\theta \left( \bar{x}, \left\langle \bar{x}_V  \right\rangle_n \right)$, it is $T$-equivalent to
\[
\theta'\Bigl(g_1((y_i : i \in s_1)),\ldots,
g_D((y_i : i \in s_D))\Bigr),
\]
where $\theta'$ is an $\mathcal L^K_{\VS}$-formula obtained from $\theta$ by an appropriate repartition of its variables, $D\in\omega$, $s_h \in[n+1]^n$, and $g_h$ is an $\mathcal L^K_{\theta,f}$-term, for each $h\in[D]$. As $T^K_{\VS,\infty}$ is the $\mathcal L^K_{\VS}$-reduct of $T$, and $T^K_{\VS,\infty}$ is NIP by assumption and Corollary \ref{cor: VS's are NIP}, Theorem \ref{Composition Lemma} implies that $\varphi(y_1,\ldots,y_{n+1})$ is $n$-dependent.

And when the form is generic, it is \emph{strictly} $n$-dependent.  Indeed, fix $p\geq1$ and
choose vectors
\[
 (a^t_i:1\leq t\leq n-1,\ i\in\set{p})
\]
which are collectively linearly independent.  Then the elements
\[
 t_{i_1,\ldots,i_{n-1}}
 :=\overline{a^1_{i_1}\ten\cdots\ten a^{n-1}_{i_{n-1}}}
 \quad ((i_1,\ldots,i_{n-1})\in\set{p}^{n-1})
\]
are linearly independent in $\lozenge^{n-1}V$ (for either the alternating
or the symmetric power).  For every
$S\subseteq\set{p}^{n-1}$, genericity gives $b_S\in V$ such that
\[
 \langle a^1_{i_1},\ldots,a^{n-1}_{i_{n-1}},b_S\rangle_n
 =
 \begin{cases}
 1,&(i_1,\ldots,i_{n-1})\in S,\\
 0,&(i_1,\ldots,i_{n-1})\notin S.
 \end{cases}
\]
Thus the formula
$\langle y_1,\ldots,y_{n-1},x\rangle_n=1$ has $\IP_{n-1}$.

\noindent \textbf{(2)} Fix $k\geq1$ and assume that $K$ has $\IP_k$. Then, by
Fact \ref{fac: props of n-dependent formulas}(2),(3) it must be witnessed by some $\mathcal{L}^K$-formula $\varphi(\bar{x};y_1, \ldots, y_k)$ with each $y_i$ a single variable. So for any $p \geq 1$ and $1 \leq \ell \leq k$  we can find sequences 
$$\left(c^{\ell}_{(i_1^{\ell},\dots, i^{\ell}_n)} : (i_1^{\ell},\dots, i^{\ell}_n) \in [p]^n \right)$$
 with $[p]^n$
ordered lexicographically, and with all the  $c^{\ell}_{(i_1^{\ell},\dots, i^{\ell}_n)}$ pairwise distinct elements of $K$, such that for every $I\subseteq ([p]^n)^k$ there is some $\bar{e}_I$ satisfying 
$$\models \varphi \left(\bar{e}_I; c^1_{(i_1^{1},\dots, i^{1}_n)},\ldots, c^k_{(i_1^{k},\dots, i^{k}_n)} \right) \iff \left((i_1^{1},\dots, i^{1}_n), \ldots, (i_1^{k},\dots, i^{k}_n) \right) \in I.$$
As the space $V$ is infinite-dimensional,  we can choose  
$$\bar{a} := \left(a^{\ell}_{t,i} : 1 \leq \ell \leq k,  1 \leq t \leq n-1, i \in [p] \right)$$ 
a finite tuple consisting of linearly independent elements from $V$. 
%
 
 Since the $n$-linear form is  generic, and $\bar{a}$ is linearly independent, for each $1 \leq \ell \leq k$ and $j \in [p]$ there exists some $b^{\ell}_j \in V$ satisfying
 $$ \left \langle a^{\ell}_{1,i_1}, \dots, a^{\ell}_{n-1,i_{n-1}} , b^{\ell}_j \right \rangle_n = c^{\ell}_{(i_1,\dots, i_{n-1}, j)} $$
 for all $(i_1,\dots, i_{n-1}) \in [p]^{n-1}$. But then, identifying $([p]^n)^k$ with $[p]^{nk}$, for any set $I \subseteq [p]^{nk}$, we have 
\begin{align*}
	&\models \varphi \left(\bar{e}_I, \left \langle a^1_{1,i_{1}^1}, \dots, a^1_{n-1,i^{1}_{n-1}},  b^1_{j_1} \right \rangle_n, \ldots, \left \langle a^k_{1,i^{k}_1}, \dots, a^k_{n-1,i^{k}_{n-1}},  b^k_{j_k} \right \rangle_n \right)\\
	 \iff & \left(i_{1}^1, \dots, i^{1}_{n-1}, j_1 ; i^{2}_1, \dots, i^{2}_{n-1}  ,j_2; \dots; i^{k}_1, \dots, i^{k}_{n-1},j_k \right) \in I.
	 \end{align*}
As $p$ was arbitrary,  we conclude that the formula $$\psi(\bar{x};y_1, \ldots, y_{nk}) := \varphi \left(\bar{x}; \langle y_1,\dots,y_{n} \rangle_{n}, \ldots, \langle y_{n(k-1)+1}, \ldots, y_{nk} \rangle_n \right)$$ has $\IP_{nk}$.
\end{proof}

\subsection{NSOP$_1$}\label{sec: NSOP1 preserved}

In this section we consider another model-theoretic tameness property, NSOP$_1$, orthogonal to the $n$-dependence hierarchy, in the context of multilinear forms. We refer to \cite{chernikov2016model, kaplan2020kim, kaplan2021transitivity, kaplan2019local, dobrowolski2022independence, chernikov2023transitivity} for the recently developed basic theory of NSOP$_1$ and Kim-independence.
Preservation of NSOP$_1$ in bilinear forms over NSOP$_1$ fields and related questions has been considered in various contexts, starting with \cite{chernikov2016model}, where it was demonstrated (relying on Granger's results \cite{granger1999stability}) that if $K$ is an algebraically closed field, then the theory of non-degenerate symmetric or alternating bilinear forms over $K$ is NSOP$_1$. Following this, \cite{kaplan2020kim} proposed a description of Kim-independence over models, and \cite{met2023sets} proposed corrections to that description and a generalization from models to arbitrary sets. A variant incorporating these corrections, and also allowing real closed fields, was given in \cite{kruckman2024new}; a different argument was given in \cite{bossut2023note}. The proofs in \cite{kaplan2020kim, met2023sets, kruckman2024new} contain gaps:
\begin{remark}\label{rem: errors in proofs}
	\begin{enumerate}
		\item Some issues with the proof of \cite[Proposition 9.37]{kaplan2020kim} are already pointed out in \cite{met2023sets}, which proposed a correct description of Kim-independence.
		\item The issue with the proposed correction in the proof of \cite[Proposition 8.12]{met2023sets} is that, following the notation there, $\textrm{Lin}_{L}(A' \cup B_i)$ need not be a substructure containing $A' \cup B_i$ since it need not be closed under applying the bilinear form to pairs of vectors in $A' \times B_i$, so the claim that $A' B_0 \equiv A' B_i$ is not justified.
		\item In the proof of \cite[Theorem 4.9]{kruckman2024new}, the issue is that, following the notation there, it is possible that e.g.~$[a_{i'}, b_{0,k}] \in (B_1)_K \setminus (B_0)_K$, so assigning both $[a'_{i'}, b_{0,k}]^N := [a_{i'}, b_{0,k}]^{\mathbb{M}}$ and $[a'_{i'}, b_{1,k}]^N := [a_{i'}, b_{0,k}]^{\mathbb{M}}$ as suggested there, we again cannot conclude $A' B_1 \equiv_{M} A' B_0$. 
	\end{enumerate}
\end{remark}
 \noindent Here we correct the approach in \cite{kaplan2020kim, met2023sets, kruckman2024new}, generalizing to multilinear forms (and working over arbitrary NSOP$_1$ fields).

First we recall some basic notions and facts around NSOP$_1$.
\begin{defn} Let $M \models T$.
	\begin{enumerate}
		\item A formula $\varphi(x,a) \in \mathcal{L}(\mathbb{M})$ \emph{Kim-divides over $M$} if there is a global $M$-invariant type $q(y)$ extending $\tp(a/M)$ and a Morley sequence $(a_i)_{i \in \omega}$ in $q$ over $M$ (i.e.~$a_i \models q|_{M a_{<i}}$ for all $i \in \omega$; in particular $(a_i)_{i \in \omega}$ is indiscernible over $M$) so that $\{ \varphi(x,a_i) : i \in \omega \}$ is inconsistent.
			\item A formula $\varphi(x,a) \in \mathcal{L}(\mathbb{M})$ \emph{Kim-forks over $M$} if there are $k<\omega$ and formulas $\varphi_i(x,b_i)\in\mathcal L(\mathbb M)$ for $i<k$ such that each $\varphi_i(x,b_i)$ Kim-divides over $M$ and
			\[
			\varphi(x,a)\vdash\bigvee_{i<k}\varphi_i(x,b_i).
			\]
		\item A (partial) type $\pi(x)$ Kim-divides (Kim-forks) over $M$ if it implies a formula that Kim-divides (respectively, Kim-forks) over $M$.
		\item In this section, we will write $a \ind_M b$ to denote that $\tp(a/Mb)$ does not Kim-divide over $M$.
	\end{enumerate}
\end{defn}

\begin{remark}\label{rem: inv implies Kim indep}
		We note that in any theory, for any tuples $A,B,C$, $A \ind^{u}_C B$ (i.e.~$\tp(A/BC)$ is finitely satisfiable in $C$) implies $A \ind^{i}_C B$ (i.e.~$\tp(A/BC)$ extends to a global type invariant over $C$), which implies $A \ind^{f}_C B$ (i.e.~$\tp(A/BC)$ does not fork over $C$), which implies $A \ind_C B$ (i.e.~$\tp(A/BC)$ does not Kim-divide over $C$).
\end{remark}

\begin{fact}\label{fac: basic props of NSOP1}
	\begin{enumerate}
	\item If $T$ is  NSOP$_1$, then a formula $\varphi(x,a)$ Kim-divides over $M$ if and only if it Kim-forks over $M$, and this is witnessed by a Morley sequence of \emph{any} $M$-invariant global type  extending $\tp(a/M)$ \cite[Proposition 3.19]{kaplan2020kim}.
		\item A theory $T$ is NSOP$_1$ if and only if Kim-dividing is symmetric over models, i.e.~for any $M \models T$ and tuples $a,b$, $\tp(a/bM)$ does not Kim-divide over $M$ if and only if $\tp(b/aM)$ does not Kim-divide over $M$ \cite[Theorem 5.16]{kaplan2020kim}.
		\item A theory $T$ is simple if and only if $T$ is NSOP$_1$ and Kim-independence $\ind$ satisfies base monotonicity over models: whenever $M \preceq N \models T$, if $a \ind_M Nb$, then $a \ind_N b$  \cite[Proposition 8.8]{kaplan2020kim}.
	\end{enumerate}
\end{fact}

For the rest of the section, let $T := \prescript{}{\Alt}T^K_n$ be the theory of non-degenerate alternating forms over fields elementarily equivalent to $K$. 
 For a set $A \subseteq \mathbb{M}$ we write $A_K$ for the elements of $A$ of the field sort and $A_V$ for the elements of $A$ of the vector space sort. For 
$A \subseteq \mathbb{M}^V$, we write $\langle A \rangle$ to denote $\Span_{\mathbb{M}_K}(A)$ (so $\langle A \rangle$ can be large even when $A$ is small). For $A,B,C \subseteq \mathbb{M}$, we write 
$A \ind^V_{C} B$ to denote $\langle A_V \rangle \cap \langle B_V \rangle \subseteq \langle C_V \rangle$. We also write $A \ind^{K}_C B$ to denote that $A_{K}$ and $B_K$ are Kim--independent over $C_K$ in the reduct $\mathbb{M}_{K}$ to a model of $\Th(K)$. We will also use the notation from Definition \ref{def: language of bilin forms}.

\begin{fact}\label{fac: KR lemma}
If $\tp(A/MB)$ does not Kim--divide over $M$ in $T$ and $\Th(K)$ is NSOP$_1$, then $A \ind_M^K B$ and $A \ind_M^V B$.

\end{fact}
\begin{proof}
	This follows from the proof of \cite[Lemma 4.7]{kruckman2024new} (more directly here, since $\Th(K)$ is NSOP$_1$ and every coheir sequence witnesses Kim-dividing in $K$ by Fact \ref{fac: basic props of NSOP1}(1)), together with the proof of \cite[Lemma 4.8(3)]{kruckman2024new} (whose stated assumption $K\models\RCF$ is not used).
\end{proof}

\begin{remark}\label{rem: field sort stab emb}
	For any tuple $a$ in $\mathbb{M}_K$ and any $\mathcal{L}^{K}_{\theta,f}$-substructure $C$ of $\mathbb{M}$, $\tp(a/C)$ is determined by $\tp_{\mathcal{L}^K}(a/C_K)$. In particular, in any model of $T := \prescript{}{\Alt}T^{K}_n$, the field sort is stably embedded, and the induced structure is just the $\mathcal{L}^K$-structure of $\Th(K)$.
\end{remark}

\begin{proof}
Let $a'\in\mathbb M_K$ satisfy $\tp_{\mathcal L^K}(a'/C_K)=\tp_{\mathcal L^K}(a/C_K)$. Let $\chi(x;c)$ be a formula with $c$ a finite tuple from $C$, and let $C_0$ be the finitely generated substructure of $C$ containing $c$. By quantifier elimination (Theorem \ref{thm: QE for multilinear forms}) and Claim \ref{cla: multilin QE proof 1}, the partial $\mathcal L^K$-isomorphism fixing $(C_0)_K$ and sending $a$ to $a'$ extends to a partial $\mathcal L^K_{\theta,f}$-isomorphism fixing $C_0$. Hence $\chi(a;c)$ holds if and only if $\chi(a';c)$ holds. Since every formula uses only finitely many parameters, this proves the asserted determination of types and stable embeddedness.
%
\end{proof}

\begin{prop}\label{prop: char of Kim div}
		Assume that $\Th(K)$ is NSOP$_1$. Let $M \prec \mathbb{M}$ be a small model and $A, B \supseteq M$ small substructures. Then $\tp(A/B)$ does not Kim-divide over $M$ if and only if $A \ind^{K}_M B$ and $A \ind^{V}_M B$.
\end{prop}
\begin{proof}	
 If $\tp(A/MB)$ does not Kim-divide over $M$, then $A \ind^{K}_{M} B$ and $A \ind^{V}_M B$ by Fact \ref{fac: KR lemma}.
	
	For the other direction, assume $A \ind^{K}_{M} B$ and $A \ind^{V}_M B$.
		Let $\mathfrak q(\bar y)\in S_{\bar y}(\mathbb M)$ be an arbitrary global $M$-invariant extension of $\tp(B/M)$ (in the sense of $T$), and let $(B_i)_{i\in\omega}$ be a Morley sequence in $\mathfrak q$ over $M$. Let $B^*\models\mathfrak q$ in a larger monster model $\mathbb M'\succeq\mathbb M$ of $T$, and put
	$\mathfrak q_K(\bar y_K):=\tp_{\mathcal L^K}(B^*_K/\mathbb M_K)$. Then $\mathfrak q_K$ is a global type extending $\tp_{\mathcal L^K}(B_K/ M_K)$
	which is moreover $M_K$-invariant in the sense of $\Th_{\mathcal L^K}(K)$.
 Indeed, given any formula $\psi(\bar{y}_K,z) \in \mathcal{L}^K$ and tuples $c,c' \in \mathbb{M}_K$ with $c \equiv_{M_K}^{\mathcal{L}^K} c'$, by Remark \ref{rem: field sort stab emb} we have $c \equiv_{M} c'$. Hence by $M$-invariance of $\mathfrak q$ we have $\psi(\bar{y}_K,c) \in \mathfrak q_K \iff \psi(\bar{y}_K,c) \in \mathfrak q \iff \psi(\bar{y}_K,c') \in \mathfrak q \iff \psi(\bar{y}_K,c') \in \mathfrak q_K$. It follows that $((B_i)_K)_{i \in \omega}$ is a Morley sequence over $M_K$ of the global invariant type $\mathfrak q_K$ in the $\mathcal{L}^K$-reduct to a model of $\Th(K)$.

	Let $p(x;B) := \tp(A_K / B)$, and let $p'(x;B_K) := \tp_{\mathcal{L}^K}(A_K / B_K)$. As $B$ is a substructure of $\mathbb{M}$ by assumption, by Remark \ref{rem: field sort stab emb} we have $p'(x;B_K) \vdash p(x;B)$ in $T$.
		For each $i$, let $p(x;B_i)$ and $p'(x;(B_i)_K)$ be the images of $p(x;B)$ and $p'(x;B_K)$, respectively, under an $\mathcal L^K_{\theta,f}$-isomorphism fixing $M$ and sending $B$ to $B_i$. We still have $p'(x;(B_i)_K)\vdash p(x;B_i)$. Since $A\ind^K_MB$, $p'(x;B_K)$ does not Kim-divide over $M_K$ in $\Th(K)$, hence there is some $A'_K$ in $\mathbb{M}_K$ such that
\[
A'_K\models\bigcup_{i\in\omega}p'(x;(B_i)_K).
\]
Hence $A'_K\models\bigcup_{i\in\omega}p(x;B_i)$, so $A'_K B_i\equiv_M A_K B$ for every $i\in\omega$.
		For each $i \in \omega$, fix an $\mathcal{L}_{\theta,f}^K$-automorphism
	\begin{gather*}
			\sigma_i : A_K \cup B \to A'_K \cup B_i, \qquad \sigma_i \restriction_{M} = \id \restriction_{M}.
	\end{gather*}

	Let $C$ be the  $\mathcal{L}^{K}_{\theta,f}$-substructure of $\mathbb{M}$ generated by $A \cup B$, in particular $A_K \cup B_K \subseteq C_K$. As $\sigma_i \restriction_{A_K \cup B_K}: A_K \cup B_K \to  A'_K \cup (B_i)_K$ is elementary, we can extend it to an elementary map
	\begin{gather}
		\sigma^K_i : C_K \to K_i,  \  	\sigma^K_i \restriction_{A_K \cup B_K} = \sigma_i \restriction_{A_K \cup B_K}, \textrm{ in particular }  \sigma^K_i \restriction_{M_K} = \id \restriction_{M_K}  \label{eq: multilin form NSOP1 0}
	\end{gather}
	for some field $K_i \subseteq \mathbb{M}_K$ with $A'_K \cup (B_i)_K \subseteq K_i$.
	
		Choose a small $\widetilde M\preceq\mathbb M$ containing $\bigcup_{i\in\omega}K_i\cup\bigcup_{i\in\omega}B_i\cup A'_K\cup A_K$, and let $\widetilde K:=\widetilde M_K$. In particular, $\widetilde K\models\Th(K)$.

	Let $\overline{m} = (m_i)_{i < \alpha}$ be a tuple from $M_V$ which is a basis for $\langle M_V \rangle$ viewed as a vector space over $\mathbb{M}_K$. Similarly, let $\overline{a} = (a_i)_{i < \beta}$ be a tuple from $A_V \supseteq M_V$ such that $\overline{m} \, \overline{a} $ is a basis of $\langle A_V \rangle$ as a vector space over $\mathbb{M}_K$. Note that by Lemma \ref{lem QE1} (as $A$ is a substructure) also
	\begin{gather}
			\overline{m} \, \overline{a} \textrm{ is a basis for } A_V \textrm{ as a vector space over } A_K.
		\end{gather}
		Let $\overline{b} = (b_i)_{i < \gamma}$ be a tuple from $B_V \supseteq M_V$ such that $\overline{m}\,\overline{b}$ is a basis of $\langle B_V \rangle$ as a vector space over $\mathbb{M}_K$. 
		For each $i<\omega$ and $j<\gamma$, let $b_{i,j}:=\sigma_i(b_j)$, and let $\bar b_i:=(b_{i,j}:j<\gamma)$. Since $\sigma_i$ is $\mathcal L^K_{\theta,f}$-elementary and fixes $M$, it follows that $\bar m\,\bar b_i$ is a basis for $\langle(B_i)_V\rangle$ as a vector space over $\mathbb M_K$.
	 By assumption $A \ind^V_M B$, so 
	\begin{gather}
		\overline{m} \cup \overline{a} \cup \overline{b} \textrm{ is a basis for } \langle A_V \cup B_V \rangle \textrm{ as a vector space over } \mathbb{M}_K.
	\end{gather}
	
	As $(B_i)_{i \in \omega}$ is a Morley sequence in an $M$-invariant type, it is also a $\ind^V_M$-independent sequence (by Fact \ref{fac: KR lemma} and Remark \ref{rem: inv implies Kim indep}). In particular, the set $\left\{m_i\right\}_{i < \alpha} \cup \left\{ b_{i,j} : i \in \omega, j < \gamma\right\}$ is linearly independent over $\mathbb{M}_K$.
	
		 Let $\widetilde{V} := \Span_{\widetilde{K}}(\overline{m} (\overline{b}_i)_{i \in \omega})$, the vector space over $\widetilde{K}$ spanned by $\overline{m} (\overline{b}_i)_{i \in \omega}$. Let $\widetilde{N} := (\widetilde{V}, \widetilde{K})$. Then $\widetilde{N}$ is a substructure of $\mathbb{M}$ (by Lemma \ref{lem QE2} and the choice of $\widetilde{K}$) and $B_i \subseteq \widetilde{N}$ for all $i \in \omega$ (because the choice of $\widetilde{K}$ ensures that it contains all $f_\ell^p$-values on tuples from $((B_i)_V)_{i \in \omega}$). Note that $\widetilde{N}$ is small, and that the intended interpretations of all $\theta_p$ and $f_\ell^p$ in $\widetilde{N}$ agree with the interpretations of these symbols in $\mathbb{M}$ restricted to $\widetilde{N}$.
	
	Let $\bar{a}' = (a'_i)_{i < \beta}$ be a tuple of new vectors (not in $\mathbb{M}_V$) of the same length as $\bar{a}$. Let $W$ be the $\widetilde{K}$-vector space extending $\widetilde{V}$ with ordered basis $\overline{m} \cup \overline{a}' \cup (\overline{b}_i)_{i \in \omega}$, ordered as written with order type $\alpha+\beta+\gamma\cdot\omega$. We define a new $\mathcal{L}^{K}_{\theta,f}$-structure $N$ extending $\widetilde{N}$ with $N_K = \widetilde{N}_K = \widetilde{K}$ and $N_V = W$. The field structure on $\widetilde{K}$ is determined by $\widetilde{N}$ and the vector space structure has been determined, so it remains to define the $n$-linear form $\langle -, \ldots, - \rangle_n^{N}$ on $W$ extending the $n$-linear form on $\widetilde{V}$ determined by $\widetilde{N}$.
	
	Note that \emph{any} function on strictly increasing $n$-tuples of vectors in an ordered basis extends uniquely to an alternating form on the whole vector space.	
	We take $\langle -, \ldots, - \rangle_n^{N}$  to agree with $\langle -, \ldots, - \rangle_n^{\mathbb{M}}$ on all such $n$-tuples from $\overline{m} \cup (\overline{b}_i)_{i \in \omega}$.
	
	 For any $q, r, s \in \{0, \ldots, n\}$ with $q+r+s = n$ and $r \geq 1$ (so there is at least one element from $\bar{a}'$ in the tuple) and $\alpha_1 < \ldots < \alpha_{q} < \alpha$, $\beta_{1} < \ldots < \beta_{r} < \beta$, $\gamma_{1} < \ldots < \gamma_{s} < \gamma$ and any $i < \omega$ we define
 \begin{gather}
 	\langle  m_{\alpha_{1}}, \ldots,  m_{\alpha_q}, a'_{\beta_{1}}, \ldots, a'_{\beta_{r}}, b_{i, \gamma_{1}}, \ldots, b_{i, \gamma_{s}} \rangle_n^{N} := \label{eq: multilin form NSOP1 1}\\
 \sigma_i^K \left( 	\langle  m_{\alpha_{1}}, \ldots,  m_{\alpha_q}, a_{\beta_{1}}, \ldots, a_{\beta_{r}}, b_{\gamma_{1}}, \ldots, b_{\gamma_{s}} \rangle_n^{\mathbb{M}} \right) \in K_i \subseteq \widetilde{K}. \nonumber
 \end{gather}
	(When $s=0$, this prescription is independent of $i$: the displayed value lies in $A_K$, and all maps $\sigma_i^K\upharpoonright A_K$ send the fixed enumeration of $A_K$ to the same tuple $A'_K$.)
	For all other strictly increasing $n$-tuples from $\overline{m} \cup \overline{a}' \cup (\overline{b}_i)_{i \in \omega}$ (i.e.~containing $b_{i,*}$ and $b_{j,\ast}$ for some $i,j\in\omega$ with $i\neq j$) with at least one element in $\bar{a}'$, we define $\langle -, \ldots, - \rangle_n^{N}$ to be an arbitrary element in $\widetilde{K}$. This uniquely determines an alternating $n$-linear form $\langle -, \ldots, - \rangle_n^{N}$ on the vector space $W$ extending $\langle -, \ldots, - \rangle_n^{\widetilde{N}}$. The natural interpretations of all $\theta_p$ and $f_i^p$ on $(W,\widetilde{K})$ agree with those on $\widetilde{N}$ (by the choice of $\widetilde{K}$, $\widetilde{N}_{K} = N_K$). Thus:
  \begin{gather}
  \widetilde{N} \leq N \textrm{ (an } \mathcal{L}^{K}_{\theta,f} \textrm{-substructure).}
  \end{gather}

	   As $\widetilde{K} \models \Th(K)$, by Lemma \ref{lem: ext multilin to gen} and completeness of $T$ (Theorem \ref{thm: QE for multilinear forms}) there exists $N' = \left(N'_V,\widetilde{K} \right)$ such that $N' \models T$ (interpreting the $f_i^p$ and $\theta_p$ naturally in $N'$), $\dim(N'_V) \leq \dim(N_V) + \aleph_0$ (so $N'$ is still small) and $N \leq N'$ (an $\mathcal{L}^{K}_{\theta,f}$-substructure, using that the field is the same). By quantifier elimination in $T$ there is an $\mathcal{L}^{K}_{\theta,f}$-embedding $\iota: N' \to \mathbb{M}$ over $\widetilde{N}$ (see e.g.~\cite[Proposition 4.3.28]{marker2006model}).

	    Let $a''_i := \iota(a'_i)$ for $i < \beta$, and $\bar{a}'' := (a''_i : i < \beta)$. Let $A'_V:=\iota(\Span_{A'_K}(\overline m\,\overline a'))=\Span_{A'_K}(\overline m\,\overline a'')$ and $A':=(A'_V,A'_K)$.
   
   \begin{claim}\label{cla: elem maps in NSOP1}
   We have $A' B_i \equiv_{M} A B$ for all $i \in \omega$.
   \end{claim}
\begin{proof}

 Fix $i \in \omega$.

   Note that $\left( \Span_{C_K}(\overline{m} \overline{a} \overline{b}), C_K \right)$ is an $\mathcal{L}^{K}_{\theta,f}$-substructure (by Lemma \ref{lem QE2}, as $\overline{m} \overline{a} \overline{b}$  are $\mathbb{M}_K$-linearly independent and the choice of $C_K$) containing $A B$ (by the choice of $C_K$).
 
 Using \eqref{eq: multilin form NSOP1 1} and that  $\iota$ is an embedding which is an identity on $\widetilde{N}$ 
 we have
  \begin{gather}
 	\langle  m_{\alpha_{1}}, \ldots,  m_{\alpha_q}, a''_{\beta_{1}}, \ldots, a''_{\beta_{r}}, b_{i, \gamma_{1}}, \ldots, b_{i, \gamma_{s}} \rangle_n^{\mathbb{M}} =  \label{eq: multilin form NSOP1 3} \\
	\iota \left( \langle  m_{\alpha_{1}}, \ldots,  m_{\alpha_q}, a'_{\beta_{1}}, \ldots, a'_{\beta_{r}}, b_{i, \gamma_{1}}, \ldots, b_{i, \gamma_{s}} \rangle_n^{N} \right) = \nonumber \\ 
	\langle  m_{\alpha_{1}}, \ldots,  m_{\alpha_q}, a'_{\beta_{1}}, \ldots, a'_{\beta_{r}}, b_{i, \gamma_{1}}, \ldots, b_{i, \gamma_{s}} \rangle_n^{N} \in K_i. \nonumber
 \end{gather}

   We have that $\overline{m} \overline{a}'' \overline{b}_i$ is an $\mathbb{M}_K$-linearly independent set (as $\overline{m} \overline{a}' \overline{b}_i$ was linearly independent over $N'_K$ in $N'$ and the embedding $\iota$ preserves the $\theta_p$) and the values of $\langle -, \ldots, - \rangle_n$ on $\overline{m} \overline{a}'' \overline{b}_i$ are contained in $K_i$ by the above. Hence, using Lemma \ref{lem QE2} again, $\left( \Span_{K_i}(\overline{m} \overline{a}'' \overline{b}_i), K_i \right)$ is an $\mathcal{L}^{K}_{\theta,f}$-substructure containing $A' B_i$.

By \eqref{eq: multilin form NSOP1 0}, $\sigma_i^K:C_K\to K_i$ is an $\mathcal L^K$-isomorphism. Define $\sigma_i^V : \Span_{C_K}(\overline{m} \overline{a}\overline{b}) \to \Span_{K_i}(\overline{m} \overline{a}'' \overline{b}_i)$ to fix $\bar m$ pointwise, send $a_j$ to $a''_j$ for $j<\beta$ and $b_j$ to $b_{i,j}$ for $j<\gamma$, and extend by linearity using $\sigma_i^K$: for $q,r,s\in\omega$, indices $\alpha_j<\alpha$ for $j\in[q]$, $\beta_j<\beta$ for $j\in[r]$, and $\gamma_j<\gamma$ for $j\in[s]$, and coefficients $k_j\in C_K$ for $j\in[q]$, $k'_j\in C_K$ for $j\in[r]$, and $k''_j\in C_K$ for $j\in[s]$, we define
\begin{gather*}
		\sigma^V_i\left(\sum_{j \in [q]} k_j m_{\alpha_j} +  \sum_{j \in [r]} k'_j a_{\beta_j}  + \sum_{j \in [s]} k''_j b_{\gamma_j}  \right) := \\
		\sum_{j \in [q]} \sigma_i^K(k_j) m_{\alpha_j} +  \sum_{j \in [r]} \sigma_i^K(k'_j) a''_{\beta_j}  + \sum_{j \in [s]} \sigma_i^K(k''_j) b_{i, \gamma_j}.
\end{gather*}

Let $\widetilde{\sigma}_i := \sigma^V_i \cup \sigma_i^K$. We claim that 
\begin{gather*}
	\widetilde{\sigma}_i: \left( \Span_{C_K}(\overline{m} \overline{a} \overline{b}), C_K \right) \to \left( \Span_{K_i}(\overline{m} \overline{a}'' \overline{b}_i), K_i \right)
\end{gather*}
is an isomorphism of $\mathcal{L}^{K}_{\theta,f}$-structures.
By definition of $\widetilde{\sigma}_i$ and Lemma \ref{lem QE3}, it remains to verify that $\widetilde{\sigma}_i$ preserves $\langle -, \ldots, - \rangle_n$.
So we consider 
$$\langle  m_{\alpha_{1}}, \ldots,  m_{\alpha_q}, a_{\beta_{1}}, \ldots, a_{\beta_{r}}, b_{\gamma_{1}}, \ldots, b_{\gamma_{s}} \rangle_n$$ for $q, r, s \in \{0, \ldots, n\}$ with $q+r+s = n$, where the indices are strictly increasing within each block.

If $r=0$ then $\langle  m_{\alpha_{1}}, \ldots,  m_{\alpha_q}, b_{\gamma_{1}}, \ldots, b_{\gamma_{s}} \rangle_n^{\mathbb{M}} \in B_K$, so using $\sigma^K_i\upharpoonright_{A_K\cup B_K}=\sigma_i\upharpoonright_{A_K\cup B_K}$ and the fact that $\sigma_i\upharpoonright_B:B\to B_i$ is an $\mathcal{L}^{K}_{\theta,f}$-isomorphism fixing $M$, we have
\begin{gather*}
 	\widetilde{\sigma}_i \left( \langle  m_{\alpha_{1}}, \ldots,  m_{\alpha_q}, b_{\gamma_{1}}, \ldots, b_{\gamma_{s}} \rangle_n^{\mathbb{M}} \right) = \\
 	\sigma_i  \left( \langle  m_{\alpha_{1}}, \ldots,  m_{\alpha_q}, b_{\gamma_{1}}, \ldots, b_{\gamma_{s}} \rangle_n^{\mathbb{M}} \right) = \\
 \langle  \sigma_i(m_{\alpha_{1}}), \ldots, \sigma_i( m_{\alpha_q}), \sigma_i(b_{\gamma_{1}}), \ldots, \sigma_i(b_{\gamma_{s}}) \rangle_n^{\mathbb{M}} =\\
	 \langle  m_{\alpha_{1}}, \ldots, m_{\alpha_q}, b_{i, \gamma_{1}}, \ldots, b_{i, \gamma_{s}} \rangle_n^{\mathbb{M}}.
 \end{gather*}
 
 If $r \geq 1$, then using \eqref{eq: multilin form NSOP1 3} we have
 \begin{gather*}
 	\widetilde{\sigma}_i \left( \langle  m_{\alpha_{1}}, \ldots,  m_{\alpha_q}, a_{\beta_{1}}, \ldots, a_{\beta_{r}},  b_{\gamma_{1}}, \ldots, b_{\gamma_{s}} \rangle_n^{\mathbb{M}} \right) = \\
 	\sigma^K_i  \left(  \langle  m_{\alpha_{1}}, \ldots,  m_{\alpha_q}, a_{\beta_{1}}, \ldots, a_{\beta_{r}},  b_{\gamma_{1}}, \ldots, b_{\gamma_{s}} \rangle_n^{\mathbb{M}} \right)  \overset{\eqref{eq: multilin form NSOP1 1}}{=} \\
 	\langle  m_{\alpha_{1}}, \ldots,  m_{\alpha_q}, a'_{\beta_{1}}, \ldots, a'_{\beta_{r}}, b_{i, \gamma_{1}}, \ldots, b_{i, \gamma_{s}} \rangle_n^{N} \overset{\eqref{eq: multilin form NSOP1 3}}{=}\\
 		\langle  m_{\alpha_{1}}, \ldots,  m_{\alpha_q}, a''_{\beta_{1}}, \ldots, a''_{\beta_{r}}, b_{i, \gamma_{1}}, \ldots, b_{i, \gamma_{s}} \rangle_n^{\mathbb{M}}.\qedhere
 \end{gather*}
\end{proof}

		    Since $\mathfrak q$ was arbitrary, the claim shows that $\tp(A/BM)$ does not Kim-divide over $M$, as wanted.
\end{proof}

\begin{theorem}\label{thm: NSOP1 and Kim indep}
		If $\Th(K)$ is NSOP$_1$, then $\prescript{}{\Alt}T^{K}_n$ is also NSOP$_1$. Moreover, for any $M\models\prescript{}{\Alt}T^{K}_n$ and small substructures $A,B\supseteq M$, $\tp(A/B)$ does not Kim-divide over $M$ if and only if $A \ind^{K}_M B$ and $A \ind^{V}_M B$.
\end{theorem}
\begin{proof}
	By Proposition \ref{prop: char of Kim div} and Fact \ref{fac: basic props of NSOP1}(2).
\end{proof}

\begin{cor}\label{cor: multilin fin field simple}
	If the field $K$ is finite, then  $\prescript{}{\Alt}T^{K}_n$ is simple.
\end{cor}
\begin{proof}
Take any model $M$ and substructures $A,B \supseteq M$.
		As $K$ is finite, $M_K=\mathbb M_K=K$, and hence $A_K=B_K=K$. Since $K\ind^u_KK$, we have $K\ind^K_KK$ by Remark \ref{rem: inv implies Kim indep}, and in particular $A\ind^K_MB$. Hence, by Proposition \ref{prop: char of Kim div}, $\tp(A/MB)$ does not Kim-divide over $M$ if and only if $\Span_K(A_V)\cap\Span_K(B_V)\subseteq\Span_K(M_V)$; that is, if and only if $A_V$ and $B_V$ are independent over $M_V$ in the stable theory of infinite-dimensional vector spaces over finite fields. Thus Kim-independence in $\prescript{}{\Alt}T^K_n$ satisfies base monotonicity, so $\prescript{}{\Alt}T^K_n$ is simple by Fact \ref{fac: basic props of NSOP1}(3).
\end{proof}

\section{Invariant connected components $G^{\infty}$ in $n$-dependent groups}
\label{sec: inv conn comp}

\subsection{Invariant subgroups of bounded index}\label{sec: inv conn prelims}

Throughout the section, we let $T$ be a complete theory in a language $\mathcal{L}$. Let $G = G(\mathbb{M})$ be a type-definable group (to simplify the notation, over $\emptyset$) and $S \subseteq \mathbb{M}$ a parameter set.

\begin{defn}
	For a small set of parameters $S \subseteq \mathbb{M}$, we let: 
	\begin{enumerate}
		\item $G^0_S$ be the intersection of all $S$-definable subgroups of $G$ of finite index,
		\item $G^{00}_S$ be the intersection of all $S$-type-definable subgroups of $G$ of bounded index,
		\item $G^{\infty}_S$ (sometimes also denoted $G^{000}_S$) be the intersection of all $S$-invariant subgroups of $G$ of bounded index.
	\end{enumerate}
\end{defn}

We have $G^{\infty}_S \subseteq G^{00}_{S} \subseteq G^{0}_S$ (all of these coincide in stable theories, but the inclusions can be proper already in dependent theories \cite{conversano2012connected}); and as $S$ is small, $G^0_S, G^{00}_S$ are $S$-type-definable subgroups of $G$ of bounded index, and $G^{\infty}_S$ is an $S$-invariant subgroup of $G$ of bounded index. A fundamental fact about dependent groups is the ``absoluteness'' of their connected components:
\begin{fact}\label{fac: NIP G00 absolute}
	Let $T$ be dependent. Then for any small set of parameters $S$ we have $G^{00}_S = G^{00}_{\emptyset}$ (Shelah, \cite{shelah2008minimal}) and $G^{\infty}_S = G^{\infty}_{\emptyset}$ (Shelah \cite{MR3666349} in the abelian case, and Gismatullin \cite{Gismatullin2011} in general).
\end{fact}
\noindent This is no longer true in $2$-dependent groups (see Section \ref{sec: Ginfty in multilin}). However, the following ``relative absoluteness'' holds:
\begin{fact}(Shelah, \cite{MR3666349})\label{fac: Sh G00 2-dep}
	Let $T$ be $2$-dependent, $G$ an $S$-type-definable group, $\kappa := \beth_2(|S| + |T|)^{+}$, $\M \supseteq S$ a $\kappa$-saturated model, and $\bar{b}$ an arbitrary finite tuple in $\mathbb{M}$. Then 
	$$G^{00}_{\M \bar{b}} = G^{00}_{\M} \cap G^{00}_{C \bar{b}}$$
	for some $C \subseteq \M$ with $|C| < \kappa$.
\end{fact}

In \cite[Section 4]{chernikov2021n} we generalized Fact \ref{fac: Sh G00 2-dep} to $n$-dependent groups for arbitrary $n$. This required us to consider tuples in arbitrary (or $n$-dependent) theories that are sufficiently independent, in the sense of a notion of independence considered in \cite[Section 4]{chernikov2021n}.

\begin{defn}(\emph{$\kappa$-coheirs}) For any infinite cardinal $\kappa$, any model $\M$, a tuple $a$, and a set of parameters $B$, we write
	$$a \ind_{\M}^{u,\kappa} B$$
if every partial type $p_0(x)\subseteq\tp(a/B\M)$ of cardinality
$<\kappa$ is realized in $\M$.
\end{defn}
Thus $a\ind^{u, \aleph_0}_{\mathcal{M}}B \iff a\ind^{u}_{\mathcal{M}}B$, i.e.~if and only if $\tp(a/B\M)$ is finitely satisfiable in $\mathcal{M}$.  If $a$ is finite and $\kappa>|T|$ is regular, this definition is equivalent to requiring that $\tp(a/C)$ be realized in $\M$ for every $C\subseteq B\cup\M$ of cardinality $<\kappa$: one direction follows by restricting to the parameters occurring in a partial type, and the other because $|\tp(a/C)|<\kappa$.  All cardinals used below are chosen in this range.
Recall that for an infinite cardinal $\kappa$ and $n \in \omega$, the cardinal $\beth_n(\kappa)$ is defined inductively by $\beth_0(\kappa) := \kappa$ and $\beth_{n+1}(\kappa) := 2^{\beth_n(\kappa)}$. Then the Erd\H{o}s-Rado theorem says that $\left( \beth_r(\kappa) \right)^+ \rightarrow (\kappa^+)^{r+1}_{\kappa}$ for all infinite $\kappa$ and $r\in \omega$.

\begin{defn}(\emph{Generic position})\label{def: generic position}
Let $\M$
 be a small model, $A$ a subset of $\M$, and $\bar{b}_1, \ldots, \bar{b}_{k-1}$ finite tuples in $\C$. We say that $(\M, A, \bar{b}_1, \ldots, \bar{b}_{k-1})$ is in \emph{generic position} if there exist regular cardinals $ \kappa_1 < \kappa_2 < \ldots < \kappa_{k-1}$, all greater than $|T|$, and models $\M_0 \preceq \M_1 \preceq \ldots \preceq \M_{k-1} = \M$ such that $A \subseteq \M_0$, $\beth_2(|\M_i|)^+ \leq \kappa_{i+1}$ for $i=0, \ldots, k-2$ and 
$$ \bar b_i \ind_{\M_i}^{u,\kappa_i}  \bar b_{<i} \M$$
for all $1\leq i \leq k-1$. 
\end{defn}

\begin{remark}
	We think about generic position as expressing that the tuples $\bar{b}_1, \ldots, \bar{b}_{k-1}$ form a sufficiently independent set over $\mathcal{M}$ (imitating realizations of pairwise commuting global invariant types, without being able to choose such types in general; in particular, while $\ind^{u, \kappa}$ is not symmetric, the assumption provides a controlled substitute for symmetry among the $\bar{b}_i$'s). See \cite[Remark 4.12 and Problem 4.13]{chernikov2021n} for the open question of whether generic position is necessary in these connected-component results; \cite[Remark 4.14]{chernikov2021n} explains how to arrange it in the setup used there.
\end{remark}

\begin{fact}\cite[Corollary 4.10]{chernikov2021n}\label{fac: G00 for n-dep}
For any $k \geq 1$, let $T$ be a $k$-dependent theory, $A \subseteq \mathbb{M} \models T$ a small set of parameters and $G = G(\mathbb{M})$ a type-definable group over $A$. Let $\M \supseteq A$ be a small model and $\bar{b}_1, \ldots, \bar{b}_{k-1}$ finite tuples in $\mathbb{M}$ so that $(\M, A, \bar{b}_1, \ldots, \bar{b}_{k-1})$ is in generic position. Then there is some $C \subseteq \M$ with $|C| \leq \beth_2(|T| + |A|)$ such that
	\begin{gather*}
		G^{00}_{\M \cup \bar{b}_1 \cup \ldots \cup \bar{b}_{k-1}} = \left( \bigcap_{i=1}^{k-1} G^{00}_{\M \cup \bar{b}_1 \cup \ldots \cup \bar{b}_{i-1} \cup \bar{b}_{i+1} \cup \ldots \cup \bar{b}_{k-1}} \right) \cap G^{00}_{C \cup \bar{b}_1 \cup \ldots \cup \bar{b}_{k-1}}.
	\end{gather*}
\end{fact}
\begin{remark}
\begin{enumerate}
	\item For $k=1$ the assumptions of generic position in Fact \ref{fac: G00 for n-dep} are trivially satisfied by any sufficiently large model $\M = \M_0$, and the conclusion gives $G^{00}_{\M} = G^{00}_{C}$ for some small subset $C$ of $\M$ (since the first intersection on the right-hand side is over the empty set). This easily implies absoluteness of $G^{00}$ in Fact \ref{fac: NIP G00 absolute}. 
	\item For $k=2$, the assumption $\bar{b}_1 \ind ^{u, \kappa_1}_{\M_1} \M_{1}$ is clearly satisfied by any $\kappa_1$-saturated model $\M_1 \supseteq A$ (taking $A \subseteq \M_0 \preceq \M_1$ arbitrary), and the conclusion gives $G^{00}_{\M_1 \bar{b}_1} = G^{00}_{\M_1} \cap G^{00}_{\bar{b}_1 C}$ --- hence Fact \ref{fac: Sh G00 2-dep} follows from Fact \ref{fac: G00 for n-dep}.
\end{enumerate}
\end{remark}

In this section we prove an analog of Fact \ref{fac: G00 for n-dep} (and in particular of Fact \ref{fac: Sh G00 2-dep}) for $G^{\infty}$ in arbitrary $k$-dependent groups:
\begin{theorem}\label{thm: Ginfty k-dep main}
	For any $k\geq 1$, let $T$ be a $k$-dependent theory, $A \subseteq \mathbb{M} \models T$ a small set of parameters and $G = G(\mathbb{M})$ a type-definable group over $A$. Let $\M \supseteq A$ be a small model and $\bar{b}_1, \ldots, \bar{b}_{k-1}$ finite tuples in $\mathbb{M}$ so that $(\M, A, \bar{b}_1, \ldots, \bar{b}_{k-1})$ is in generic position. Then there is some $C \subseteq \M$ with $|C| \leq \beth_2(|T| + |A|)$ such that
	\begin{gather*}
		G^{\infty}_{\M \cup \bar{b}_1 \cup \ldots \cup \bar{b}_{k-1}} = \left( \bigcap_{i=1}^{k-1} G^{\infty}_{\M \cup \bar{b}_1 \cup \ldots \cup \bar{b}_{i-1} \cup \bar{b}_{i+1} \cup \ldots \cup \bar{b}_{k-1}} \right) \cap G^{\infty}_{C \cup \bar{b}_1 \cup \ldots \cup \bar{b}_{k-1}}.
		\end{gather*}
\end{theorem}

\begin{remark}\label{rem: amenable follows}
\begin{enumerate}
	\item We note that in any theory, if a definable group $G$ is definably amenable (so for example abelian), then for any set $S$ we have $G^{00}_S = G^{\infty}_S$ (this is a special case of \cite[Theorem 0.5]{krupinski2019amenability}). Hence the definably amenable case of Theorem \ref{thm: Ginfty k-dep main} also follows from Fact \ref{fac: G00 for n-dep}.
	
		\item Similarly, if the theory $T$ is simple, then $G^{00}_S = G^{\infty}_S$ (see, for example, \cite[Chapter 4]{wagner2000simple}); hence, in the special case of a simple theory, Theorem \ref{thm: Ginfty k-dep main} also follows from Fact \ref{fac: G00 for n-dep}. 
		\end{enumerate}
\end{remark}
%


\subsection{Lascar strong types and thick formulas}\label{sec: Lstp and thick}
We recall some material about Lascar strong types; see \cite{Ziegler1998,CasLasPilZiegler2001} and \cite[Chapter 9]{casanovas2011simple} for references.
	Recall that two (possibly infinite) tuples $a,b$ have the same \emph{Lascar strong type} over a small set of parameters $S \subseteq \mathbb{M}$ if $a E b$ holds for every $S$-invariant bounded equivalence relation $E$ on tuples of the appropriate sorts of $\mathbb{M}$. In this case we write $a \equiv^{L}_S b$.

\begin{defn}\cite[Definition 9.5]{casanovas2011simple} 
\begin{enumerate}
		\item A symmetric formula $\theta(x, y) \in  \mathcal{L}(\mathbb{M})$, with $x,y$ tuples of variables of the same sorts (and length), is \emph{thick} if there is no sequence $(a_i)_{i \in \omega}$ from $\mathbb{M}^x$ (with some of the $a_i$ possibly repeated) with $\models \neg \theta (a_i, a_j)$ for all $i < j < \omega$. Note that, in particular, every thick formula is reflexive.
	\item  For $S \subseteq \mathbb{M}$ and $x,y$ fixed tuples of variables  of the same sorts, let $\Theta_S(x,y)$ be the set of all  thick formulas with parameters in $S$ in (finite subtuples of) variables $x,y$.
\end{enumerate}
\end{defn}
\begin{fact}\label{fac: thick vs indisc}
\cite[Lemma 9.7]{casanovas2011simple} For any set $S$ and tuples $a,b$ of the same sorts, $\models \Theta_S(a,b)$ if and only if there is a sequence $(c_i)_{i \in \omega}$  indiscernible over $S$ with $c_0 = a, c_1 = b$.
\end{fact}

\begin{defn}
	We will say that a partitioned formula $\theta(x,y; z) \in \mathcal{L}$ is \emph{uniformly thick} if for every $c \in \mathbb{M}^{z}$, the formula $\theta(x,y;c)$ is thick.
\end{defn}
\begin{remark}\label{rem: uniformly thick iff thick}
	\begin{enumerate}
	\item $\theta(x,y; z)$ is uniformly thick if and only if there is some $n = n(\theta) \in \omega$ so that: for any $c \in \mathbb{M}^{z}$, $\theta(x,y;c)$ is symmetric and there are no $(a_0, \ldots, a_{n-1})$ with $\models \neg \theta(a_i,a_j;c)$ for all $i<j<n$.
		\item Every thick formula $\theta(x,y) \in \mathcal{L}(\mathbb{M})$ is equivalent to an instance (over the same parameters) of some uniformly thick formula $\theta'(x,y;z) \in \mathcal{L}$.
		\item Let $\Theta^{u}(x,y;z)$ be the set of all uniformly thick formulas (without parameters) in (finite subtuples of) the variables $x,y,z$.  Then if $S$ is a set and $z$ corresponds to an enumeration of $S$, we have  $\Theta_S(x,y) \equiv \Theta^u(x,y;S)$.
	\end{enumerate}
\end{remark}
\begin{proof}
	(1) By compactness.
	
	(2) Assume $\theta(x,y) \in \mathcal{L}(\mathbb{M})$ is thick, say $\theta(x,y) = \varphi(x,y;c)$ for some $\varphi(x,y;z) \in \mathcal{L}$ and $c \in \mathbb{M}^z$, and choose $n \geq 2$ such that there are no $(a_0, \ldots, a_{n-1})$ with $\models \neg \theta(a_i,a_j)$ for all $i<j<n$. Put
	\[
	\psi(x,y;z):=\varphi(x,y;z)\land\varphi(y,x;z)
	\]
	and let
	\[
	\operatorname{Bad}_n(z):=
	\exists x_0,\ldots,x_{n-1}
	\bigwedge_{i<j<n}\neg\psi(x_i,x_j;z).
	\]
	Then
	\[
	\theta'(x,y;z):=\operatorname{Bad}_n(z)\lor\psi(x,y;z)
	\]
	is uniformly thick. Indeed, if $\operatorname{Bad}_n(c')$ holds, then
	$\theta'(x,y;c')$ is the tautological relation; if it fails, then
	$\theta'(x,y;c') \equiv \psi(x,y;c')$ is symmetric and has no bad clique of size $n$.
	Finally, $\operatorname{Bad}_n(c)$ fails and, as $\theta$ is symmetric,
	$\psi(x,y;c)$ is equivalent to $\theta(x,y)$. Thus
	$\theta(x,y) \equiv \theta'(x,y;c)$.
	
	(3) Follows from (2).
\end{proof}
\begin{remark}\label{rem: thick conj}
	Note that by Ramsey's theorem, both sets of formulas $\Theta^{u}(x,y;z)$ and $\Theta_S(x,y)$ (for an arbitrary $S$) are closed under conjunctions.
\end{remark}

\begin{defn}\label{def: Lascar distance} For $n \in \omega$, we let the \emph{Lascar distance $\dist_S(a,b)$ over $S$ between $a$ and $b$} be the smallest $n \in \omega$ such that there exist $a_0, \ldots, a_n$ with $a_0 = a$, $a_n = b$ and $\models \Theta_S(a_i,a_{i+1})$ for each $i < n$; or $\infty$ if no such $n$ exists.
\end{defn}
\begin{fact}\label{fac: thick vs lascar distance}\cite[Lemma 7]{Ziegler1998}
	\begin{enumerate}
		\item If $\dist_S(a,b) = 1$, then there is a small model $M \supseteq S$ such that $a \equiv_M b$.
\item If $a \equiv_M b$ then $\dist_M(a,b) \leq 2$.
	\end{enumerate}
\end{fact}

\begin{fact}\cite[Corollary 8]{Ziegler1998}
	The relation $\equiv^L_S$ is the finest bounded $S$-invariant equivalence relation, and $a \equiv^L_S b$ if and only if $\dist_{S}(a,b) < \infty$.
\end{fact}

\begin{remark}\label{rem: las dist type def}
	Let $\mathcal I(x,y,z)$ be the set of triples $(x',y',z')$ such that $x'\subseteq x$ and $y'\subseteq y$ are corresponding finite subtuples of the same sorts, and $z'\subseteq z$ is finite. For any fixed (finite or small infinite) tuples of variables $x,y,z$ and $n \in \omega$, there is a partial $\mathcal{L}(\emptyset)$-type $\pi_{n}(x,y;z)$ so that for any tuples $a,b,c$ in $\mathbb{M}$ of appropriate length we have
	$\models \pi_n(a,b;c) \iff \dist_{c}(a,b) \leq n$. Moreover,
	\[
	\pi_n(x,y;z)\equiv
	\bigwedge_{(x',y',z')\in\mathcal I(x,y,z)}\pi_n(x',y';z').
	\]
\end{remark}
\begin{proof}
	By compactness, Remark \ref{rem: uniformly thick iff thick} and Remark \ref{rem: thick conj}, we can take $\pi_n(x,y;z)$ to be
	\begin{gather*}
		\bigwedge_{(x',y',z')\in\mathcal I(x,y,z)}
		\ \bigwedge_{\theta(x',y';z') \in \Theta^u(x',y';z')}
		\exists u_0, \ldots, u_{n}
		\left(u_0=x'\land u_n=y'\land
		\bigwedge_{i<n}\theta(u_i,u_{i+1};z')\right),
	\end{gather*}
	%
The ``moreover'' part follows from the definition of $\pi_n$ and compactness.
\end{proof}

\begin{lemma}\label{lem: lasc dist 1 on an array}
	Assume that $(a_i : i \in \mathbb{Q})$ is an $S$-indiscernible sequence (of finite or infinite tuples) and $I,J \subseteq \mathbb{Q}$ are finite with $I \cap J = \emptyset$, say $I = \{i_1, \ldots, i_n\}$ for some $n \in \omega$ and $i_1 < \ldots < i_n$. For $t \in [n]$, let $i'_t \in \mathbb{Q}$  be arbitrary so that $i_t < i'_t$ and 
	$$\qftp_{<}(i'_t / J \cup \{ i_k : k \in [n] \setminus \{t\} \}) = \qftp_{<}(i_t / J \cup \{ i_k : k \in [n] \setminus \{t\} \}).$$
	Then 
	$$\dist_{S \cup (a_{i})_{i \in J}} \left( (a_{i_1}, \ldots, a_{i_n}), (a_{i'_1}, \ldots, a_{i'_n})\right) \leq 1.$$
\end{lemma}
\begin{proof}
	For each $t\in[n]$, density of $\mathbb Q$ allows us to extend
	$i_{t,0}:=i_t<i_{t,1}:=i'_t$ to a strictly increasing sequence
	$(i_{t,\alpha})_{\alpha<\omega}, i_{t,\alpha} \in \mathbb{Q}$ contained in the same cut over
	$J\cup\{i_k:k\ne t\}$. 
	By $S$-indiscernibility of $(a_i:i\in\mathbb Q)$, the sequence of tuples 
	\[
	\bigl((a_{i_{1,\alpha}},\ldots,a_{i_{n,\alpha}}):\alpha\in\omega\bigr)
	\]
	is indiscernible over $S\cup(a_i)_{i\in J}$. Its first two elements are $(a_{i_1},\ldots,a_{i_n})$ and $(a_{i'_1},\ldots,a_{i'_n})$, so the conclusion follows from Fact \ref{fac: thick vs indisc}.
\end{proof}

\begin{remark}\label{rem : dcl of lascar distance}
	If  for some tuples $a,b$ we have $\dist_S(a,b) \leq n$ and $f$ is an $S$-definable function, then $\dist_S(f(a),f(b)) \leq n$.
\end{remark}
\begin{proof}
	By induction on $n$, using that if $(c_i: i \in \omega)$ is an $S$-indiscernible sequence with $c_0 = a, c_1 = b$, then $(f(c_i) : i \in \omega)$ is also an $S$-indiscernible sequence starting with $f(a), f(b)$.
\end{proof}

\subsection{$G^{\infty}$ and Lascar strong types}\label{sec: Ginfty and Lstp}
Let $G$ be an $\emptyset$-type-definable group.
For $X \subseteq G$ and $m \in \omega$, we let $X^m := \left\{g_1 \cdot \ldots \cdot g_m : g_i \in X \right\}$, with the convention $X^0:=\{1_G\}$, and let $\langle X \rangle$ be the subgroup of $G$ generated by $X$.
\begin{defn}
Let $S \subseteq \mathbb{M}$.
\begin{enumerate}
	\item We let $X_S = X_{S,1} := \{ a^{-1}b : a,b \in G(\mathbb{M}), \models \Theta_S(a,b) \}$.
 \item For $n \in \omega$, we let $X_{S,n} := \{ a^{-1}b : a,b \in G(\mathbb{M}), \dist_{S}(a,b) \leq n \}$. We also let $X_{S,<\omega} := \bigcup_{n \in \omega} X_{S,n} = \{ a^{-1}b : a,b \in G(\mathbb{M}), a \equiv^{L}_{S} b \}$.
 \item Given a set of formulas $\Phi(x, y;z) \subseteq \Theta^u(x,y;z)$, if $S$ is a set such that $z$ corresponds to an enumeration of $S$, we define 
$$X_S^{\Phi}=\{ a^{-1}b : a,b\in G(\mathbb M),\
\models\theta(a,b;S)\text{ for every }\theta(x,y;z)\in\Phi\}.$$
We write $X_S^{\theta}$ instead of $X_S^{\{\theta\}}$.
By Remark \ref{rem: uniformly thick iff thick}(3), Remark \ref{rem: thick conj} and compactness, $X_S = \bigcap_{\theta(x,y;z) \in \Theta^u(x,y;z)}X^{\theta}_{S}$.
\end{enumerate}
\end{defn}

We summarize some basic properties of the sets $X_{S,n}$ that we will often use, sometimes without explicitly mentioning them. 
\begin{remark}\label{rem: basic props of Xs}
	\begin{enumerate}
	\item If $1_{G}$ is the identity of $G$, then $1_{G} \in X_{S,n}$ for all $S$ and $n \in \omega$.
		\item For every $\Phi(x,y;z)\subseteq\Theta^u(x,y;z)$, if $a \in X^{\Phi}_{S}$, then also $a^{-1} \in X^{\Phi}_{S}$. Similarly, for all $n$, if $a \in X_{S,n}$ then also $a^{-1} \in X_{S,n}$.
		\item $X_{S,n}^m \subseteq X_{S',n'}^{m'}$ for all $S' \subseteq S$ and $n \leq n', m \leq m'$.
		\item $X_{S,n} \subseteq X_{S}^{n}$.
		\item If $d \in G$ and $n \in \omega$ are arbitrary, then $d (X_{S,n}) d^{-1} \subseteq (X_{S, 2n})^2$.
	\item For any fixed $m,n \in \omega$ and (finite or small infinite) tuple of variables $z$, there is a partial $\mathcal{L}(\emptyset)$-type $\rho_{m,n}(x;z)$ (with $x$ of the same type as the elements of $G$) so that for any $g \in G(\mathbb{M})$ and any tuple $c\in\mathbb{M}$ of appropriate length we have 
	$\models \rho_{m,n}(g;c) \iff g \in (X_{c,n})^m$. Moreover, 		$\rho_{m,n}(x;z) \equiv \bigwedge_{z' \subseteq z, |z'| < \aleph_0} \rho_{m,n}(x;z')$. 
	\item In particular, for any $S$ and $m \in \omega$, 
		$$X_{S}^m = \bigcap_{S' \subseteq S \text{ finite}} X_{S'}^m = \bigcap_{\substack{S' \subseteq S\text{ finite}\\ \theta(x,y;z) \in \Theta^{u}(x,y;z)}}(X_{S'}^{\theta})^m.$$
\end{enumerate}
\end{remark}
\begin{proof}
(1) As $\dist_S(a,a) = 0$ for all $S$ and $a \in G$, and $1_{G} = a^{-1}a$.

	(2) All formulas in $\Phi(x,y;S)$ are thick, hence symmetric, so if $a = b^{-1} c$ and $\models \Phi(b,c;S)$, then $a^{-1} = c^{-1} b$ and $\models \Phi(c,b;S)$, and the second part follows as $\dist_{S}(a,b) \leq n \iff \dist_{S}(b,a) \leq n$.
	
	(3) Clear from the definitions, using (1).
	
	(4) Assume $c \in X_{S,n}$, that is $c = a^{-1}  b$ and there exist some $c_0, \ldots, c_n \in G(\mathbb{M})$ such that $a = c_0, b = c_n$ and $\models \Theta_S(c_i, c_{i+1})$ for all $i < n$. But then we have 
	$$a^{-1} b = \prod_{i<n}(c_i^{-1} c_{i+1}) \in X_{S,1}^n.$$
	(5) Let $g=a^{-1}b\in X_{S,n}$ and let $d\in G$ be arbitrary. By Fact \ref{fac: thick vs lascar distance}(1), choose $c_0=a,\ldots,c_n=b$ and models $\M_i\supseteq S$ such that $c_i\equiv_{\M_i}c_{i+1}$ for every $i<n$. Put $h_0 := d^{-1}$ and, recursively, choose $h_{i+1}$ so that
\[
(c_i,h_i)\equiv_{\M_i}(c_{i+1},h_{i+1}).
\]
Fact \ref{fac: thick vs lascar distance}(2) and the triangle inequality give
\[
\dist_S\big((a,d^{-1}),(b,h_n)\big)\leq 2n.
\]
Hence both $(a d^{-1})^{-1}(b h_n)$ and $h_n^{-1}d^{-1}$ belong to $X_{S,2n}$ (Remark \ref{rem : dcl of lascar distance}), and therefore
\[
dgd^{-1}=(a d^{-1})^{-1}(b h_n)\,h_n^{-1}d^{-1}\in (X_{S,2n})^2.
\]

(6) By Remark \ref{rem: las dist type def} and compactness we can define
$\rho_{m,n}(x;z)$ by the following partial-type expression (that is, by
all of its finite fragments):
\begin{gather*}
	\exists x_1, \ldots, x_m \exists y_1, \ldots, y_{m}
	\left(
	\bigwedge_{i=1}^{m}
	\big(\pi_G(x_i)\land\pi_G(y_i)\land\pi_n(x_i,y_i;z)\big)
	\land x=x_1^{-1}y_1\cdots x_m^{-1}y_m
	\right).
\end{gather*}
The ``moreover'' part follows by compactness from this definition and the ``moreover'' part of Remark \ref{rem: las dist type def}.

(7) By the ``moreover'' part of (6), definition of $\rho_{m,1}$ and Remark \ref{rem: thick conj}.
\end{proof}
The main point for us is the following description of $G^{\infty}$:
\begin{fact}\label{fac: generates G infty}
 \cite[Lemma 2.2]{Gismatullin2011} 
 For any set of parameters $S$, 
 $$G_S^{\infty} = \langle X_{S, < \omega} \rangle = \langle X_{S} \rangle = \bigcup_{n \in \omega} X_S^{n} \textrm{ and } \left[ G:G^{\infty}_{S} \right] \leq 2^{|\mathcal{L}(S)|} $$
 (where the second equality is by Remark \ref{rem: basic props of Xs}(4)).
\end{fact}

\begin{defn}
For $S \subseteq \mathbb{M}$, we let
\[
Y_S:=\bigcup_{g\in G}gX_Sg^{-1},
\qquad
Y_S^\Phi:=\bigcup_{g\in G}gX_S^\Phi g^{-1},
\]
and write $Y_S^{\theta}$ instead of $Y_S^{\{\theta\}}$.
\end{defn}

\noindent The advantage of the sets $Y_S$ is that they are invariant under conjugation, while they still generate $G_S^\infty$ and admit the same kind of finite approximations as the sets $X_S$.
\begin{lemma}\label{lem: normal hull X}
For every parameter set $S$, every uniformly thick formula $\theta(x,y;z) \in \mathcal{L}$, and every $r<\omega$ we have:
\begin{enumerate}
	\item $Y_S$ and $Y_S^\theta$ are symmetric and invariant under conjugation by elements of $G$, and
	$
	X_S\subseteq Y_S\subseteq X_S^4
$;
	\item $\langle Y_S\rangle=G_S^\infty$;
	\item
	$
	Y_S^r=
	\bigcap_{\substack{S_0\subseteq S\text{ finite}\\
	\theta(x,y;z)\in\Theta^u(x,y;z)}}
	\left(Y_{S_0}^{\theta}\right)^r.
	$
\end{enumerate}
\end{lemma}
\begin{proof}
Symmetry is immediate by Remark \ref{rem: basic props of Xs}(2), and conjugation is  immediate by definition. The first inclusion in (1) follows by taking the conjugating element to be $1_G$. For the second one, Remark \ref{rem: basic props of Xs}(5) and (4) gives
\[
gX_Sg^{-1}\subseteq (X_{S,2})^2\subseteq X_S^4
\]
for every $g\in G$. This also proves (2), using Fact \ref{fac: generates G infty}.

The left-to-right inclusion in (3) is clear. Conversely, suppose that $d$ belongs to every set on the right.  Introduce tuples $(g_i,a_i,b_i)_{i<r}$ and consider the partial type saying that all these tuples belong to $G$, that $\Theta_S(a_i,b_i)$ holds for every $i<r$, and that $d=\prod_{i<r}g_i(a_i^{-1}b_i)g_i^{-1}$. 
Every finite part of this type only mentions a finite subset $S_0\subseteq S$ and finitely many uniformly thick formulas. Their conjunction is uniformly thick by Remark \ref{rem: thick conj}, so the finite part is realized by the hypothesis on $d$. Compactness therefore realizes the whole type, showing that $d\in Y_S^r$.
\end{proof}

\begin{defn}\label{def: eta for Y_S}
	Fix a partial type $\pi_G(x)$ defining $G$.  If
$\Delta\subseteq\pi_G$ is finite, let $\Delta(t)$ denote its
conjunction, and, for a uniformly thick formula $\theta(x,y;z)$ and
$r<\omega$, let $\eta_{r,\Delta}^{\theta}(x;z)$ be the formula
\[
\exists(g_i,a_i,b_i)_{i<r}\left(
\bigwedge_{i<r}\big(
\Delta(g_i)\land\Delta(a_i)\land\Delta(b_i)
\land\theta(a_i,b_i;z)\big)
\land x=\prod_{i<r}g_i(a_i^{-1}b_i)g_i^{-1}
\right).
\]
By compactness, for any $d\in G$ and $c \in \mathbb{M}^z$ we have
\begin{equation}\label{eq: finite normal approximation}
d\in(Y_c^\theta)^r
\quad\Longleftrightarrow\quad
\models\eta_{r,\Delta}^{\theta}(d;c)
\text{ for every finite }\Delta\subseteq\pi_G.
\end{equation}
When $G$ is definable, one may take $\Delta$ to define $G$.
\end{defn}

\begin{remark}\label{rem: choosing Delta_1 for Delta_0}
	Fix a uniformly thick formula $\theta(x,y;z)$.  For every finite
$\Delta_0\subseteq\pi_G$ and $r,s<\omega$, there is a finite
$\Delta_1\subseteq\pi_G$ such that, for all $c \in \mathbb{M}^z$ and
$a,b,u\in G$, we have:
\begin{equation}\label{eq: finite fragment multiplication}
\begin{gathered}
a=bu^{-1},\qquad b\in(Y_c^\theta)^s,\qquad
\models  \eta_{r,\Delta_1}^{\theta}(u;c)\\
\Longrightarrow\quad
\models \eta_{r+s,\Delta_0}^{\theta}(a;c).
\end{gathered}
\end{equation}
\end{remark}
\begin{proof}
	Indeed, otherwise compactness would give $a,b,u,c$ satisfying the
first two conditions, all the formulas
$\eta_{r,\Delta}^{\theta}(u;c)$, and
$\neg\eta_{r+s,\Delta_0}^{\theta}(a;c)$.  By
\eqref{eq: finite normal approximation}, this says that
$u\in(Y_c^\theta)^r$.  Since $Y_c^\theta$ is symmetric,
$a=bu^{-1}\in(Y_c^\theta)^{r+s}$, which implies
$\eta_{r+s,\Delta_0}^{\theta}(a;c)$ by \eqref{eq: finite normal approximation} again, and gives the required
contradiction.
\end{proof}

%
%
%

\begin{lemma}\label{lem: countable intersec}
For any small set of parameters $C \subset \mathbb{M}$ we have
$$G^\infty_C = \bigcap \left\{ G^\infty_S : S \subseteq C \textrm{ countable} \right\}.$$	
	\end{lemma}
	
	\begin{proof}
	The inclusion from left to right is obvious. To see the inclusion from right to left, assume that $d \in G \setminus G^\infty_C$, i.e.~$d \notin \langle X_C \rangle$ by Fact \ref{fac: generates G infty}.
By Remark \ref{rem: basic props of Xs}(7), for any $m < \omega$, we can choose a finite set  $S_{m}\subseteq C$, such that $d \not \in (X_{S_{m}})^m$. Let $S := \bigcup_{m< \omega} S_{m}  $. Then, using Remark \ref{rem: basic props of Xs}(3), $d \not \in  \langle X_{S}\rangle = G^{\infty}_{S}$, and $S$ is countable. 
\end{proof}

\subsection{Proof of Theorem \ref{thm: Ginfty k-dep main}}\label{sec: proof of Ginf}
We fix $k \geq 1$ and assume that $T$ is a $k$-dependent theory, $A \subseteq \C \models T$ is a small parameter set and $G = G(\C)$ is a group type-definable over $A$. To ease the notation, let us name the elements of $A$ by constants in our language; in the expanded language, we continue to write $|T|$ for $|T|+|A|$. Let $\M$ be a model and let $\bar b_1,\dots,\bar b_{k-1}$ be finite tuples such that $(\M, \emptyset,\bar b_1,\dots,\bar b_{k-1})$ is in generic position.

Assume the conclusion of Theorem \ref{thm: Ginfty k-dep main} fails. 
	Then, using Lemma \ref{lem: countable intersec}, by transfinite induction on $\alpha < \beth_2(|T|)^+$ we can choose \emph{countable} tuples $S_{\alpha}$ of elements of $\mathcal{M}$ and elements $d_{\alpha} \in G(\mathbb{M})$ so that for every $\alpha < \beth_2(|T|)^+$ we have:
	\begin{gather}
		d_\alpha \in \bigcap_{i\in \set{k-1}} G_{\M \cup \bigcup_{\substack{j\in\set{k-1}\\j\ne i}}\bar b_j}^{\infty},\label{eq: con comp 1}\\
		d_\alpha \in \bigcap_{\beta < \alpha} G^{\infty}_{S_{\beta} \cup \bar b_1 \cup \dots \cup \bar b_{k-1}},\label{eq: con comp 2}\\
		d_\alpha \not\in G^{\infty}_{S_{\alpha} \cup \bar b_1 \cup \dots \cup \bar b_{k-1}}. \label{eq: con comp 3}
	\end{gather}

	
	By \eqref{eq: con comp 1}, \eqref{eq: con comp 2}, Fact \ref{fac: generates G infty} and  Remark \ref{rem: basic props of Xs}(3), for every $\alpha < \beta$ there exists some $m_{\alpha,\beta} \in \omega$ so that 
	\begin{gather*}
		d_{\alpha} \in \bigcap_{i\in \set{k-1}} \left( X_{\M \cup \bigcup_{\substack{j\in\set{k-1}\\j\ne i}}\bar b_j} \right)^{m_{\alpha, \beta}} \textrm{ and}\\
		d_{\beta} \in \left( X_{S_\alpha \cup \bar b_1 \cup \dots \cup \bar b_{k-1}}\right)^{m_{\alpha, \beta}}.
	\end{gather*}
 By the Erd\H{o}s--Rado theorem, passing to a subsequence of $\left( (d_{\alpha}, S_{\alpha}) : \alpha < \beth_2(|T|)^+ \right)$ of length $\kappa_0 := |T|^+$, we may assume that there is some $1 \leq m \in \omega$ so that $m_{\alpha, \beta} = m$ for all $\alpha < \beta < \kappa_0$.
Put
\[
m':=2^{k-1}m,\qquad p:=2^k-1,\qquad q_1:=(m'+2)p,
\qquad q_k:=q_1+(2^{k-1}-1)m.
\]
Next, by \eqref{eq: con comp 3} and Lemma \ref{lem: normal hull X}(2) and (3), for every $ \alpha < \kappa_0$ there is some uniformly thick formula $\theta_{\alpha} (x,y;z) \in \Theta^u(x,y;z)$ and $n_{\alpha}\in\omega$ such that 
\[d_{\alpha} \not\in \left(Y_{(S_{\alpha} \upharpoonright n_{\alpha})  \cup \bar b_{1} \cup  \dots \cup\bar b_{k-1}}^{\theta_{ \alpha}} \right)^{q_k},\]
where $S_{\alpha} \upharpoonright n_{\alpha}$ denotes the initial segment of length $n_{\alpha}$ of the countable tuple $S_{\alpha}$.
Indeed, the lemma first supplies a finite subset of
$S_\alpha\cup\bar b_1\cup\cdots\cup\bar b_{k-1}$.  Enlarge its
$S_\alpha$-part to an initial segment, include all the finite tuples
$\bar b_i$, and add the resulting extra parameter variables as dummy
variables to the uniformly thick formula.  Its relevant instance, and
hence the nonmembership, is unchanged.
By \eqref{eq: finite normal approximation} in Definition \ref{def: eta for Y_S}, we may also choose a finite
$\Delta_{k,\alpha}\subseteq\pi_G$ such that the corresponding formula
$\eta_{q_k,\Delta_{k,\alpha}}^{\theta_\alpha}$ fails at $d_\alpha$.
As $\kappa_0 >  |T|$, by pigeonhole, passing to a subsequence of length $\kappa_0$, we may assume that $\theta_{\alpha} (x,y;z) = \theta(x,y;z)$, $n_{\alpha} = n$, and $\Delta_{k,\alpha}=\Delta_k$ for some fixed uniformly thick formula $\theta(x,y;z) \in  \Theta^u(x,y;z)$, $n \in \omega$, and finite $\Delta_k\subseteq\pi_G$. That is, for all $\alpha < \beta < \kappa_0$ we have:
\begin{gather}
	d_\alpha \in \bigcap_{i\in \set{k-1}} \left( X_{\M \cup \bigcup_{\substack{j\in\set{k-1}\\j\ne i}}\bar b_j} \right)^m, \label{eq: con comp 4}\\
	d_\beta \in \left( X_{S_\alpha \cup \bar b_1 \cup \dots \cup \bar b_{k-1}}\right)^m, \label{eq: con comp 5}\\
	\models\neg\eta_{q_k,\Delta_k}^{\theta}
	(d_\alpha;S_{\alpha}\upharpoonright n,\bar b_1,\ldots,\bar b_{k-1}),
	\quad\text{and hence}\quad
	d_{\alpha} \not\in \left(Y_{S_{\alpha} \upharpoonright n \cup \bar b_{1} \cup  \dots \cup\bar b_{k-1 }}^{\theta} \right)^{q_k}. \label{eq: con comp 6}
\end{gather}

As $(\M, \emptyset,\bar b_1,\dots,\bar b_{k-1})$ is in generic position, there are regular cardinals $\kappa_1 < \kappa_2 < \ldots < \kappa_{k-1}$ and models $\M_0 \preceq \M_1 \preceq \ldots \preceq \M_{k-1} = \M$ such that $\beth_2(|\M_i|)^+ \leq \kappa_{i+1}$ for $i=0, \ldots, k-2$ and 
 \begin{gather}
 	\bar b_i \ind_{\M_i}^{u,\kappa_i}  \bar b_{<i} \M_{k-1} \label{eq: con comp 7}
 \end{gather}
for all $1\leq i \leq k-1$. 

\begin{claim} \label{claim Ginftymain} There exist a finite set $\Delta_1\subseteq\pi_G$, sequences $\left(\bar{b}_{l,\alpha}: \alpha< \omega \right)$ for $l \in \set{k-1}$ (in $\M_{k-1}$) and elements $\left(d_{\alpha_1, \dots, \alpha_{k}} : 
(\alpha_1,\dots,\alpha_{k})\in \omega^k \right)$ in $G$ such that:
\begin{enumerate}
\item  for all  $(\alpha_1,  \dots, \alpha_k) <^{\lex} (\beta_1, \dots, \beta_k) \in \omega^k$ we have 
$$d_{\beta_1, \dots, \beta_{k}}  \in \left(X_{S_{\alpha_k} \cup \bar b_{1,\alpha_1}\cup \dots\cup \bar b_{k-1, \alpha_{k-1}}} \right)^{m'};  $$
\item for all $(\alpha_1, \dots, \alpha_k) \in \omega^k$ we have 
\[
\models\neg\eta_{q_1,\Delta_1}^{\theta}
(d_{\alpha_1,\ldots,\alpha_k};
S_{\alpha_k}\upharpoonright n,\bar b_{1,\alpha_1},\ldots,
\bar b_{k-1,\alpha_{k-1}}),
\]
and hence
\[
d_{\alpha_1, \dots, \alpha_k} \not\in \left(Y_{S_{\alpha_k} \upharpoonright n \cup \bar b_{1, \alpha_1} \cup  \dots \cup\bar b_{k-1, \alpha_{k-1} }}^{\theta} \right)^{q_1}.
\]
\end{enumerate}
\end{claim}

\proof[Proof of Claim \ref{claim Ginftymain}]
\newcommand{\lam}[1]{\lambda_{#1}}
For $l = 1, \ldots, k$, let
\begin{gather}
	\lam{l} := \omega^{k-l + 1} = \prod_{i=l}^{k} \omega, \qquad
	m_l := 2^{k-l}m, \qquad
	q_l = q_1+\sum_{i=2}^{l}m_i \textrm{ (as above)}. \label{eq: con comp -1}
\end{gather}
Here the sum is empty when $l=1$. In particular, $m'=m_1$, and for
$1\leq l<k$ we have $2m_{l+1}=m_l$ and
$q_{l+1}=m_{l+1}+q_l$.

Starting with the finite set $\Delta_k$ chosen above, use
\eqref{eq: finite fragment multiplication} in Remark \ref{rem: choosing Delta_1 for Delta_0} recursively, for
$l=k-1,\ldots,1$, to choose a finite
$\Delta_l\subseteq\pi_G$ such that for any $a,b,e \in G$ and $c \in \mathbb{M}^z$ we have 
\begin{equation}\label{eq: finite fragment reverse induction}
\begin{gathered}
a=be^{-1},\quad b\in(Y_c^\theta)^{m_{l+1}},\quad
\models \eta_{q_l,\Delta_l}^{\theta}(e;c)\\
\Longrightarrow \quad 
\models \eta_{q_{l+1},\Delta_{l+1}}^{\theta}(a;c).
\end{gathered}
\end{equation}

	To show the claim, we choose by reverse induction on $l=k,\ldots,1$ sequences
	\[
	(\bar b_{l,\alpha}:\alpha<\omega)
	\]
	(these are chosen and needed only for $l<k$) and elements
	\[
	(d_{\alpha_l,\ldots,\alpha_k}:
	(\alpha_l,\ldots,\alpha_k)\in\lam{l})
	\]
	in $G$ such that the following hold:
\begin{description}
\item [($\dagger_1^l$)]   
for all  $(\alpha_l,\dots, \alpha_k), (\beta_l, \dots, \beta_k) \in \lam{l}$ we have
$$d_{\beta_l, \dots, \beta_{k}} \in \bigcap_{i\in \set{l-1}} \left( X_{\M_{l-1} \cup S_{\alpha_k}
\cup \bigcup_{\substack{j\in\set{l-1}\\j\ne i}}\bar b_j
\cup \bigcup_{j=l}^{k-1}\bar b_{j,\alpha_j}} \right)^{m_l};$$
	\item [($\dagger_2^l$)]  
	for all  $(\alpha_l,\dots, \alpha_k) <^{\lex} (\beta_l, \dots, \beta_k) \in \lam{l}$ we have
$$d_{\beta_l, \dots, \beta_{k}}  \in \left( X_{S_{\alpha_k} \cup \bar b_{1}\cup  \dots\cup \bar b_{l-1}\cup \bar b_{l, \alpha_l}\cup	 \dots\cup \bar b_{k-1, \alpha_{k-1}}} \right)^{m_l};$$	 
\item [($\dagger_3^l$)]
for all  $(\alpha_l,\dots, \alpha_k) \in \lam{l}$ we have
\[
\models\neg\eta_{q_l,\Delta_l}^{\theta}
(d_{\alpha_l,\ldots,\alpha_k};
S_{\alpha_k}\upharpoonright n,\bar b_1,\ldots,
\bar b_{l-1},\bar b_{l,\alpha_l},\ldots,
\bar b_{k-1,\alpha_{k-1}}).
\]
In particular,
\[
d_{\alpha_l, \dots, \alpha_{k}} \not\in \left(Y_{S_{\alpha_k} \upharpoonright n   \cup \bar b_{1} \cup  \dots \cup \bar b_{l-1}  \cup \bar b_{l, \alpha_{l}}\cup  \dots\cup \bar b_{k-1, \alpha_{k-1} }}^{\theta} \right)^{q_{l}}.
\]
\end{description}

For $l=1$ this completes the proof of Claim \ref{claim Ginftymain} (by ($\dagger_2^1$), ($\dagger_3^1$), and the choice of $\lambda_1,m_1,q_1$ in \eqref{eq: con comp -1}).

Since every entry of $S_{\alpha_k}$ lies in $\M=\M_{k-1}$ for all $\alpha_k < \kappa_0$, the above holds for $l=k$ with $(d_{\alpha_k} : \alpha_k < \omega)$, i.e.~($\dagger_1^k$), ($\dagger_2^k$), and ($\dagger_3^k$) hold by \eqref{eq: con comp 4}, \eqref{eq: con comp 5}, and \eqref{eq: con comp 6}, respectively (and by the choice of $\lambda_k,m_k,q_k$ in \eqref{eq: con comp -1}).

Now suppose we have found the data as above satisfying ($\dagger_1^{l+1}$), ($\dagger_2^{l+1}$), and ($\dagger_3^{l+1}$) for $2 \leq l+1 \leq k$, and we want to choose the corresponding data for $l$. Let $S$ be the concatenation of all the tuples $S_\alpha$ for $\alpha < \omega$ (then $S$ is countable), and put
$$A_l := \M_{l-1} \cup S \cup \bigcup_{j=1}^{l-1}\bar b_j
\cup \left\{\bar b_{i,\beta}: l+1\leq i<k,\ \beta < \omega \right\}.$$
As $A_l \subseteq \M_{k-1}\cup\bigcup_{j=1}^{l-1}\bar b_j$, we have $|A_l| \leq |\M_{l-1}| + \aleph_0 < \kappa_l$. Moreover,
$$ \bar b_l \ind_{\M_l}^{u,\kappa_l}  \bar b_{<l} \M_{k-1}$$
by \eqref{eq: con comp 7}. We can therefore choose, by transfinite induction, a sequence of tuples
$$\left(
\left(
\bar b_{l,\alpha},
(d_{\alpha,\alpha_{l+1},\dots,\alpha_k}:
(\alpha_{l+1},\dots,\alpha_k)\in\lam{l+1})
\right):
\alpha<\kappa_l
\right)$$
with  $\bar{b}_{l, \alpha} \in \M_l$ and $d_{\alpha, \alpha_{l+1}, \dots, \alpha_{k}} \in G(\mathbb{M})$ such that, taking $A_{l,\alpha} := A_l \cup \left\{\bar b_{l,\beta}: \beta < \alpha \right \} $, for each $\alpha < \kappa_l$ we have
\begin{description}
	\item[$(\star)$] 
	$\left(\bar{b}_{l,\alpha}, (d_{\alpha, \alpha_{l+1}, \dots, 
	\alpha_{k}})_{ (\alpha_{l+1}, \dots, \alpha_{k}) \in \lam{l+1}} \right) \models \tp \left(\bar{b}_l,(d_{\alpha_{l+1}, \dots, 
	\alpha_{k}})_{ (\alpha_{l+1}, \dots, \alpha_{k}) \in  \lam{l+1}}/ A_{l,\alpha} \right)$
\end{description}
(we apply automorphisms of $\mathbb{M}$ over $A_{l,\alpha}$ sending $\bar{b}_l$ to $\bar{b}_{l,\alpha}$  to find the $d$'s).

We claim that ($\dagger_2^{l}$) holds in the following strong form: for
any $(\alpha_l,\dots, \alpha_k) <^{\lex} (\beta_l, \dots, \beta_k) \in \kappa_{l} \times \lambda_{l+1}$ we have 
\begin{gather}
	d_{\beta_l, \dots, \beta_{k}}  \in \left( X_{S_{\alpha_k} \cup \bar b_{1}\cup  \dots\cup \bar b_{l-1}\cup \bar b_{l, \alpha_l}\cup	 \dots\cup \bar b_{k-1, \alpha_{k-1}}} \right)^{m_{l+1}}. \label{eq: con comp 8}
\end{gather}

\noindent Indeed, fix any $(\alpha_l,\dots, \alpha_k) <^{\lex} (\beta_l, \dots, \beta_k) \in \kappa_{l} \times \lambda_{l+1}$.

\noindent If $\beta_l = \alpha_l$, then $(\alpha_{l+1},\dots, \alpha_k) <^{\lex} (\beta_{l+1}, \dots, \beta_k)$, hence by $(\dagger_2^{l+1})$ 
we have 
\begin{gather*}
	d_{\beta_{l+1}, \dots, \beta_{k}}  \in \left( X_{S_{\alpha_k} \cup \bar b_{1}\cup  \dots\cup \bar b_{l}\cup \bar b_{l+1, \alpha_{l+1}}\cup	 \dots\cup \bar b_{k-1, \alpha_{k-1}}} \right)^{m_{l+1}},
\end{gather*}
which, using $(\star)$ (for $\alpha_l$) and Remark \ref{rem: basic props of Xs}(6), implies
\begin{gather*}
d_{\beta_l,\beta_{l+1},\dots, \beta_{k}}= d_{\alpha_l, \beta_{l+1}, \dots, \beta_{k}}  \in \left( X_{S_{\alpha_k} \cup \bar b_{1}\cup  \dots\cup \bar b_{l-1}\cup \bar b_{l, \alpha_l}\cup	 \dots\cup \bar b_{k-1, \alpha_{k-1}}} \right)^{m_{l+1}}.	
\end{gather*}

\noindent Otherwise $\beta_l > \alpha_l$.  By $(\dagger_1^{l+1})$ we have in particular
\begin{gather*}
	d_{\beta_{l+1}, \dots, \beta_{k}} \in  \left( X_{\M_{l} \cup S_{\alpha_k}\cup  \bar b_1 \cup \dots \cup \bar{b}_{l-1}\cup \bar b_{l+1, \alpha_{l+1}}\cup	 \dots\cup \bar b_{k-1, \alpha_{k-1}}} \right)^{m_{l+1}}.
\end{gather*}
As  $\M_{l-1} \subseteq \M_l$ and $ \bar b_{l, \alpha_{l}}\in \M_l$, by Remark \ref{rem: basic props of Xs}(3) this implies  
\begin{gather*}
	d_{\beta_{l+1}, \dots, \beta_{k}} \in \left( X_{\M_{l-1} \cup S_{\alpha_k}\cup  \bar b_1 \cup \dots \cup \bar{b}_{l-1}\cup  \bar b_{l, \alpha_{l}} \cup \bar b_{l+1, \alpha_{l+1}}\cup	 \dots\cup \bar b_{k-1, \alpha_{k-1}} } \right)^{m_{l+1}}.
\end{gather*}
\noindent Moreover, $\bar b_{l, \alpha_{l}} \subseteq A_{l, \beta_l}$ as $\beta_l > \alpha_l$, so $(\star)$ and Remark \ref{rem: basic props of Xs}(6) again imply \eqref{eq: con comp 8}.

Transporting $(\dagger_3^{l+1})$ using $(\star)$, we obtain the following intermediate condition: for all $(\alpha_l, \ldots, \alpha_k) \in \kappa_{l} \times \lambda_{l+1}$ we have
\begin{gather}
	\models\neg\eta_{q_{l+1},\Delta_{l+1}}^{\theta}
	(d_{\alpha_l,\ldots,\alpha_k};
	S_{\alpha_k}\upharpoonright n,\bar b_1,\ldots,
	\bar b_{l-1},\bar b_{l,\alpha_l},\ldots,
	\bar b_{k-1,\alpha_{k-1}}). \label{eq: con comp 9}
\end{gather}

Concerning ($\dagger_1^{l}$), we first show the following:
for all $(\alpha_{l+1},\dots,\alpha_k),(\beta_{l+1},\dots,\beta_k)\in\lam{l+1}$ and all $\alpha_l\leq\beta_l<\kappa_l$,
\begin{gather}
	d_{\beta_l, \dots, \beta_{k}} \in \bigcap_{i\in \set{l-1}} \left( X_{\M_{l-1} \cup S_{\alpha_k}
	\cup\bigcup_{\substack{j\in\set{l-1}\\j\ne i}}\bar b_j
	\cup\bar b_{l,\alpha_l}
	\cup\bigcup_{j=l+1}^{k-1}\bar b_{j,\alpha_j}} \right)^{m_{l+1}}.\label{eq: con comp 10}
\end{gather}

\noindent Indeed, fix any $(\alpha_{l+1},\dots, \alpha_k), (\beta_{l+1}, \dots, \beta_k) \in \lam{l+1}$. By ($\dagger_1^{l+1}$) we have

\begin{gather}
	d_{\beta_{l+1}, \dots, \beta_{k}} \in \bigcap_{i\in \set{l}} \left( X_{\M_{l} \cup S_{\alpha_k}
	\cup\bigcup_{\substack{j\in\set{l}\\j\ne i}}\bar b_j
	\cup\bigcup_{j=l+1}^{k-1}\bar b_{j,\alpha_j}} \right)^{m_{l+1}}.\label{eq: con comp 10 a}
\end{gather}

\noindent  If $\alpha_l = \beta_l$, then by \eqref{eq: con comp 10 a} we have in particular
\begin{gather*}
	d_{\beta_{l+1}, \dots, \beta_{k}} \in \bigcap_{i\in \set{l-1}} \left( X_{\M_{l-1} \cup S_{\alpha_k}
	\cup\bigcup_{\substack{j\in\set{l-1}\\j\ne i}}\bar b_j
	\cup\bar b_l
	\cup\bigcup_{j=l+1}^{k-1}\bar b_{j,\alpha_j}} \right)^{m_{l+1}},
\end{gather*}
\noindent which, using $(\star)$ (and $d_{\beta_l,\beta_{l+1},\dots, \beta_{k}}= d_{\alpha_l, \beta_{l+1}, \dots, \beta_{k}}$), implies \eqref{eq: con comp 10}.

\noindent If $\alpha_l < \beta_l$, then from \eqref{eq: con comp 10 a} we have (using $\bar{b}_{l,\alpha_l} \in \M_{l}$) in particular
\begin{gather*}
	d_{\beta_{l+1}, \dots, \beta_{k}} \in \bigcap_{i\in \set{l-1}} \left( X_{\M_{l-1} \cup S_{\alpha_k}
	\cup\bigcup_{\substack{j\in\set{l-1}\\j\ne i}}\bar b_j
	\cup\bar b_{l,\alpha_l}
	\cup\bigcup_{j=l+1}^{k-1}\bar b_{j,\alpha_j}} \right)^{m_{l+1}},
\end{gather*}
which, using $(\star)$ (and $\alpha_l < \beta_l$), again implies \eqref{eq: con comp 10}.

\medskip

However, to achieve ($\dagger_1^{l}$) in full (i.e.~also when $\beta_l < \alpha_l < \omega$), we must modify the sequence and increase the power bound from $m_{l+1}$ to $m_l=2m_{l+1}$.
Consider the sequence of \emph{countable} tuples
\begin{gather*}
	\left(
	\left(
	\bar b_{l,\alpha_l},
	(d_{\alpha_l,\ldots,\alpha_k}:
	(\alpha_{l+1},\ldots,\alpha_k)\in\lam{l+1})
	\right):
	\alpha_l<\kappa_l
	\right).
\end{gather*}

\noindent For $\beta_l < \kappa_l$, let $\M_{l, \beta_l} \prec \mathbb{M}$ be an arbitrary model containing $A_l \cup \bar{b}_{l, \beta_l}$ with $\left \lvert \M_{l, \beta_l} \right \rvert = \left \lvert \M_{l-1} \right \rvert$. Let $\left( p^{\beta_l}_{\gamma}(x) : \gamma < 2^{|\M_{l-1}|}  \right)$ be an arbitrary enumeration (possibly with repetitions if there are fewer types) of the set $S_{x}\left(\M_{l, \beta_l} \right)$ of complete types over $\M_{l, \beta_l}$ in the tuple of variables $x$ corresponding to the sort of $G$. 

\noindent For $\alpha_l < \beta_l < \kappa_l$ and $ (\alpha_{l+1}, \ldots, \alpha_k) \in \lam{l+1}$, let $t_{\alpha_l, \beta_l}^{(\alpha_{l+1}, \ldots, \alpha_k)}$ be the smallest $\gamma < 2^{|\M_{l-1}|} $ so that $\tp \left( d_{\alpha_l, \ldots, \alpha_k} / \M_{l, \beta_l} \right) = p^{\beta_l}_{\gamma}$, and consider the tuple
\begin{gather*}
	\bar{t}_{\alpha_l, \beta_l} := \left( t_{\alpha_l, \beta_l}^{(\alpha_{l+1}, \ldots, \alpha_k)}   : (\alpha_{l+1}, \ldots, \alpha_k) \in \lam{l+1} \right).
\end{gather*}
There are at most $ \left(2^{|\M_{l-1}|} \right)^{\aleph_0} \leq 2^{|\M_{l-1}|}$ possible choices for $\bar{t}_{\alpha_l, \beta_l}$. As $\kappa_l \geq \beth_2 \left( |\M_{l-1}| \right)^+$
by assumption, applying Erd\H{o}s-Rado and passing to a countable subsequence we have that \eqref{eq: con comp 8} and \eqref{eq: con comp 9} still hold (replacing $\kappa_l$ by $\omega$) and for any  fixed $(\alpha_{l+1}, \ldots, \alpha_k) \in \lambda_{l+1}$  we have additionally:  for all $\alpha_l < \alpha'_l < \beta_l < \omega$,
\begin{gather}
	d_{\alpha_l, \alpha_{l+1}, \ldots, \alpha_k} \equiv_{\M_{l,\beta_l}} d_{\alpha'_l, \alpha_{l+1}, \ldots, \alpha_k},\label{eq: con comp 11}
\end{gather}
	hence, using Fact \ref{fac: thick vs lascar distance}(2) and Remark \ref{rem: basic props of Xs}(4), in particular, if
\[
B_{l,\beta}:=\M_{l-1}\cup S\cup\bigcup_{j=1}^{l-1}\bar b_j
\cup\bar b_{l,\beta}
\cup\{\bar b_{i,\gamma}:l+1\leq i<k,\ \gamma<\omega\},
\]
then
\begin{gather}
	d^{-1}_{\alpha_l, \alpha_{l+1}, \ldots, \alpha_k} \cdot d_{\alpha'_l, \alpha_{l+1}, \ldots, \alpha_k} \in X_{B_{l,\beta_l},2}
	\subseteq \left(X_{B_{l,\beta_l}}\right)^2. \label{eq: con comp 12}
\end{gather}

Now for $(\alpha_l, \ldots, \alpha_k) \in \lam{l}$, we define
\begin{gather*}
	e_{\alpha_l, \alpha_{l+1}, \ldots, \alpha_k} := d^{-1}_{2 \alpha_l, \alpha_{l+1}, \ldots, \alpha_k} \cdot d_{2 \alpha_l + 1, \alpha_{l+1}, \ldots, \alpha_k} \in G(\mathbb{M}),\\
	\bar{c}_{l, \alpha_l} := \bar{b}_{l, 2 \alpha_l} \in \M_{l},
\end{gather*}
and claim that the sequences $(\bar{c}_{l,\alpha_l} : \alpha_l < \omega), (\bar{b}_{l+1, \alpha_{l+1}} : \alpha_{l+1} < \omega), \ldots, (\bar{b}_{k-1, \alpha_{k-1}} : \alpha_{k-1} < \omega)$ and elements $\left(e_{\alpha_l, \ldots, \alpha_k} : 
(\alpha_l,\dots,\alpha_{k})\in \lam{l} \right)$ satisfy the requirements ($\dagger_1^{l}$), ($\dagger_2^{l}$), and ($\dagger_3^{l}$).

Fix any $(\alpha_l,\dots, \alpha_k),(\beta_l, \dots, \beta_k) \in \lam{l}$. If $\alpha_l \leq \beta_l$, then $2\alpha_l\leq2\beta_l$ and $2\alpha_l\leq2\beta_l+1$, so by \eqref{eq: con comp 10} we have
\begin{gather*}
	e_{\beta_l, \beta_{l+1}, \ldots, \beta_k} = d^{-1}_{2 \beta_l, \beta_{l+1}, \ldots, \beta_k} \cdot d_{2 \beta_l + 1, \beta_{l+1}, \ldots, \beta_k} \in \\
	\bigcap_{i\in \set{l-1}} \left( X_{\M_{l-1} \cup S_{\alpha_k}
	\cup\bigcup_{\substack{j\in\set{l-1}\\j\ne i}}\bar b_j
	\cup\bar b_{l,2\alpha_l}
	\cup\bigcup_{j=l+1}^{k-1}\bar b_{j,\alpha_j}} \right)^{2m_{l+1}}.
\end{gather*}

\noindent If $\beta_l < \alpha_l$, then $2 \beta_l < 2 \beta_l + 1 < 2 \alpha_l$, so by \eqref{eq: con comp 12} we have
\begin{gather*}
	e_{\beta_l, \beta_{l+1}, \ldots, \beta_k} \in \left(X_{B_{l,2\alpha_l}}\right)^2,
\end{gather*}
hence, by monotonicity in the parameter set, in particular
\begin{gather*}
e_{\beta_l, \beta_{l+1}, \ldots, \beta_k} \in 
	\bigcap_{i\in \set{l-1}}  \left( X_{\M_{l-1} \cup S_{\alpha_k}
	\cup\bigcup_{\substack{j\in\set{l-1}\\j\ne i}}\bar b_j
	\cup\bar b_{l,2\alpha_l}
	\cup\bigcup_{j=l+1}^{k-1}\bar b_{j,\alpha_j}} \right)^2.
\end{gather*}

\noindent In either case, we have (as $1 \leq m = m_k \leq m_{l+1}$ and $ 2 m_{l+1} = m_l$ by \eqref{eq: con comp -1})
\begin{gather*}
		e_{\beta_l, \dots, \beta_{k}} \in \bigcap_{i\in \set{l-1}} \left( X_{\M_{l-1} \cup S_{\alpha_k}
		\cup\bigcup_{\substack{j\in\set{l-1}\\j\ne i}}\bar b_j
		\cup\bar c_{l,\alpha_l}
		\cup\bigcup_{j=l+1}^{k-1}\bar b_{j,\alpha_j}} \right)^{m_l},
\end{gather*}
so ($\dagger_1^{l}$) holds.

Fix any $(\alpha_l,\dots, \alpha_k) <^{\lex} (\beta_l, \dots, \beta_k) \in \lambda_{l}$. Then both 
\begin{gather*}
(2 \alpha_l, \alpha_{l+1}, \dots, \alpha_k) <^{\lex} (2 \beta_l, \beta_{l+1}, \dots, \beta_k) \textrm{ and } \\
(2 \alpha_l, \alpha_{l+1}, \dots, \alpha_k) <^{\lex} (2 \beta_l + 1, \beta_{l+1}, \dots, \beta_k)
\end{gather*}
hold. Hence, using \eqref{eq: con comp 8} and Remark \ref{rem: basic props of Xs}(2), we have
\begin{gather*}
e_{\beta_l, \beta_{l+1}, \ldots, \beta_k} = d^{-1}_{2 \beta_l, \beta_{l+1}, \ldots, \beta_k} \cdot d_{2 \beta_l + 1, \beta_{l+1}, \ldots, \beta_k}  \in \\
 \left( X_{S_{\alpha_k} \cup \bar b_{1}\cup  \dots\cup \bar b_{l-1}\cup \bar b_{l, 2 \alpha_l}\cup \bar{b}_{l+1, \alpha_{l+1}} \cup	 \dots\cup \bar b_{k-1, \alpha_{k-1}}} \right)^{2 m_{l+1}},
\end{gather*}
\noindent demonstrating that ($\dagger_2^{l}$) holds (as $ 2 m_{l+1} = m_l$ by \eqref{eq: con comp -1}).

Finally, assume towards a contradiction that ($\dagger_3^{l}$) does not
hold. Choose $(\alpha_l,\dots,\alpha_k)\in\lam{l}$ witnessing this failure, and write
\[
c:=(S_{\alpha_k}\upharpoonright n,\bar b_1,\ldots,\bar b_{l-1},
\bar b_{l,2\alpha_l},\bar b_{l+1,\alpha_{l+1}},\ldots,
\bar b_{k-1,\alpha_{k-1}})
\]
and
\[
e:=d^{-1}_{2\alpha_l,\alpha_{l+1},\ldots,\alpha_k} \cdot 
d_{2\alpha_l+1,\alpha_{l+1},\ldots,\alpha_k},
\]
we have
\[
\models\eta_{q_l,\Delta_l}^{\theta}(e;c).
\]

\noindent As $(2 \alpha_l + 1, \alpha_{l+1},\dots, \alpha_k) >^{\lex} (2 \alpha_l, \alpha_{l+1},\dots, \alpha_k)$, by \eqref{eq: con comp 8} we have
\begin{gather*}
	 d_{2 \alpha_l + 1, \alpha_{l+1}, \ldots, \alpha_k}  \in \left( X_{S_{\alpha_k} \cup \bar b_{1}\cup  \dots\cup \bar b_{l-1}\cup \bar b_{l, 2 \alpha_l}\cup \bar{b}_{l+1, \alpha_{l+1}}	\cup  \dots\cup \bar b_{k-1, \alpha_{k-1}}} \right)^{m_{l+1}}. 
\end{gather*}
\noindent Since $\theta$ is uniformly thick, combining this with Lemma \ref{lem: normal hull X}(1) and (3) (and Remark \ref{rem: basic props of Xs}), we get 
\[
d_{2\alpha_l+1,\alpha_{l+1},\ldots,\alpha_k}
\in(Y_c^\theta)^{m_{l+1}}.
\]
Moreover,
\[
d_{2\alpha_l,\alpha_{l+1},\ldots,\alpha_k}
=d_{2\alpha_l+1,\alpha_{l+1},\ldots,\alpha_k} \cdot e^{-1}.
\]
Therefore \eqref{eq: finite fragment reverse induction} yields
\[
\models\eta_{q_{l+1},\Delta_{l+1}}^\theta
(d_{2\alpha_l,\alpha_{l+1},\ldots,\alpha_k};c),
\]
contradicting \eqref{eq: con comp 9}. So ($\dagger_3^{l}$) holds.
\qed$_{\text{Claim \ref{claim Ginftymain}}}$

We can additionally arrange that the combined $k$-dimensional array is indiscernible (and indexed by $\mathbb{Q}^k$ instead of $\omega^k$):
\begin{claim}\label{cla ind}
There exist sequences $\left(\bar{b}_{l,\alpha}: \alpha \in \mathbb{Q} \right)$ for $l \in \set{k-1}$ and  $(S_{\alpha_k} : \alpha_k \in \mathbb{Q})$,  and elements $\left(d_{\alpha_1, \dots, \alpha_{k}} : 
(\alpha_1,   \dots, \alpha_{k})\in \mathbb{Q}^k \right)$ in $G(\mathbb{M})$ such that:
\begin{enumerate}
\item  for all  $(\alpha_1,  \dots, \alpha_k) <^{\lex} (\beta_1, \dots, \beta_k) \in \mathbb{Q}^k$ we have 
$$d_{\beta_1, \dots, \beta_{k}}  \in \left(X_{S_{\alpha_k} \cup \bar b_{1,\alpha_1}\cup \dots\cup \bar b_{k-1, \alpha_{k-1}}} \right)^{m'};  $$
\item for all $(\alpha_1, \dots, \alpha_k) \in \mathbb{Q}^k$ we have
\[
\models\neg\eta_{q_1,\Delta_1}^{\theta}
(d_{\alpha_1,\ldots,\alpha_k};
S_{\alpha_k}\upharpoonright n,\bar b_{1,\alpha_1},\ldots,
\bar b_{k-1,\alpha_{k-1}}),
\]
and hence, by \eqref{eq: finite normal approximation}, 
\[
d_{\alpha_1, \dots, \alpha_k} \not\in
\left(Y_{S_{\alpha_k} \upharpoonright n \cup \bar b_{1, \alpha_1}
\cup  \dots \cup\bar b_{k-1, \alpha_{k-1} }}^{\theta} \right)^{q_1};
\]
\item Taking
\[
\bar{a}_{\alpha_1, \ldots, \alpha_k}:=
(S_{\alpha_k},\bar b_{1,\alpha_1},\ldots,
\bar b_{k-1,\alpha_{k-1}},d_{\alpha_1,\ldots,\alpha_k}),
\]
the $k$-dimensional array $(\bar{a}_{\bar \alpha} : \bar\alpha\in\mathbb{Q}^k)$ is indiscernible in the sense of Definition \ref{def: indiscernible array}. In particular, for every $i\in\set{k}$ the sequence $( H^i_q : q\in\mathbb Q)$ of $i$-th coordinate hyperplanes
\[
 H^i_q:=
 \bigl(\bar a_{\gamma_1,\ldots,\gamma_{i-1},q,
                    \gamma_i,\ldots,\gamma_{k-1}}
 \bigr)_{(\gamma_1,\ldots,\gamma_{k-1})\in\mathbb Q^{k-1}}
\]
 is indiscernible (with the same fixed enumeration of $\mathbb Q^{k-1}$ used for all $q$).
\end{enumerate}
\end{claim}

\proof[Proof of Claim \ref{cla ind}] 

Start with the array formed by the sequences given by Claim \ref{claim Ginftymain}, it satisfies (1) and (2) (with $\omega^k$ instead of $\mathbb{Q}^k$). Note that (1) and (2) still hold restricting to any subarray. The existence of an array satisfying (1),(2),(3) is equivalent to the consistency of a small partial type (using Remark \ref{rem: basic props of Xs}(6) to express (1)), and by Lemma \ref{lem: indisc subarray}  every finite subtype is consistent (witnessed by some finite subarray of the original array), so we can conclude by compactness. 
\qed$_{\text{Claim \ref{cla ind}}}$
	

For \(\bar\alpha=(\alpha_1,\ldots,\alpha_k)\in\omega^k\), put
\[
 \bar c_{2\bar\alpha}:=
 (S_{2\alpha_k}\upharpoonright n,\bar b_{1,2\alpha_1},
 \ldots,\bar b_{k-1,2\alpha_{k-1}}),
 \qquad
 X_{\bar\alpha}:=X_{\bar c_{2\bar\alpha}},
 \qquad
 Y_{\bar\alpha}:=Y_{\bar c_{2\bar\alpha}}.
\]

 Given \(0\leq t\leq k\), a finite set
\(J\subseteq\omega^t\), and a family
\[
 \boldsymbol{x}=(x_{\bar\beta,\bar\epsilon}:
 \bar\beta\in J,\ \bar\epsilon\in\{0,1\}^t)
\]
of elements of \(G\), define an element \(W_t(J;\boldsymbol{x})\in G\)
by recursion on $t$.
For \(t=0\), define
\[
W_0(J;\boldsymbol{x}):=
\begin{cases}
1_G,&J=\emptyset,\\
x_{(),()},&J=\{()\}.
\end{cases}
\]
(These are the only possibilities since \(J\subseteq\omega^0=\{()\}\).)
For \(t>0\), write each \(\bar\beta\in J\) uniquely as
\((r,\bar\gamma)\), where \(r\in\omega\) and
\(\bar\gamma\in\omega^{t-1}\), and put
\[
 R(J):=\operatorname{pr}_1(J)
 =\{r\in\omega:(r,\bar\gamma)\in J
      \text{ for some }\bar\gamma\in\omega^{t-1}\}.
\]
For each \(r\in R(J)\) and \(e\in\{0,1\}\), let
\[
 J_r:=\{\bar\gamma\in\omega^{t-1}:(r,\bar\gamma)\in J\}
\]
and define the reindexed family
\[
 \boldsymbol{x}^{r,e}
 :=\bigl(x_{(r,\bar\gamma),(e,\bar\delta)}:
 \bar\gamma\in J_r,\ 
 \bar\delta\in\{0,1\}^{t-1}\bigr).
\]
Thus \(\boldsymbol{x}^{r,e}\) is obtained from \(\boldsymbol{x}\) by
fixing the first coordinate of \(\bar\beta\) to be \(r\), fixing the
first bit of \(\bar\epsilon\) to be \(e\), and then suppressing those
two fixed coordinates. Taking the products in increasing order of
\(r\), put
\[
 P_e(J;\boldsymbol{x})
 :=\prod_{r\in R(J)}^{\nearrow}
 W_{t-1}(J_r;\boldsymbol{x}^{r,e}),
 \qquad
 W_t(J;\boldsymbol{x}):=P_0(J;\boldsymbol{x})^{-1}
 P_1(J;\boldsymbol{x}).
 \tag{\(\clubsuit\)}
\]

 Define  an ordering \(\prec_t\) of
\(\{0,1\}^t\) and signs
\(\sigma_t:\{0,1\}^t\to\{1,-1\}\) as follows.  For \(t=0\), take the
unique ordering and put \(\sigma_0(())=1\).  Given
\((\prec_{t-1},\sigma_{t-1})\), first list the tuples
\((0,\bar\delta)\) with \(\bar\delta\) in reverse \(\prec_{t-1}\)-order,
then the tuples \((1,\bar\delta)\) with \(\bar\delta\) in
\(\prec_{t-1}\)-order, and put
\[
 \sigma_t(0,\bar\delta):=-\sigma_{t-1}(\bar\delta),
 \qquad
 \sigma_t(1,\bar\delta):=\sigma_{t-1}(\bar\delta).
\]
Equivalently, for
\(\bar\epsilon=(\epsilon_1,\ldots,\epsilon_t)\in\{0,1\}^t\),
\[
 \sigma_t(\bar\epsilon)
 =(-1)^{t-(\epsilon_1+\cdots+\epsilon_t)}.
\]

\begin{claim}\label{C: recursive deletion}
Let \(0\leq t\leq k\), let \(J\subseteq\omega^t\) be finite, and let $
 \boldsymbol{x}=(x_{\bar\beta,\bar\epsilon}:
 \bar\beta\in J,\ \bar\epsilon\in\{0,1\}^t)$ 
be a family of elements of \(G\).  Fix \(\bar\alpha\in J\), put
\(J':=J\setminus\{\bar\alpha\}\), and define
\[
 \boldsymbol{x}' := (x_{\bar\beta,\bar\epsilon}:
 \bar\beta\in J',\ \bar\epsilon\in\{0,1\}^t).
\]
 Then there are elements
\(h_{\bar\epsilon}\in G\) for 
\(\bar\epsilon\in\{0,1\}^t\) such that
\[
 W_t(J;\boldsymbol{x})=
 \left(\prod_{\bar\epsilon\in\{0,1\}^t}^{\prec_t}
 h_{\bar\epsilon}
 x_{\bar\alpha,\bar\epsilon}^{\sigma_t(\bar\epsilon)}
 h_{\bar\epsilon}^{-1}\right)
 W_t(J';\boldsymbol{x}').
 \tag{\(\mathrm{D}_t\)}
 \label{eq: recursive deletion}
\]
\end{claim}

\proof[Proof of Claim \ref{C: recursive deletion}]
We argue by induction on \(t\).  If \(t=0\), then
\(J=\{()\}\), \(\bar\alpha=()\), and \(J'=\emptyset\).  Hence
\(W_0(J;\boldsymbol{x})=x_{(),()}\) and
\(W_0(J';\boldsymbol{x}')=1_G\), so \((\mathrm{D}_0)\) holds with
\(h_{()}=1_G\).

Suppose \(t>0\), and write
\(\bar\alpha=(a,\bar\alpha')\), where
\(a\in\omega\) and \(\bar\alpha'\in\omega^{t-1}\).  Then
\(a\in R(J)\) and \(\bar\alpha'\in J_a\). For \(r\in\omega\), write
\[
 J'_r:=\{\bar\gamma\in\omega^{t-1}:(r,\bar\gamma)\in J'\}.
\]
In particular
\[
 J'_a =J_a\setminus\{\bar\alpha'\}\subseteq\omega^{t-1}.
\]
For each \(e\in\{0,1\}\), define
\[
 (\boldsymbol{x}')^{a,e}:=
 \bigl(x_{(a,\bar\gamma),(e,\bar\delta)}:
 \bar\gamma\in J'_a,\
 \bar\delta\in\{0,1\}^{t-1}\bigr).
\]
Thus \((\boldsymbol{x}')^{a,e}\) is the restriction of
\(\boldsymbol{x}^{a,e}\) to the first-index set \(J'_a\); it is the
empty family when \(J'_a=\emptyset\).  Applying the induction hypothesis
to \((J_a,\boldsymbol{x}^{a,e})\) gives elements
\(g_{e,\bar\delta}\in G\), indexed by
\(\bar\delta\in\{0,1\}^{t-1}\), such that
\[
 W_{t-1}(J_a;\boldsymbol{x}^{a,e})
 =Q_eW_{t-1}(J'_a;(\boldsymbol{x}')^{a,e}),
\]
where
\[
 Q_e:=
 \prod_{\bar\delta\in\{0,1\}^{t-1}}^{\prec_{t-1}}
 g_{e,\bar\delta}
 x_{\bar\alpha,(e,\bar\delta)}^{\sigma_{t-1}(\bar\delta)}
 g_{e,\bar\delta}^{-1}.
\]
For each \(e\in\{0,1\}\), let
\[
 L_e:=\prod_{\substack{r\in R(J)\\r<a}}^{\nearrow}
 W_{t-1}(J_r;\boldsymbol{x}^{r,e}),
 \qquad
 S_e:=\prod_{\substack{r\in R(J)\\r>a}}^{\nearrow}
 W_{t-1}(J_r;\boldsymbol{x}^{r,e}),
\]
and put \(P'_e:=P_e(J';\boldsymbol{x}')\).  Then for $r\in R(J),\ r\neq a$ we have 
\[
 J'_r=J_r
 \quad\text{and}\quad
 (\boldsymbol{x}')^{r,e}=\boldsymbol{x}^{r,e}.
\]
Consequently the definitions of the ordered products give
\[
\begin{aligned}
 P_e(J;\boldsymbol{x})
 &=L_eW_{t-1}(J_a;\boldsymbol{x}^{a,e})S_e,\\
 P'_e
 &=L_eW_{t-1}(J'_a;(\boldsymbol{x}')^{a,e})S_e.
\end{aligned}
\]
When \(J'_a=\emptyset\), the second equality remains valid because
\(a\notin R(J')\) and
\(W_{t-1}(\emptyset;(\boldsymbol{x}')^{a,e})=1_G\).  Substituting the
induction-hypothesis equality for the \(a\)-fiber in the first line now
gives, for \(e=0,1\),
\[
\begin{aligned}
 P_e(J;\boldsymbol{x})
 &=L_eQ_eW_{t-1}(J'_a;(\boldsymbol{x}')^{a,e})S_e\\
 &=(L_eQ_eL_e^{-1})P'_e.
\end{aligned}
\]

Substituting these two identities into \((\clubsuit)\), we obtain
\[
\begin{split}
 W_t(J;\boldsymbol{x})
 &=(P_0'^{-1}L_0Q_0^{-1}L_0^{-1}P_0')
   (P_0'^{-1}L_1Q_1L_1^{-1}P_0')
   W_t(J';\boldsymbol{x}').
\end{split}
\]
For \(e\in\{0,1\}\) and
\(\bar\delta\in\{0,1\}^{t-1}\), set
\[
 h_{(e,\bar\delta)}:=P_0'^{-1}L_eg_{e,\bar\delta}\in G.
\]
Let \(N:=2^{t-1}\), and enumerate \(\{0,1\}^{t-1}\) in
\(\prec_{t-1}\)-order as $\bar\delta_1\prec_{t-1}\cdots\prec_{t-1}\bar\delta_N$. 
Put \(s_i:=\sigma_{t-1}(\bar\delta_i)\).  Then
\[
\begin{aligned}
 Q_e =\prod_{i=1}^{N}
 g_{e,\bar\delta_i}
 x_{\bar\alpha,(e,\bar\delta_i)}^{s_i}
 g_{e,\bar\delta_i}^{-1}, \quad 
 Q_0^{-1} =\prod_{i=1}^{N}
 g_{0,\bar\delta_{N+1-i}}
 x_{\bar\alpha,(0,\bar\delta_{N+1-i})}^{-s_{N+1-i}}
 g_{0,\bar\delta_{N+1-i}}^{-1}.
\end{aligned}
\]
Using $(P_0'^{-1}L_eg_{e,\bar\delta})^{-1}
 =g_{e,\bar\delta}^{-1}L_e^{-1}P_0'$ 
and the definition of \(h_{(e,\bar\delta)}\), we now calculate
\begin{gather*}
	(P_0'^{-1}L_0Q_0^{-1}L_0^{-1}P_0')
   (P_0'^{-1}L_1Q_1L_1^{-1}P_0')\\
   =
 \left(\prod_{i=1}^{N}
 h_{(0,\bar\delta_{N+1-i})}
 x_{\bar\alpha,(0,\bar\delta_{N+1-i})}^{-s_{N+1-i}}
 h_{(0,\bar\delta_{N+1-i})}^{-1}\right) \cdot
 \left(\prod_{i=1}^{N}
 h_{(1,\bar\delta_i)}
 x_{\bar\alpha,(1,\bar\delta_i)}^{s_i}
 h_{(1,\bar\delta_i)}^{-1}\right)\\
 =\prod_{\bar\epsilon\in\{0,1\}^{t}}^{\prec_t}
 h_{\bar\epsilon}
 x_{\bar\alpha,\bar\epsilon}^{\sigma_t(\bar\epsilon)}
 h_{\bar\epsilon}^{-1}.
\end{gather*}
Indeed, the indices in the first two products, in order, are
\[
 (0,\bar\delta_N),\ldots,(0,\bar\delta_1),
 (1,\bar\delta_1),\ldots,(1,\bar\delta_N),
\]
which is precisely their \(\prec_t\)-order, and
\[
 \sigma_t(0,\bar\delta_i)=-s_i,
 \qquad
 \sigma_t(1,\bar\delta_i)=s_i.
\]
Substituting this into the previous formula for $ W_t(J;\boldsymbol{x})$ gives \((\mathrm{D}_t)\).
\qed$_{\text{Claim \ref{C: recursive deletion}}}$

For a finite \(I\subseteq\omega^k\), we use the particular family $
 \boldsymbol{x}=(x_{\bar\beta,\bar\epsilon}:
 \bar\beta\in I,\ \bar\epsilon\in\{0,1\}^k)
$ with 
\[
 x_{\bar\beta,\bar\epsilon}:=
 d_{2\bar\beta+\bar\epsilon}
 =d_{2\beta_1+\epsilon_1,\ldots,2\beta_k+\epsilon_k}
\]
and write \(w_I:=W_k(I;\boldsymbol{x})\).

\begin{claim}\label{C: recursive word outside}
Let \(I\subseteq\omega^k\) be finite, \(\bar\alpha\notin I\) and $p=2^k-1$. Then
\[
 w_I\in Y_{\bar\alpha}^{p}.
\]
\end{claim}

\proof[Proof of Claim \ref{C: recursive word outside}]
Fix \(\bar\alpha=(\alpha_1,\ldots,\alpha_k)\).  For
\(0\leq t\leq k\), put
\[
 \bar\alpha^{[t]}:=(\alpha_{k-t+1},\ldots,\alpha_k)\in\omega^t,
\]
where \(\bar\alpha^{[0]}:=()\).  Given
\(\bar\delta=(\delta_1,\ldots,\delta_{k-t})
\in\{0,1\}^{k-t}\) and a finite \(J\subseteq\omega^t\), define
\[
 \boldsymbol{x}^{t,\bar\delta}(J)
 :=\bigl(d_{2\alpha_1+\delta_1,\ldots,
       2\alpha_{k-t}+\delta_{k-t},
       2\beta_1+\epsilon_1,\ldots,2\beta_t+\epsilon_t}:
       \bar\beta\in J,\ \bar\epsilon\in\{0,1\}^t\bigr),
\]
where
\(\bar\beta=(\beta_1,\ldots,\beta_t)\),
\(\bar\epsilon=(\epsilon_1,\ldots,\epsilon_t)\).  Put
\[
 u_{t,\bar\delta}(J)
 :=W_t\bigl(J;\boldsymbol{x}^{t,\bar\delta}(J)\bigr),
 \qquad p_t:=2^t-1.
\]
We prove by induction on \(t\) that, for every
\(\bar\delta\in\{0,1\}^{k-t}\) and every finite
\(J\subseteq\omega^t\),
\[
 \bar\alpha^{[t]}\notin J
 \quad\Longrightarrow\quad
 u_{t,\bar\delta}(J)\in Y_{\bar\alpha}^{p_t}.
 \tag{\(\mathcal A_t\)}
\]
For \(t=k\), the tuple \(\bar\delta\) is empty,
\(u_{k,()}(I)=w_I\), and \((\mathcal A_k)\) is the claim.

If \(t=0\), the hypothesis forces \(J=\emptyset\) (as $\bar{\alpha}^{[0]} = () \notin J$), so
\(u_{0,\bar\delta}(J)=1_G\in Y_{\bar\alpha}^{p_0}\).

Suppose \(t>0\) and that \((\mathcal A_{t-1})\) holds.  Set
\[
 i:=k-t+1,\qquad
 a:=\alpha_i,
\]
so \(\bar\alpha^{[t]}=(a,\bar\alpha^{[t-1]})\).  Fix
\(\bar\delta\in\{0,1\}^{i-1}\)
and a finite \(J\subseteq\omega^t\) such that
\((a,\bar\alpha^{[t-1]})\notin J\).  For every \(r\in\omega\), let
\[
 J_r:=\{\bar\gamma\in\omega^{t-1}:(r,\bar\gamma)\in J\}.
\]

For \(r\in\omega\) and \(e\in\{0,1\}\), put
\[
 B_{r,e}:=W_{t-1}\left(J_r;
 \bigl(d_{2\alpha_1+\delta_1,\ldots,
  2\alpha_{i-1}+\delta_{i-1},2r+e,
  2\gamma_1+\eta_1,\ldots,2\gamma_{t-1}+\eta_{t-1}}:
  \bar\gamma\in J_r,\ \bar\eta\in\{0,1\}^{t-1}\bigr)\right),
\]
with empty coordinate blocks omitted.  Thus \(B_{r,e}=1_G\) if
\(J_r=\emptyset\).  Put \(R:=R(J)\).  For \(e\in\{0,1\}\), define
\[
 T_e:=\prod_{r\in R\setminus\{a\}}^{\nearrow}B_{r,e},  \quad 
 A_e:=
 \left(\prod_{\substack{r\in R\\r<a}}^{\nearrow}B_{r,e}\right)
 B_{a,e}
 \left(\prod_{\substack{r\in R\\r<a}}^{\nearrow}B_{r,e}\right)^{-1}.
\]
We verify a required factorization $(\mathrm{F}_e)$ below.  Suppose first that
\(a\in R\).  In the following calculation all \(r\)-products are over
\(r\in R\) and are taken in increasing order:
\[
\begin{aligned}
 A_eT_e
 &=\left(\prod_{r<a}^{\nearrow}B_{r,e}\right)
   B_{a,e}
   \left(\prod_{r<a}^{\nearrow}B_{r,e}\right)^{-1}
   \left(\prod_{r<a}^{\nearrow}B_{r,e}\right)
   \left(\prod_{r>a}^{\nearrow}B_{r,e}\right)\\
 &=\left(\prod_{r<a}^{\nearrow}B_{r,e}\right)
   B_{a,e}
   \left(\prod_{r>a}^{\nearrow}B_{r,e}\right) =\prod_{r\in R}^{\nearrow}B_{r,e}
 =P_e\bigl(J;\boldsymbol{x}^{t,\bar\delta}(J)\bigr).
\end{aligned}
\]
If \(a\notin R\), then \(J_a=\emptyset\), so \(B_{a,e}=1_G\) and
hence \(A_e=1_G\), while
\(T_e=\prod_{r\in R}^{\nearrow}B_{r,e}
=P_e(J;\boldsymbol{x}^{t,\bar\delta}(J))\).  Thus, in both cases,
\[
 P_e\bigl(J;\boldsymbol{x}^{t,\bar\delta}(J)\bigr)=A_eT_e.
 \tag{\(\mathrm{F}_e\)}
\]

We next prove
\[
 T_0^{-1}T_1\in X_{\bar\alpha}\subseteq Y_{\bar\alpha}.
 \tag{\(\mathrm{T}_t\)}
\]
Enumerate
\[
 R\setminus\{a\}=\{r_1<\cdots<r_s\}.
\]
If \(s=0\), then \(T_0=T_1=1_G\), so \((\mathrm{T}_t)\) follows
from \(1_G\in X_{\bar\alpha}\); assume \(s>0\).  

\noindent Recall that the 
sequence \((H_q^i:q\in\mathbb Q)\) from Claim \ref{cla ind}(3) is indiscernible. Since \(2a\) and the \(2r_v\)'s are distinct even integers, for each \(1\leq u\leq s\),
\[
 \begin{aligned}
 &\qftp_{<}\!\left(
  2r_u/\{2a\}\cup\{2r_v:1\leq v\leq s,\ v\neq u\}\right)\\
 &\qquad=\qftp_{<}\!\left(
  2r_u+1/\{2a\}\cup\{2r_v:1\leq v\leq s,\ v\neq u\}\right).
 \end{aligned}
\]
Applying Lemma \ref{lem: lasc dist 1 on an array} (with the lemma's \(S :=\emptyset\), \(I := \{2r_1,\ldots,2r_s\}\), \(J := \{2a\}\) and
\(i'_u :=2r_u+1\))  we get 
\[
 \dist_{H^i_{2a}}
 \bigl((H^i_{2r_u})_{u=1}^s,
       (H^i_{2r_u+1})_{u=1}^s\bigr)\leq1.
\]

\noindent Note that we have \[
 F(H^i_{2r_1+e},\ldots,H^i_{2r_s+e})=T_e
 \qquad(e=0,1)
\]
for an $\emptyset$-definable function $F$. Namely, for
\(\bar\gamma\in\omega^{t-1}\) and
\(\bar\eta\in\{0,1\}^{t-1}\), let
\(\pi_{\bar\gamma,\bar\eta}\) select the \(d\)-component of the hyperplane tuple in the position 
\[
 (2\alpha_1+\delta_1,\ldots,
  2\alpha_{i-1}+\delta_{i-1},
  2\gamma_1+\eta_1,\ldots,
  2\gamma_{t-1}+\eta_{t-1})\in\mathbb Q^{k-1}.
\]
Thus, with empty coordinate blocks omitted,
\[
 \pi_{\bar\gamma,\bar\eta}(H_q^i)
 =d_{2\alpha_1+\delta_1,\ldots,
      2\alpha_{i-1}+\delta_{i-1},q,
      2\gamma_1+\eta_1,\ldots,
      2\gamma_{t-1}+\eta_{t-1}}.
\]
For infinite tuples \(Z_1,\ldots,Z_s\) of the same enumeration and sorts as
\(H_q^i\), we can then take the \(\emptyset\)-definable function 
\[
 F(Z_1,\ldots,Z_s)
 :=\prod_{u=1}^{s}
 W_{t-1}\left(J_{r_u};
  (\pi_{\bar\gamma,\bar\eta}(Z_u):
   \bar\gamma\in J_{r_u},\
   \bar\eta\in\{0,1\}^{t-1})\right).
\]
Hence, by Remark
\ref{rem : dcl of lascar distance}, 
\[
 \dist_{H^i_{2a}}(T_0,T_1)\leq1.
\]
Also from the definitions, the tuple
\(\bar c_{2\bar\alpha}\)  is a subtuple of $\bar a_{2\bar\alpha}$, which is an element of  the tuple $H^i_{2a}$ (recall that $a = \alpha_i$, hence $2a = 2 \alpha_i$). Monotonicity in the base therefore gives
\[
 \dist_{\bar c_{2\bar\alpha}}(T_0,T_1)\leq1.
\]
\noindent By the definition of
\(X_{\bar c_{2\bar\alpha}}\), this implies
\(T_0^{-1}T_1\in X_{\bar c_{2\bar\alpha}}=X_{\bar\alpha}\), proving
\((\mathrm{T}_t)\).

  We have
\(\bar\alpha^{[t-1]}\notin J_a\).  Since \(a=\alpha_i\), the family
in the definition of \(B_{a,e}\) is precisely
\(\boldsymbol{x}^{t-1,(\bar\delta,e)}(J_a)\), where
\((\bar\delta,e)\) denotes concatenation.  Hence
\[
 B_{a,e}=u_{t-1,(\bar\delta,e)}(J_a)
 \in Y_{\bar\alpha}^{p_{t-1}}
 \qquad(e=0,1)
\]
by the inductive assumption \((\mathcal A_{t-1})\).

Using  \((\mathrm{F}_e)\) we have 
\[
\begin{aligned}
 u_{t,\bar\delta}(J)
 &= W_t\bigl(J;\boldsymbol{x}^{t,\bar\delta}(J)\bigr) = T_0^{-1}A_0^{-1}A_1T_1\\
 &=\bigl(T_0^{-1}A_0^{-1}T_0\bigr)
   \bigl(T_0^{-1}A_1T_0\bigr)
   \bigl(T_0^{-1}T_1\bigr).
\end{aligned}
\]
By Lemma \ref{lem: normal hull X}(1), \(Y_{\bar\alpha}\) is symmetric
and invariant under conjugation, hence $ Y_{\bar\alpha}^{p_{t-1}}$ is also symmetric and invariant under conjugation.  Since \(A_e\) is a conjugate of
\(B_{a,e}\) by definition, the first two factors therefore belong to
\(Y_{\bar\alpha}^{p_{t-1}}\), while the last belongs to
\(Y_{\bar\alpha}\) by \((\mathrm{T}_t)\).
Therefore
\[
 u_{t,\bar\delta}(J)
 \in Y_{\bar\alpha}^{2p_{t-1}+1}
 =Y_{\bar\alpha}^{2^t-1}
 =Y_{\bar\alpha}^{p_t}.
\]
This proves \((\mathcal A_t)\), and hence the claim.
\qed$_{\text{Claim \ref{C: recursive word outside}}}$

Put \(\bar0:=(0,\ldots,0)\in\{0,1\}^k\).  If
\(\bar\alpha\in\omega^k\) and
\(\bar\epsilon\in\{0,1\}^k\setminus\{\bar0\}\), then
\(2\bar\alpha<^{\lex}2\bar\alpha+\bar\epsilon\).  Hence Claim
\ref{cla ind}(1) and Lemma \ref{lem: normal hull X} give 
\begin{gather}
	 d_{2\bar\alpha+\bar\epsilon}
 \in X_{\bar\alpha}^{m'}
 \subseteq (Y_{\bar\alpha}^{\theta})^{m'}.\label{eq: d2alpha in Y}
\end{gather}

We now extract a formula witnessing the failure of $k$-dependence.  Recall that
\(p=2^k-1\) and \(q_1=(m'+2)p\).   
\begin{claim}\label{cla: final formula extraction}
	There is a finite set of formulas \(\Delta_*\subseteq\pi_G\) so that the following holds for all $c \in \mathbb{M}^z$ and 
\(w,w',z_{\bar\epsilon},h_{\bar\epsilon}\in G\), $\bar\epsilon\in\{0,1\}^k$:
if 
\begin{gather}
	w=\left(\prod_{\bar\epsilon\in\{0,1\}^k}^{\prec_k}
 h_{\bar\epsilon}z_{\bar\epsilon}^{\sigma_k(\bar\epsilon)}
 h_{\bar\epsilon}^{-1}\right)w'\label{eq: group deletion equation}
\end{gather}
and $w'\in(Y_c^\theta)^p$, $z_{\bar\epsilon}\in(Y_c^\theta)^{m'}$  for all  $\bar\epsilon\neq\bar0$ and $\models\eta_{p,\Delta_*}^\theta(w;c)$, then also  $\models\eta_{q_1,\Delta_1}^\theta(z_{\bar0};c)$.
\end{claim}
\proof[Proof of Claim \ref{cla: final formula extraction}]
Suppose not.  For every finite \(\Delta\subseteq\pi_G\), choose a
counterexample with \(\Delta\) in place of \(\Delta_*\).  Consider the
partial type saying that \(w,w'\), and all
\(z_{\bar\epsilon},h_{\bar\epsilon}\)
\((\bar\epsilon\in\{0,1\}^k)\), realize \(\pi_G\), that
\eqref{eq: group deletion equation} holds, and that
for every finite \(\Delta\subseteq\pi_G\),
\[
 \eta_{p,\Delta}^{\theta}(w;c),\qquad
 \eta_{p,\Delta}^{\theta}(w';c),\qquad
 \eta_{m',\Delta}^{\theta}(z_{\bar\epsilon};c)
 \quad(\bar\epsilon\neq\bar0)
\]
hold, while
\(\neg\eta_{q_1,\Delta_1}^{\theta}(z_{\bar0};c)\) holds.  Every finite
part of this type is realized by the counterexample corresponding to
the union of the finitely many sets of formulas appearing in it.  The required
formulas for \(w\) hold because each appearing fragment is contained in
that union.  The formulas for \(w'\) and the nonzero
\(z_{\bar\epsilon}\)'s follow from the exact memberships satisfied by
the counterexample and \eqref{eq: finite normal approximation}.
Compactness and \eqref{eq: finite normal approximation} therefore give
\[
 w,w'\in(Y_c^\theta)^p,\qquad
 z_{\bar\epsilon}\in(Y_c^\theta)^{m'}
 \quad(\bar\epsilon\neq\bar0).
\]
By Lemma
\ref{lem: normal hull X}(1), \(Y_c^\theta\) is symmetric and invariant
under conjugation (hence also \((Y_c^\theta)^{N}\) for any $N \in \omega$).  Solving the group equation \eqref{eq: group deletion equation} for the factor containing
\(z_{\bar0}\) gives
\begin{gather*}
	 h_{\bar0}z_{\bar0}^{\sigma_k(\bar0)}h_{\bar0}^{-1}
 = 
 \left(
  \prod_{\substack{\bar\epsilon\in\{0,1\}^k\\
                    \bar\epsilon\prec_k\bar0}}^{\prec_k}
  h_{\bar\epsilon}z_{\bar\epsilon}^{\sigma_k(\bar\epsilon)}
  h_{\bar\epsilon}^{-1}
 \right)^{-1}
 w(w')^{-1}
 \left(
  \prod_{\substack{\bar\epsilon\in\{0,1\}^k\\
                    \bar0\prec_k\bar\epsilon}}^{\prec_k}
  h_{\bar\epsilon}z_{\bar\epsilon}^{\sigma_k(\bar\epsilon)}
  h_{\bar\epsilon}^{-1}
 \right)^{-1}.
\end{gather*}

\noindent The two outside products contain altogether the other \(2^k-1\)
conjugate factors.  Each such factor and its inverse belongs to
\((Y_c^\theta)^{m'}\), while symmetry gives
\(w(w')^{-1}\in(Y_c^\theta)^{2p}\).  Therefore
\[
 h_{\bar0}z_{\bar0}^{\sigma_k(\bar0)}h_{\bar0}^{-1}
 \in (Y_c^\theta)^{\,2p+(2^k-1)m'}
 =(Y_c^\theta)^{q_1}.
\]
Symmetry and conjugation invariance of \((Y_c^\theta)^{q_1}\) therefore give
\(z_{\bar0}\in(Y_c^\theta)^{q_1}\), which implies
\(\eta_{q_1,\Delta_1}^{\theta}(z_{\bar0};c)\), a contradiction. 
\qed$_{\text{Claim \ref{cla: final formula extraction}}}$

Finally, define the formula
\[
 \psi(x;z):=\eta_{p,\Delta_*}^\theta(x;z).
\]
If \(\bar\alpha\notin I\), then Claim
\ref{C: recursive word outside} gives
\[
 w_I\in (Y_{\bar\alpha})^p
 \subseteq (Y_{\bar c_{2\bar\alpha}}^{\theta})^p,
\]
so \eqref{eq: finite normal approximation} implies
\(\models\psi(w_I;\bar c_{2\bar\alpha})\).

Suppose instead that \(\bar\alpha\in I\) and \(\models\psi(w_I;\bar c_{2\bar\alpha})\) holds, and put
\(I':=I\setminus\{\bar\alpha\}\).  Claim
\ref{C: recursive deletion}, with \(t=k\), \(J=I\), and distinguished
element \(\bar\alpha\), gives \( \left( h_{\bar\epsilon}\in G : \bar\epsilon\in\{0,1\}^t \right)\)  
 so that the equation
\eqref{eq: group deletion equation} holds with
\[
 w=w_I,\qquad w'=w_{I'},\qquad
 z_{\bar\epsilon}=d_{2\bar\alpha+\bar\epsilon}.
\]
Indeed, its restricted family is exactly the family defining \(w_{I'}\).
By  assumption we have $\models \eta_{p,\Delta_*}^\theta(w_I;\bar c_{2\bar\alpha})$, Claim \ref{C: recursive word outside} applied to
\(I'\) gives $w_{I'} \in (Y_{\bar c_{2\bar\alpha}}^\theta)^p$, and \eqref{eq: d2alpha in Y} gives $d_{2\bar\alpha+\bar\epsilon} \in(Y_{\bar c_{2\bar\alpha}}^\theta)^{m'}$  for all  $\bar\epsilon\neq\bar0$. Hence, by  Claim \ref{cla: final formula extraction} with  \(c=\bar c_{2\bar\alpha}\), we get $\models\eta_{q_1,\Delta_1}^\theta(d_{2\bar\alpha};\bar c_{2\bar\alpha})$. But this contradicts Claim \ref{cla ind}(2) at the index \(2\bar\alpha\).
Consequently
\[
 \models\psi(w_I;S_{2\alpha_k}\upharpoonright n,\bar b_{1,2\alpha_1},
 \ldots,\bar b_{k-1,2\alpha_{k-1}})
 \quad\Longleftrightarrow\quad
 (\alpha_1, \ldots, \alpha_k) \notin I.
\]

Thus, by compactness, the sequences 
\[
\{ 
 (\bar b_{i,2\alpha}:\alpha\in\omega):1\leq i<k \} \cup \{
 (S_{2\alpha}\upharpoonright n:\alpha\in\omega) \}
\]
witness that $\psi$ is not 
$k$-dependent. \qed

\subsection{Example: $G^{\infty}$ and $G^{00}$ in multilinear forms over prime fields} \label{sec: Ginfty in multilin}
As an example, we calculate the connected component of the additive group of an infinite-dimensional vector space over a prime field with an alternating non-degenerate $n$-linear form (Definition \ref{def: non-degen}). Let
\[
\mathcal M=(V,\mathbb F_p,\langle- ,\dots,-\rangle_n)
\]
be viewed in the full two-sorted language from Definition \ref{def: language of bilin forms}, where $V$ is an infinite-dimensional vector space over $\mathbb F_p$ and
$\langle- ,\dots,-\rangle_n:V^n\to\mathbb F_p$ is alternating and non-degenerate.  Note that each field element is represented by a closed ring term since $1$ is a constant in the language of rings. The theory $\Th(\M)$ is (strictly) $n$-dependent (by Theorem \ref{thm: Granger}), $\omega$-categorical (by Remark \ref{rem: VS complete and omega-cat}(2)) and simple (Corollary \ref{cor: multilin fin field simple}).

\begin{claim}
	For any small set of parameters $A \subset \mathbb{M}$, put $A_V:=A\cap V$ and
	\[
	H_A:=\bigcap_{\pmb a\in A_V^{n-1}}V_{\pmb a},
	\qquad
	V_{\pmb a}:=\{v\in V:\langle\pmb a,v\rangle_n=0\}.
	\]
	Then $V^{\infty}_A = V^{00}_A = V^{0}_{A}=H_A$. 
\end{claim}

\begin{remark}
	This generalizes \cite[Example 4.1.14]{wagner2000simple}, although our form is alternating rather than symmetric.
\end{remark}

\begin{proof}
For every $\pmb a\in A_V^{n-1}$, the group $V_{\pmb a}$ is either
$V$ (when $\overline{\pmb a}=0$ in $\bigwedge^{n-1}V$) or an
$A$-definable subgroup of $V$ of index $p$.  Hence $V_A^{00}\subseteq V_A^0\subseteq H_A$.

We next show the reverse inclusion $H_A \subseteq V_A^{00}$.  
By quantifier elimination from Theorem \ref{thm: QE for multilinear forms} (specialized to the case of a finite field)  we have that for any $v \in V(\mathbb{M})$ and set of parameters $A$,  $\tp(v/A)$ is determined by one of the following conditions:
	\begin{enumerate}
	\item $v \in \operatorname{Span}(A_V)$, i.e.~(using that the field is finite) $v = \sum_{i=1}^m a_i$ for some $m \in \mathbb{N}$ and $a_i \in A_V$ (not necessarily distinct); or
	\item $v \not \in \operatorname{Span}(A_V)$ and $\bigwedge_{\pmb a \in A_V^{n-1}} \langle \pmb a, v \rangle_n = i_{\pmb a}$ for some $i_{\pmb a} \in \{0, \dots, p-1\}$ (using that the form is alternating,  its values with $v$ placed in the other coordinates are determined by this).
	\end{enumerate}

\noindent Since $A$ is small and $V_A^{00}$ has bounded index in $V$, $V_A^{00}$ is not
contained in the small subspace $\Span_{\mathbb F_p}(A_V)$.  Choose $v\in V_A^{00}\setminus\Span_{\mathbb F_p}(A_V)$. 
The inclusion already proved gives $v\in H_A$, so by the above observation it follows that $\tp(v/A)$ is determined by $v \not \in \operatorname{Span}(A_V)$ and $\bigwedge_{\pmb a \in A_V^{n-1}} \langle \pmb a, v \rangle_n = 0$.  Hence every
$v'\in H_A\setminus\Span_{\mathbb F_p}(A_V)$ satisfies $v'\equiv_A v$, and
therefore belongs to the $A$-type-definable group $V_A^{00}$. Finally, let $u\in H_A\cap\Span_{\mathbb F_p}(A_V)$ (if it exists).  Choose
$w\in H_A\setminus\Span_{\mathbb F_p}(A_V)$, then also $w-u \in H_A\setminus\Span_{\mathbb F_p}(A_V)$, in particular $w, w-u \in V_A^{00}$ by the above.  Thus $u=w-(w-u)\in V_A^{00}$.  We have thus proved $
H_A\subseteq V_A^{00}$, so $H_A = V_A^{00} = V_A^0$.  Since $\Th(\M)$ is simple, Remark
\ref{rem: amenable follows}(2) gives $V_A^{\infty}=V_A^{00}$.
\end{proof}

In particular, for any small sets $A_1, \ldots, A_n \subseteq \mathbb{M}$ we have
$$V^{\infty}_{A_1 \cup \ldots \cup A_n} = \bigcap_{i \in [n]} V^{\infty}_{A_1 \cup \ldots \cup A_{i-1} \cup A_{i+1} \cup \ldots \cup A_n}.$$
Thus a strong form of Theorem \ref{thm: Ginfty k-dep main} is satisfied without any additional assumptions.

\bibliographystyle{plain}
\bibliography{ref_corrected}

\end{document}